\documentclass[twoside,11pt]{article}
\usepackage{DGQ}
\usepackage[toc,page]{appendix}

\newtheorem*{thm1}{Theorem 1}
\newtheorem*{thm2}{Theorem 2}

\begin{document}

\begin{titlepage}

\title{\textbf{Derived smooth stacks and prequantum categories}}
\author{James Wallbridge\\
Kavli IPMU (WPI), UTIAS, University of Tokyo\\
5-1-5 Kashiwanoha, Kashiwa, 277-8583, Japan\\
james.wallbridge@ipmu.jp}
\date{{\today}}
\maketitle

\begin{center}
\textbf{Abstract}
\end{center}

The Weil-Kostant integrality theorem states that given a smooth manifold endowed with an integral complex closed 2-form, then there exists a line bundle with connection on this manifold with curvature the given 2-form.  It also characterises the moduli space of line bundles with connection that arise in this way.  This theorem was extended to the case of $p$-forms by Gajer in \cite{Ga}.  In this paper we provide a generalization of this theorem where we replace the original manifold by a derived smooth Artin stack.  Our derived Artin stacks are geometric stacks on the \'etale $\infty$-site of affine derived smooth manifolds.  We introduce the notion of a $n$-shifted $p$-preplectic derived smooth Artin stack in analogy with the algebraic case constructed by Pantev-To\"en-Vaqui\'e-Vezzosi in \cite{PTVV}.  This is a derived smooth Artin stack endowed with a complex closed $(p+1)$-form which has been cohomologically shifted by degree $n$.  It is a far reaching generalization of a $p$-preplectic manifold which includes orbifolds and other highly singular objects.  We then show that when its $n$-shifted $p$-preplectic form is integral, then there exists a $(p+n-1)$-gerbe with $p$-connection data and curvature corresponding to the original $p$-preplectic form.  We also provide the characterization of the moduli stack of gerbes with connections arising in this context.  We construct a canonical functor from the $(\infty,1)$-category of integral $n$-shifted $p$-preplectic derived smooth Artin stacks to the $(\infty,1)$-category of linear $(\infty,p+n-1)$-categories.  When $n=0$ and $p=1$, this functor can be thought of like a cohomology functor in that it associates to a derived presymplectic smooth Artin stack a linear invariant in the form of a differential graded module.  In the general case we obtain higher prequantum categories which requires the machinery of linear $(\infty,n)$-categories.
\end{titlepage}


\tableofcontents

\section*{Introduction}

In this paper we construct linear invariants of certain derived stacks of a smooth nature.  Our first task is to prove a version of the Weil-Kostant integrality theorem in the setting of derived smooth geometry.  The Weil-Kostant integrality theorem states that given a manifold endowed with an integral complex closed 2-form, then there exists a line bundle with connection on this manifold such that the curvature of this connection coincides with the original 2-form.  It also characterises the moduli space of line bundles with connection that arise in this way.  

More precisely we have the following theorem (see for example \cite{Br} or \cite{We}\cite{Ko} for the original references).

\begin{thm1}[Weil-Kostant]
Let $X$ be a smooth manifold endowed with a complex closed $2$-form $\omega$.  
\begin{enumerate}
\item If $\omega$ is integral, ie. the class $[\omega]$ lies in the image of the map 
\[  H^2(X,\bb{Z}(1))\ra H^2(X,\bb{C}), \]
where $\bb{Z}(1):=(2\pi\sqrt{-1})\cdot\bb{Z}$, then there exists a pair $(L,\theta)$ consisting of a line bundle $L$ on $X$ with connection $\theta$ such that $\omega$ is the curvature of $\theta$.
\item The set of isomorphism classes of pairs $(L,\theta)$ with curvature $\omega$ form a torsor for the group $H^1(X,\bb{C}^*)$ of isomorphism classes of flat line bundles over $X$.
\end{enumerate}
\end{thm1} 

This theorem was generalized to integral complex closed $p$-forms by Gajer in \cite{Ga}.  However, many objects of interest to us, for example non-transverse intersections of manifolds or quotients of a manifold by a Lie group with nonfree action, are not contained in these theorems.  We need to introduce a category of more general objects which includes these examples and which contains the category of manifolds as a full subcategory.  The first objective of this article is to provide a proof of this theorem in the case where our objects are derived smooth Artin stacks.  

We start by embedding the category of manifolds into the category of \textit{derived manifolds}.  These are a generalization of the quasi-smooth derived manifolds introduced in \cite{Sp1}.  One can think of a derived $k$-manifold, where $k$ is the field of real or complex numbers, as a dg-ringed topological space with extra structure, ie. a topological space endowed with a sheaf of commutative differential graded $k$-algebras with structure enabling one to ``compose with smooth functions" ($\cC^\infty$ or holomorphic functions) which is moreover, locally given by a finite limit of $k$-manifolds.  The foundations of the general theory of structured spaces was laid out in \cite{LV}.  

In considering derived $k$-manifolds instead of ordinary $k$-manifolds we gain, in addition to the inclusion of far more general spaces, better formal properties of the $\infty$-category of such objects.  For example, the $\infty$-category of derived $k$-manifolds is closed under finite limits.  The finite limits in the category of $k$-manifolds that are correct, for example transverse intersections, are preserved by the fully faithful functor from $k$-manifolds to derived $k$-manifolds.

Still further examples are not contained in the $\infty$-category of derived $k$-manifolds.  For example the category of derived $k$-manifolds is not closed under arbitrary colimits and we would like to include possibly singular quotients of manifolds in our theorem.  We build from the $\infty$-category of derived $k$-manifolds the notion of a derived $k$-smooth stack.  A derived $k$-smooth stack will be defined as a sheaf of spaces on the $\infty$-site of affine derived $k$-manifolds with respect to the \'etale topology.  An affine derived $k$-manifold is a local model for a derived $k$-manifold.

We show that the $\infty$-site of affine derived $k$-manifolds with the \'etale topology is subcanonical by showing that the presheaf of $\infty$-categories sending an affine derived $k$-manifold to its $\infty$-category of modules is a sheaf of $\infty$-categories.  As usual, objects in the essential image of the Yoneda embedding from the $\infty$-category of affine derived $k$-manifolds to the $\infty$-category of derived $k$-smooth stacks will be called affine derived $k$-smooth stacks. 

We then define what it means for a derived $k$-smooth stack to be Artin.  One can roughly think of a derived $k$-smooth Artin stack as a presheaf of spaces on the $\infty$-category of affine derived $k$-manifolds which is a sheaf for the \'etale topology and which is locally representable by an affine derived $k$-manifold with respect to the smooth topology.  This uses the theory of geometries outlined in \cite{TVII} which we recall.  We also discuss examples of derived $k$-smooth Artin stacks.  They can be presented as quotients by derived Lie groupoid actions.  The main reason for restricting to the collection of Artin stacks is that it includes all the examples of interest to us whilst guaranteeing the existence of a cotangent complex for such objects.  This is necessary for studying presymplectic geometry in our context in the subsequent sections.  

In summary, we will introduce three $\infty$-categories which lie to the right of the category of $k$-manifolds in a chain of inclusions
\[   \bMan_k\subset\dMan_k\subset\dSmAr_k\subset\dSmSt_k   \]
with obvious notation, in order to deal with examples whose structure is inaccessible from the first category.  This chain can be compared with the algebraic setting from smooth varieties and derived schemes up to derived (Artin) stacks.  A similar chain holds in the complex analytic setting.  From the discussion above, derived $k$-smooth Artin stacks often arise as solutions to derived moduli problems in the smooth setting.  

The analogue of a smooth manifold endowed with a complex closed $(p+1)$-form, or what one may call a $p$-preplectic manifold, in our setting is a derived $n$-shifted $p$-preplectic smooth Artin stack.  This is a derived smooth Artin stack (over $\bb{R}$) endowed with a complex closed $(p+1)$-form that has been cohomologically shifted by degree $n$.  In the case where the derived $n$-shifted $p$-preplectic smooth Artin stack is simply a smooth manifold endowed with a zero shifted 2-form, we recover the theory of presymplectic $k$-manifolds.  However, zero shifted $p$-forms exist on spaces containing singularities and so our definition is a natural extension of $p$-preplectic structures and can be utilized in many more general examples.  

We also define what it means for a complex closed $p$-form on a derived smooth Artin stack to be integral.  In analogy with the standard definition, it will mean that its cohomology class is the image of an integral class.  

The first main result of this paper, a derived version of the Weil-Kostant integrality theorem, is stated as follows (see Section~\ref{dwkit}, Theorem~\ref{mainthm}).

\begin{thm1}
Let $(X,\omega)$ be a derived $n$-shifted $p$-preplectic smooth Artin stack.  
\begin{enumerate}
\item There exists a $(p+n-1,p)$-gerbe on $X$ with curvature $\omega$ if and only if $\omega$ is integral.
\item The space of $(p+n-1,p)$-gerbes on $X$ with curvature $\omega'$ is parametrized by the space of flat $(p+n-1,p)$-gerbes.
\end{enumerate}
\end{thm1} 

A $(p,q)$-gerbe on a derived smooth stack is a $p$-gerbe with $i$-connections, where $i$ ranges from $1$ to $q$, on the derived smooth stack (see Definition~\ref{gerbeonstack}).  We also use the terminology $p$-gerbe with $q$-connection data.  It is an extension of the notion of a complex line bundle with connection on a smooth manifold.  Our theorem subsumes the classical Weil-Kostant integrality theorem, reconstructing it when the derived smooth Artin stack is an integral ($0$-shifted) $1$-preplectic $k$-manifold.  In this case a $(0,1)$-gerbe is simply a complex line bundle with connection.  When our derived smooth stack is a (singular) smooth space endowed with an integral complex closed $2$-form, this $(0,1)$-gerbe is understood as a complex line bundle with connection in a \textit{derived} sense, ie. the line bundle is a bundle of complexes of $k$-modules which encodes how the form differs from being smooth.  The notion of derived geometry arising when one deals with singular spaces is well known and exemplified in this result.

One application of the classical Weil-Kostant integrality theorem is to studying linear invariants of smooth spaces.  Classically, linear invariants of integral presymplectic smooth manifolds arise by considering the complex vector space of sections of the complex line bundle arising from the Weil-Kostant integrality theorem.  This construction is used for example in the theory of geometric quantization, where it is sometimes referred to as the associated prequantum vector space, and is useful in understanding the quantization of classical mechanical systems.  The functor which associates to an integral presymplectic smooth manifold a prequantum vector space is like a cohomology functor but satisfies different functorial properties.  

Similarly, one of the main utilities of the derived Weil-Kostant integrality theorem is to the subject of \textit{derived} geometric quantization.  This is useful in understanding the quantization of classical field theories in an extended sense, ie. as extended quantum field theories in which higher categorical data is associated to manifolds of greater codimension than one.  In this case, the prequantization functor supplies, in addition to a prequantum vector space (or more generally, a prequantum complex of vector spaces) certain prequantum linear higher categories.   

The collection of pairs $(X,\omega)$ consisting of a derived smooth Artin stack together with an integral $n$-shifted complex closed $(p+1)$-form $\omega$ on $X$ form an $\infty$-category which we denote by $p\mbox{-}\PrPlAr_{n}^{\tu{in}}$.  The objects in this $\infty$-category will be called integral $n$-shifted $p$-preplectic derived smooth Artin stacks.  The collection of $\bb{C}$-linear $(\infty,m)$-categories also form an $\infty$-category denoted $m\mbox{-}\bLin$.  Using the derived Weil-Kostant integrality theorem we prove the following second main result of this paper (see Section~\ref{pqc}, Corollary~\ref{maincor}).   

\begin{thm2}\label{mainthm2}
Let $p>0$, $n\in\bb{Z}$ and $(p+n)>0$.  There exists a prequantum functor
\[    \sP_n^p:p\mbox{-}\PrPlAr_{n}^{\tu{in}}\ra (p+n-1)\mbox{-}\bLin \]
of complex linear $(\infty,p+n-1)$-categories.
\end{thm2}

The application of this theorem to derived Artin stacks arising from moduli problems in classical field theory will appear elsewhere.

\section*{Notation}

An $\infty$-category will refer to an $(\infty,1)$-category, the theory of which is contained in \cite{L1} and \cite{Si}.  The opposite of an $\infty$-category $C$ will be denoted $C^\circ$.  Appendix~A contains a summary of the theory of $\inftyn$-categories, based on the approach in \cite{Si}, which is sufficient for our purposes.  More advanced structures in the formal theory of $\inftyn$-categories, needed in the main text, have been relegated to Appendix~B.

Given a simplicial model category $\sM$, the $\infty$-category arising from the localization of $\sM$ with respect to its class of weak equivalences will always be denoted $L(\sM)$.  In the setting of \cite{Si}, this corresponds to the Dwyer-Kan simplicial localization and in the setting of \cite{L1}, it corresponds the homotopy coherent nerve of the category $\sM^{f/c}$ of fibrant-cofibrant objects in $\sM$.

All $\infty$-categories of special note will be written in boldface.  In particular, we let $\textbf{S}$ denote the $\infty$-category
\[  \textbf{S}:=L(\tu{sSet})\xras L(\Top)=:\bTop  \] 
of spaces given by the localization of the category $\sSet$ of simplicial sets endowed with the Kan model structure.  This is equivalent to the $\infty$-category $\bTop$ of topological spaces where the category $\Top$ of topological spaces is endowed with its standard model structure.  The $\infty$-category of functors between two objects $C$ and $D$ in the model category of $\infty$-categories will be denoted $\bFun(C,D)$.

Finally, questions about categorical \textit{size} will be neglected throughout and can be addressed through the implementation of universes.

\section*{Acknowledgements}

The results of the derived Weil-Kostant integrality theorem were presented at the CATS 4 conference (Luminy, 2 - 7 July 2012).  I would like to thank the organisers for the invitation to speak.  Special thanks to Bertrand To\"en for many helpful discussions.  I also thank Jacob Lurie and Mauro Porta for correspondence related to this work.  This research was partly carried out at Harvard University and Max Planck Institute for Mathematics and partially supported by the World Premier International Research Center Initiative (WPI), MEXT, Japan.

\section{Derived manifolds}\label{dm}

Throughout this article we will fix $k\in\{\bb{R},\bb{C}\}$ to be the field of real or complex numbers.  We will call a function $\bb{R}$-smooth if it is of class $\cC^\infty$ and $\bb{C}$-smooth if it is holomorphic.  We will speak of $k$-smooth functions in general.  In this section we introduce the notion of derived $k$-manifold using the theory of structured spaces contained in \cite{LV}.

Every smooth or complex manifold has a description as a locally ringed space.  Let $V$ be an open subset of $k^n$ and $\cO_V$ the sheaf of $k$-smooth functions on $V$.  A pair $(|M|,\cO^\disc_M)$ consisting of a topological space $|M|$ (here assumed to be Hausdorff and second countable) together with a sheaf $\cO^\disc_M$ of commutative $k$-algebras on $|M|$ will be called a \textit{$k$-manifold} if for every point $x$ in $|M|$, there exists a neighborhood $U$ of $x$ in $|M|$ and a map
\[    (U,\cO^\disc_M|U)\ra(V,\cO_V)    \]
of ringed spaces such that the induced map $f:U\ra V$ is a homeomorphism of $U$ onto $V$ and the map 
\[  f^\#:\cO_V\ra f_*(\cO^\disc_M|U)   \]
is an isomorphism of sheaves.  Let $\Man_k$ denote the full subcategory of the category of ringed spaces over $k$ spanned by $k$-manifolds.

In order to define the notion of a derived $k$-manifold, it is not enough to replace the sheaf of $k$-algebras with a sheaf of ``derived $k$-algebras", for example a sheaf of simplicial commutative $k$-algebras.  One needs to include the smooth structure by considering simplicial $k$-smooth rings.  A convenient formal setting in which to consider such objects is within the theory of pregeometries found in \cite{LV}.     

Let the category $\Man_k$ of $k$-manifolds be endowed with the Grothendieck topology which is generated by locally homeomorphic (in the $k=\bb{R}$ case) or locally biholomorphic maps (in the $k=\bb{C}$ case).  For ease of notation, we follow \cite{LV} and refer to these generating maps as \textit{admissible}.  A collection $\{N_i\ra M\}$ of admissible maps generates a covering sieve on $M$ if and only if, for all $x\in M$, some inverse image $N_i\times_M\{x\}$ is nonempty.  

\begin{dfn}\label{pkstructure}
Let $\sM$ be a model category.  A functor $\cO:\Man_k\ra\sM$ is said to be a \textit{$k$-manifold structure} on $\sM$ if it satisfies the following conditions~:
\begin{enumerate}
\item The functor $\cO$ preserves finite products.
\item The functor $\cO$ preserves pullback diagrams along admissible maps.
\end{enumerate}
\end{dfn}

Let $\Str_{\Man_k}(\sM)$ denote the full subcategory of $L(\sM^{\Man_k})$ spanned by $k$-manifold structures.  Note that we have an equivalence
\[   L(\sM^{\Man_k})\xras\bFun(\Man_k,L(\sM))  \] 
of $\infty$-categories.

\begin{dfn}\label{local}
Let $\sM$ be a model category.  A $k$-manifold structure on $\sM$ is said to be \textit{local} if for any collection of admissible maps $\{N_i\ra M\}$ which generates a covering sieve, the induced map 
\[   \coprod_i\cO(N_i)\ra\cO(M)   \] 
is an effective epimorphism in $\sM$. 
\end{dfn}

This definition means that the $\check{\tu{C}}$ech nerve of the induced map is a simplicial resolution of $\cO(M)$ or equivalently, it is an effective epimorphism in the underlying homotopy category $\h(\sM)$.  

Let $\Str^\loc_{\Man_k}(\sM)$ denote the subcategory of $\Str_{\Man_k}(\sM)$ consisting of local $k$-manifold structures $\cO$ on $\sM$ and maps $f:\cO\ra\cO'$ of $k$-manifold structures on $\sM$ satisfying the following condition~: for all admissible morphisms $N\ra M$ in $\Man_k$, the diagram
\begin{diagram}
\cO(N)  &\rTo  &\cO'(N)\\
\dTo     &        &\dTo\\
\cO(M)  &\rTo  &\cO'(M)
\end{diagram}
is a pullback in $\sM$.  

\begin{construction}\label{constr}
Let $\sN$ be a (simplicial) model category and consider the functor 
\[   \Sh_\sN:\Top^\circ\ra\PC(\sSet)   \]
sending $|A|$ to the model category $\Sh_{\sN}(A)$ of $\sN$-valued sheaves on $|A|$.  Here $\PC(\sSet)$ is the model category of $(\infty,1)$-precategories (the notation of which is explained in Appendix~A).  We denote by $\int_{\Top}\Sh_\sN\ra\Top$ the fibered $\infty$-category classified by the functor $\Sh_{\sN}$.  

Let $\bTop_{\sN}(\Man_k)$ (resp. $\bTop^\loc_{\sN}(\Man_k)$) denote the opposite of the subcategory of 
\[    \bFun(\Man_k,\int_{\Top}\Sh_{\sN})\times_{\bFun(\Man_k,\bTop)}\bTop  \]
whose objects are (local) $k$-manifold structured spaces $(|A|,\cO_A)$ on $\Sh_{\sN}(A)$, and a map 
\[   (|A|,\cO_A)\ra(|B|,\cO_B)   \] 
is given by a map $f:|A|\ra |B|$ of topological spaces together with a (local) natural transformation $f^*\cO_B\ra\cO_A$ of $k$-manifold structures on $\Sh_{\sN}(A)$.  
\end{construction}

\begin{dfn}
Let $|A|$ be a topological space.  A \textit{discrete (local) $k$-manifold structure on $|A|$} is a (local) $k$-manifold structure on the trivial model category $\Sh_{\Set}(A)$ of sheaves of sets on $|A|$.
\end{dfn}

The $\infty$-category of discrete $k$-manifold structures on $|A|$ will be denoted 
\[  \Str^{\disc}_{\Man_k}(A):=\Str_{\Man_k}(\Sh_{\Set}(A)) \] 
and discrete local $k$-manifold structures on $|A|$ by $\Str^{\loc,\disc}_{\Man_k}(A):=\Str_{\Man_k}^\loc(\Sh_{\Set}(A))$.  

A pair $(|A|,\cO^\disc_A)$ consisting of a topological space $|A|$ together with a discrete (local) $k$-manifold structure $\cO_A^\disc$ on $|A|$ will be called a \textit{discrete (local) $k$-manifold structured space}.  The $\infty$-category of discrete (local) $k$-manifold structured spaces will be denoted $\bTop_{\Set}(\Man_k)$ and $\bTop^\loc_{\Set}(\Man_k)$ respectively.

\begin{ex}
When $k=\bb{R}$, the $\infty$-category $\bTop_{\Set}(\Man_k)$ is closely related to the category of $\cC^\infty$-ringed spaces.  See for example \cite{Du} for backround on $\cC^\infty$-ringed spaces.  In particular, when $|A|=*$, the $\infty$-category $\Str^{\disc}_{\Man_k}(*)$ is equivalent to the category of $\cC^\infty$-rings which preserve pullbacks of admissible maps. More generally, we call 
\[  \textbf{SmRng}_k:=\Str^{\disc}_{\Man_k}(*)  \] 
the $\infty$-category of \textit{$k$-smooth rings}.  One similarly has an $\infty$-category $\textbf{SmRng}_k^\loc:=\Str^{\loc,\disc}_{\Man_k}(*)$ of local $k$-smooth rings.
\end{ex}

\begin{notn}\label{notn1}
Let $C$ be an $\infty$-category.  If $X$ is a locally presentable $\infty$-category, we will denote by 
\[   \textbf{\bPr}_{X}(C):=\bFun(C^\circ,X)   \] 
the $\infty$-category of $X$-valued presheaves on $C$.  When $C$ is endowed with a Grothendieck topology $\tau$, we will refer to the pair $(C,\tau)$ as a \textit{$\infty$-site} and $\textbf{Sh}_{X}(C,\tau)$ the $\infty$-category of $X$-valued sheaves on $(C,\tau)$.  
A (pre) sheaf valued in the $\infty$-category $\textbf{S}$ of spaces will be called a (pre) stack and the $\infty$-category of (pre) stacks will be denoted $\textbf{Sh}(C,\tau)$ (resp. $\textbf{\bPr}(C)$).  

Recall that when $(D,\tau)$ is a model site and $Y$ is a model category satisfying $C=L(D)$ and $X=L(Y)$, then there exists a natural model category of $Y$-valued sheaves on $D$ and an equivalence
\[   \textbf{Sh}_{X}(C,\tau)\xras L(\Sh_Y(D))  \]
of $\infty$-categories.  For a topological space $|A|$, we will denote by 
\[  \textbf{\Sh}(A):=L(\Sh_{\sSet}(A))  \]
the $\infty$-category of stacks on $|A|$
\end{notn}

A useful lemma is the following~:

\begin{lem}\label{strtosh}
Let $|A|$ be a topological space and $\sM$ a model category.  Then there exists an equivalence
\[  \tu{\Str}_{\Man_k}^\loc(\tu{Sh}_{\sM}(A))\ra\tu{\bSh}_{\tu{\Str}_{\Man_k}^\loc(\sM)}(A)  \]
of $\infty$-categories.
\end{lem}

\begin{proof}
It is easy to state the sheaf condition for a $\sM$-valued sheaf on $|A|$ as a limit preserving functor from $\Sh(A)$ to $\sM$ using the equivalence $\bSh_{L(\sM)}(\bSh(A))\simeq\bSh_{L(\sM)}(A)$ (Proposition~1.1.12 of \cite{LV}).  Since (co)limits are calculated levelwise, we have a chain of equivalences
\[  \tu{\Str}_{\Man_k}^\loc(\tu{Sh}_{\sM}(A))\simeq\bFun^!(\Man_k,L(\Sh_\sM(A)))\simeq\bFun^{(!,l)}(\Man_k\times\bSh(A),L(\sM))\simeq  \]
\[  \bFun^l(\bSh(A),\bFun^!(\Man_k,L(M)))\simeq\tu{\bSh}_{\tu{\Str}_{\Man_k}^\loc(\sM)}(A) \]
of $\infty$-categories.  Here the superscript $!$ refers to the preservation of the conditions in Definition~\ref{pkstructure} and Definition~\ref{local} and $l$ refers to the preservation of limits.  The superscript $(!,l)$ refers to the preservation of the appropriate limits in each respective variable seperately.
\end{proof}

It follows from Lemma~\ref{strtosh} that there exists an equivalence
\[    \Str^{\disc}_{\Man_k}(A)\xras\bSh_{\Str_{\Man_k}(\Set)}(A)   \]
of $\infty$-categories, and likewise for local objects. 

\begin{dfn}
Let $|A|$ be a topological space.  A \textit{(local) $k$-manifold structure on $|A|$} is a (local) $k$-manifold structure on the model category $\Sh_{\sSet}(A)$ of sheaves of simplicial sets on $|A|$.
\end{dfn}

The $\infty$-category of $k$-manifold structures on $|A|$ will be denoted 
\[  \Str_{\Man_k}(A):=\Str_{\Man_k}(\Sh_{\sSet}(A))  \] 
and local $k$-manifold structures by $\Str^\loc_{\Man_k}(A):=\Str_{\Man_k}^\loc(\Sh_{\sSet}(A))$.  By Lemma~\ref{strtosh} there exists an equivalence
\[    \Str_{\Man_k}(A)\xras\bSh_{\Str_{\Man_k}(\sSet)}(A)   \]
of $\infty$-categories, and likewise for local objects.

A pair $(|A|,\cO_A)$ consisting of a topological space $|A|$ together with a $k$-manifold structure $\cO_A$ on $|A|$ will be called a \textit{$k$-manifold structured space}.  The $\infty$-category of $k$-manifold structured spaces will be denoted 
\[  \bTop(\Man_k):=\bTop_{\sSet}(\Man_k)  \]
and that of local $k$-manifold structured spaces by $\bTop^\loc(\Man_k):=\bTop_{\sSet}^\loc(\Man_k)$.  Our generalized manifolds will be defined as objects in a full subcategory of the $\infty$-category of local $k$-manifold structured spaces.  

Note that there is a natural functor
\[    \pi_0:\Str^\loc_{\Man_k}(A)\ra\Str^{\loc,\disc}_{\Man_k}(A)   \]
sending $A=(|A|,\cO_A)$ to $(|A|,\pi_0(\cO_A))$.  This functor is left adjoint to the inclusion and induces an adjunction
\[     \pi_0:\bTop(\Man_k)\rightleftarrows\bTop_{\Set}(\Man_k):i    \]
between $\infty$-categories (and their subcategories of local objects).

\begin{ex}
When $k=\bb{R}$, the $\infty$-category $\bTop(\Man_k)$ is closely related to the simplicial category of simplicial $\cC^\infty$-ringed spaces.  In particular, when $|A|=*$, the $\infty$-category $\Str_{\Man_k}(*)$ is equivalent to the $\infty$-category of simplicial $C^\infty$-rings which preserve pullbacks of admissible maps.  More generally, we call 
\[  \textbf{sSmRng}_k:=\Str_{\Man_k}(*)  \] 
the $\infty$-category of \textit{simplicial $k$-smooth rings}.  Similarly, one has the $\infty$-category $\textbf{sSmRng}_k^\loc:=\Str_{\Man_k}^{\loc}(*)$ of local simplicial $k$-smooth rings.
\end{ex}

\begin{rmk}
Let $A=(|A|,\cO_A)$ be a local $k$-manifold structured space.  The global sections functor $\Gamma:\bTop^\loc(\Man_k)\ra\textbf{sSmRng}_k$ sending $A$ to $\Gamma(A):=\cO_A(A)$ admits a right adjoint
\[    \Spec^\infty:\textbf{sSmRng}_k\ra\bTop(\Man_k)  \]
sending a $k$-smooth ring $R$ to a $k$-manifold structured space $(|R|,\cO_R)$, the details of which are carefully contained in \cite{BN}. 
\end{rmk}

Let $M=(|M|,\cO^\disc_M)$ be a $k$-manifold.  We define a (discrete) $k$-manifold structure $\cO_M$ on $|M|$ by
\[    \cO_M(N)(U):=\Hom_{\Man_k}((U,\cO_M^\disc|U),(|N|,\cO_N^\disc)    \]
for a $k$-manifold $N=(|N|,\cO_N^\disc)$ and a neighborhood $U\subseteq|M|$.  This construction defines a functor 
\[  \Spec^{\Man_k}:\tu{\bMan}_k\ra\bTop(\Man_k)  \]  
sending a $k$-manifold $M$ to the $k$-manifold structure $(|M|,\cO_M)$.  

\begin{prop}\label{fff}
The functor $\Spec^{\Man_k}$ factors fully faithfully through the $\infty$-category of local $k$-manifold structured spaces.
\end{prop}

\begin{proof}
Let $M=(|M|,\cO^\disc_M)$ be a $k$-manifold.  Then $\Spec^{\Man_k}$ sends $M$ to the $k$-manifold strutured space $(|M|,\cO_M)$.  The locality of $\cO_M$ translates into the requirement that the morphism
\[  \coprod_i\Hom_{\Man_k}(U,N_i)\ra\Hom_{\Man_k}(U,N)  \]
of sets is surjective for any collection of admissible maps $\{N_i\ra N\}$ generating a covering sieve.  This is satisfied if every $f:U\ra N$ factors through some $(|N_i|,\cO_{N_i}^\disc)$ (possibly after shrinking $U$).  But we can choose some $N_i$ and a neighborhood $V$ of some point $y$ in $|N_i|$, where $\alpha(y)=f(x)$, and owing to the admissibility of $\alpha:N_i\ra N$ construct an open embedding $\alpha_V:(V,\cO_{N_i}|V)\ra N$.  We may then shrink $U$ until $f(U)\subseteq\alpha_V(V)$.

Given a morphism $f:M\ra N$ of $k$-manifolds, the locality of a morphism $\cO_M\ra\cO_N$ is clear after unwinding the definition.  Therefore we must show that
\[   \Hom_{\Man_k}(M,N)\ra\Map_{\bTop^\loc(\Man_k)}((|M|,\cO_M),(|N|,\cO_N))  \]
is a bijection of sets (the space on the right is discrete).  It is clearly injective.  We need to then show that every morphism $(f,f^\sharp)$ on the right hand side comes from $f$ being $k$-smooth.  But given any chart $\phi_i:V_i\ra\bb{R}^n$ on $N$, the map $f^\sharp$ determines a $k$-smooth morphism $f^\sharp(\phi_i):f^{-1}(\phi_i)\ra\bb{R}^n$ by definition. 
\end{proof}

Owing to Proposition~\ref{fff}, we will often identify a manifold $M$ with its image under the fully faithful functor $\Spec^{\Man_k}$.

We will call the types of spaces which locally model our derived manifolds principal derived manifolds.  By definition, a principal derived $k$-manifold is a local $k$-manifold structured space which is given by a finite limit of $k$-manifolds.  

\begin{dfn}
A $k$-manifold structured space $A=(|A|,\cO_A)$ is said to be a \textit{principal derived $k$-manifold} if there exists an equivalence
\[   A\xras\lim_{i}M_i \]
in $\bTop^\loc(\Man_k)$ for a functor $M:I\ra\Man_k$ whose domain $I$ is a finite category.
\end{dfn}

Let $(|A|,\cO_A)$ be a $k$-manifold structured space and $i:U\subseteq |A|$ an open subset of $|A|$.  We denote by $(U,\cO_A|U)$ the $k$-manifold structured space where $\cO_A|U$ is given by the composition
\[      \Man_k\xra{\cO_A}\Sh_{\sSet}(A)\xra{i^*}\Sh_{\sSet}(U)  \]
of categories.  A derived $k$-smooth manifold is a $k$-manifold structured space which is given locally by a finite limit of $k$-smooth manifolds.

\begin{dfn}\label{derivedmanifold}
A $k$-manifold structured space $A=(|A|,\cO_A)$ is said to be a \textit{derived $k$-manifold} if for any point $x$ in $|A|$, there exists a neighborhood $U$ of $x$ such that the pair $(U,\cO_A|U)$ is a principal derived $k$-manifold.
\end{dfn}

Let $\dMan_k$ denote the full subcategory of $\bTop^\loc(\Man_k)$ spanned by derived $k$-manifolds.  It follows trivially that the functor of Proposition~\ref{fff} factors through the $\infty$-category of derived $k$-manifolds determining a fully faithful functor
\[    \Spec^{\Man_k}:\bMan_k\ra\dMan_k  \]
between $\infty$-categories.  Further, it can be deduced from Theorem~3.3.3 of \cite{Sp1} that, for any derived $k$-manifold $A=(|A|,\cO_A)$, there exists an equivalence
\[    \cO_A(M)(|A|)\xras\Map_{\dMan_k}(A,\Spec^{\Man_k}(M))  \]
of spaces.

\begin{dfn}
Let $A=(|A|,\cO_A)$ and $B=(|B|,\cO_B)$ be derived $k$-smooth manifolds.  Then a morphism $f:A\ra B$ in $\dMan_k$ is said to be a \textit{closed immersion} if the underlying morphism of topological spaces is a homeomorphism from $|A|$ to a closed subset of $|B|$ and the morphism $f^*\cO_B\ra\cO_A$ of $k$-manifold structures is an effective epimorphism.
\end{dfn}

A crucial advantage of working in the $\infty$-category of derived $k$-manifolds over the category of $k$-manifolds is the following~: 

\begin{prop}\label{finitelimits}
The $\infty$-category $\dMan_k$ of derived $k$-manifolds is closed under finite limits.
\end{prop}

\begin{proof}
Consider a diagram
\[   A=(|A|,\cO_A)\ra C=(|C|,\cO_C)\la B=(|B|,\cO_B)  \]
of derived $k$-manifolds.  The pullback of this diagram is equivalent to the pullback $D$ of the diagram
\[  A\times B\ra C\times C\xla{d} C  \]
where $d:C\ra C\times C$ is the diagonal map.  It is clear that $A\times B$ is itself a derived manifold.  The diagonal map is a closed immersion so it follows from \cite{LIX} that the pullback exists in $\bTop^\loc(\Man_k)$.  It is clear that locally, the $k$-manifold structured space $D$ is given by a finite limit of manifolds.  An $\infty$-category admits all finite limits if and only if it admits pullbacks and has a terminal object and therefore the statement is satisfied.
\end{proof} 

Other approaches to derived manifolds are contained in \cite{Sp1}, \cite{Sp2} and \cite{BN}.  Our notion of derived manifold is more general than that of those cited, the collection of whose objects are not closed under finite limits.  The relationship is as follows.  We call a $k$-manifold structured space a \textit{1-quasi smooth} derived $k$-manifold if it is given locally as a fiber product of $k$-manifolds.  Any such fiber product is locally the zero locus of a $k$-smooth function.  Therefore, the definition in \cite{Sp2} corresponds to $1$-quasi smooth derived $k$-manifolds.  

By induction, suppose we have the $\infty$-category of $n$-quasi smooth derived $k$-manifolds.  Then a \textit{$(n+1)$-quasi smooth} derived $k$-manifold is a local $k$-manifold structured space which is given locally as a fiber product of $n$-quasi smooth derived $k$-manifolds.  In this inductive definition, a $0$-quasi smooth derived $k$-manifold will be simply a $k$-manifold  

A derived $k$-manifold is therefore $n$-quasi smooth for some $n$.  In other words, if $n\mbox{-}\dMan_k$ denotes the full subcategory of $\bTop^\loc(\Man_k)$ spanned by $n$-quasi smooth derived $k$-manifolds then there exists an equivalence
\[     \dMan_k\xras\coprod_n n\mbox{-}\dMan_k  \]
of $\infty$-categories.

We would like to gain some perspective on which limits in the category of $k$-manifolds and derived $k$-manifolds coincide.  The following proposition confirms that transverse intersections of $k$-manifolds are preserved in the $\infty$-category of derived $k$-manifolds.  Moreover, the converse is true, ie. a fiber product of $k$-manifolds in the $\infty$-category of derived $k$-manifolds corresponds to the fiber product in the category of manifolds only if the intersection is transverse.  

\begin{prop}
The fiber product, if it exists, of a diagram in $\bMan_k$ is equivalent to the fiber product of the diagram in $\dMan_k$ if and only if the intersection is transverse.
\end{prop}

\begin{proof}
Let $M$, $N$ and $P$ be $k$-manifolds and consider the pullback diagram
\begin{diagram}
A &\rTo^{f'} &P\\
\dTo  &        &\dTo_g\\
N  &\rTo^{f}       &M
\end{diagram}
in the $\infty$-category of derived $k$-manifolds.  The pullback in the category of manifolds will be denoted $Q$.  If the intersection is transverse, ie. the map $f\coprod g:N\coprod P\ra M$ is a submersion, then Corollary~4.1.18 of \cite{Sp1} states that there exists an equivalence $Q\simeq A$ of $1$-quasi smooth derived $k$-manifolds and thus an equivalence of derived $k$-manifolds.  The converse follows from Theorem~4.2.1 of \textit{loc. cit.} using the same argument.
\end{proof}

Every derived $k$-manifold has an underlying algebraic description.  We will now show how to extract this underlying algebraic model.  Let $\cP_k$ denote the opposite of the full subcategory of the category $\tu{calg}_k$ of commutative $k$-algebras spanned by objects of the form $k[x_1,\ldots,x_n]$, ie. the opposite of the category of polynomial algebras over $k$.  Let $\cP_k$ be endowed with the trivial Grothendieck topology (generated by equivalences).  

Consider the $\infty$-category $\bTop(\cP_k)$ defined as in Construction~\ref{constr} by replacing $\Man_k$ with $\cP_k$, ie. it is the $\infty$-category of pairs $(|A|,F_A)$ where $|A|$ is a topological space together with a finite product preserving functor $F_A:\cP_k\ra\bSh(A)$ and a morphism is a map of topological spaces together with a local natural transformation of $\cP_k$-structures.  

Let $\dg_k$ denote the category of differential graded modules, or simply dg-modules, with its projective model structure and $\bdg_k:=L(\dg_k)$ the $\infty$-category of dg-modules over $k$ given by the localization with respect to its weak equivalences.  Similarly, let $\cdga_k$ denote the model category of coconnective commutative differential graded $k$-algebras with $\bcdga_k$ its corresponding $\infty$-category.  Let $\textbf{dgTop}_k$ denote the opposite of the $\infty$-category 
\[   \bFun(*,\int_{\Top}\Sh_{\cdga_k})\times_{\bFun(*,\bTop)}\bTop   \]
using Construction~\ref{constr}.  We will refer to $\textbf{dgTop}_k$ as the $\infty$-category of \textit{dg-ringed topological spaces} over $k$.

Consider the $\infty$-category $\textbf{scalg}_k:=L(\CMon(\tu{salg}_k))$ given by the localization of the simplicial model category  of commutative monoid objects in simplicial $k$-algebras.  The fibrations and weak equivalences in the model structure are given by those on their underlying simplicial sets.  If $C$ and $D$ are $\infty$-categories, we denote by $\bFun^l(C,D)$ the full subcategory of $\bFun(C,D)$ spanned by functors preserving all limits and $\bFun^{lex}(C,D)$ those preserving finite limits.  

For a topological space $|A|$, there exists a chain of equivalences
\[   \bSh_{\textbf{scalg}_k}(A)\xras\bFun^{l}(\bSh(A),\textbf{scalg}_k)\xras\bFun^{l}(\textbf{scalg}^\circ_k,\bSh(A))\xras\bFun^{lex}(\cP_k,\bSh(A))   \]
where the second equivalence follows from Proposition 4.1.9 of \cite{LV}.  Since there exists an equivalence $\textbf{scalg}_k\xras\bcdga_k$ of $\infty$-categories the equivalence
\[    \Str_{\cP_k}(A)\ra\bSh_{\bcdga_k}(A)   \]
follows.  Thus, there exists an equivalence
\[    \bTop(\cP_k)\ra\tu{\textbf{dgTop}}_k   \]
of $\infty$-categories.  Therefore, an object $(|A|,F_A)$ in $\bTop(\cP_k)$ may be interpreted as a topological space $|A|$ together with a $\bcdga_k$-valued sheaf on $|A|$.  An arrow from $(|A|,F_A)$ to $(|B|,F_B)$ in $\bTop(\cP_k)$ can be identified with a map $f:|A|\ra|B|$ of topological spaces together with a map $\alpha:f^*F_B\ra F_A$ of $\bcdga_k$-valued sheaves on $|A|$.  

The map $\rho:\cP_k\ra\Man_k$ sending a polynomial algebra $E$ to the set $\Hom_k(E,k)$ endowed with its natural $k$-manifold structure induces a diagram
\begin{diagram}
\bTop^\loc(\Man_k)  &&\rTo^{\rho^*}  &&\bTop(\cP_k) \\
 &\rdTo^{p_1}&         &\ldTo_{p_2}&\\
      && \bTop &&
\end{diagram}
where $p_1$ and $p_2$ project the respective structured spaces to their underlying topological spaces.  The functor $\rho^*$ is conservative meaning that a local morphism of local $k$-manifold structured spaces is an equivalence if and only if it is an equivalence of dg-ringed topological spaces.  This can be deduced from Proposition~11.9 of \cite{LIX}.  The fiber of the functor $\rho^*$ over a topological space $|A|$ will be denoted $\rho^*_{|A|}$.  We denote the image of a $k$-manifold structured space $A=(|A|,\cO_A)$ under $\rho^*_{|A|}$ by 
\[  A^\alg:=(|A|,\cO_A^\alg)   \]
and identify $\cO_A^\alg$ with a $\bcdga_k$-valued sheaf.  

The composition of $\rho^*:\bTop(\Man_k)\ra\bTop(\cP_k)$ with the inclusion determines a functor
\[   -^\alg:\dMan_k\ra\textbf{dgTop}_k  \]
sending a derived $k$-manifold $A=(|A|,\cO_A)$ to the dg-ringed topological space $A^\alg=(|A|,\cO_A^\alg)$ which we call the \textit{algebraic model} functor.   Composing this functor with the forgetful functor followed by the global sections functor $\Gamma$ determines a functor 
\[  \Gamma^\alg:\dMan_k\ra\bcdga_k  \] 
sending a derived $k$-manifold $A=(|A|,\cO_A)$ to $\Gamma^\alg(A):=\Gamma(\cO_A^\alg)$.

\begin{prop}\label{preservespullbacks}
The algebraic model functor preserves pullbacks of diagrams of the form 
\[  B\xra{f} A \xleftarrow{g} C  \] 
where $f$ induces an effective epimorphism $f^*\cO_A\ra\cO_B$.
\end{prop}

\begin{proof}
From Proposition 11.10 of \cite{LIX}, the natural map $\cP_k\ra \Man_k$ is unramified.  As a result of Proposition 10.3 of \textit{loc.cit.}, the functor
\[     \rho^*:\bTop^\loc(\Man_k)\ra\bTop(\cP_k)    \]
preserves pullbacks of diagrams of the form stated in the proposition.  The result now follows from the equivalence between $\bTop(\cP_k)$ and $\tu{\textbf{dgTop}}_k$.
\end{proof}

\begin{ex}
It follows from Proposition~\ref{finitelimits} that the natural category in which to consider the fiber product of two $k$-manifolds is $\dMan_k$.  Let $M$, $N$ and $P$ be three $k$-manifolds.  The fiber product $Q=(|Q|,\cO_Q):=M\times_P N$ is a derived $k$-manifold.  When the condition of Proposition~\ref{preservespullbacks} is satisfied, this corresponds to a pullback in the $\infty$-category of dg-ringed topological spaces over $k$.  We can assign the following model theoretic interpretation of this object. 

We define a model structure on the underlying category $\tu{dgTop}_k$ of dg-ringed topological spaces where a morphism $(|A|,\cO_A^\alg)\ra(|B|,\cO_B^\alg)$ is a fibration (resp. weak equivalence) if and only if $\cO_A^\alg\ra\cO_B^\alg$ is a fibration (resp. weak equivalence) in the model category $\Sh_{\cdga_k}(A)$ of sheaves with the pointwise projective model structure.  Then there exists a natural equivalence
\[     \textbf{dgTop}_k\ra L(\tu{dgTop}_k)   \]
of $\infty$-categories.  Therefore one can calculate the pullback using a homotopy pullback in the model category $\tu{dgTop}_k$.  In general, pullbacks of derived $k$-manifolds satisfying the condition of Proposition~\ref{preservespullbacks} have a convenient algebraic model theoretic interpretation. 
\end{ex}

\begin{ex}
The functor $\cP_k\ra\Man_k$ induces a functor $\Str_{\Man_k}(*)\ra\Str_{\cP_k}(*)$ and thus, by the above discussion, a natural map
\[  \textbf{sSmRng}_k\ra\bcdga_k  \] 
from simplicial $k$-smooth rings to commutative differential graded algebras over $k$.  Similarly, there is a natural map from local simplicial $k$-smooth rings to coconnective commutative dg-algebras over $k$.
\end{ex}

\begin{prop}
The $\infty$-category $\tu{\bTop}^{\loc}(\Man_k)$ of local $k$-manifold structured spaces is equivalent to the subcategory of $k$-manifold structured spaces spanned by pairs $(|A|,\cO_A)$ such that~: 
\begin{enumerate}
\item The discrete algebraic structure sheaf $\pi_0(\cO_A^\alg)$ is a sheaf of local $k$-algebras.
\item For any morphism $f:(|A|,\cO_A)\ra(|B|,\cO_B)$ of such objects, the induced morphism $\pi_0(f^*\cO_B^\alg)\ra\pi_0(\cO_A^\alg)$ of sheaves of local $k$-algebras is a local morphism of $k$-algebras.
\end{enumerate}
\end{prop}

\begin{proof}
For the first part, a $k$-manifold structured space $(|A|,\cO_A)$ is local if and only if, for every point $p:*\ra |A|$, the stalk $p^*(\cO_A)$ is a local simplicial $k$-smooth ring.  A map in the model category $\sSet$ is an effective epimorphism if and only if it is surjective on connected components.  Therefore the (discrete) $k$-smooth ring $\pi_0(p^*(\cO_A))$ is local.  

For any morphism $f:(|A|,\cO_A)\ra(|B|,\cO_B)$ of such objects, the induced morphism on discrete algebraic stalks $\pi_0(p^*(f^*\cO_B^\alg))\ra\pi_0(p^*(\cO_A^\alg))$ is a local morphism of $k$-algebras.  
\end{proof}

\begin{ex}\label{bk}
A \textit{derived Lie group} over $k$ is a group object in the $\infty$-category of derived $k$-manifolds.  The fully faithful functor $\Spec^{\Man_k}$ (precomposed with the inclusion) induces a fully faithful functor
\[   \textbf{Lie}_k\ra\textbf{dLie}_k  \]
between the $\infty$-category $\textbf{Lie}_k:=\textbf{Gp}(\bMan_k)$ of Lie groups over $k$ and the $\infty$-category $\textbf{dLie}_k:=\textbf{Gp}(\dMan_k)$ of derived Lie groups over $k$.  Formally, group objects exist in any $\infty$-category admitting fiber products so we have an equivalence
\[     \textbf{dLie}_k\ra\coprod_n n\mbox{-}\textbf{dLie}_k     \]
of $\infty$-categories where $n\mbox{-}\textbf{dLie}_k:=\textbf{Gp}(n\mbox{-}\dMan_k)$ is the $\infty$-category of $n$-quasi smooth derived Lie groups over $k$.
\end{ex}

\begin{dfn} 
Let $A$ be a derived $k$-manifold.  Then $A$ is said to be \textit{smooth} if there exists an equivalence 
\[    A\ra\Spec^{\Man_k}M     \]
of derived $k$-manifolds for a $k$-manifold $M$.  
\end{dfn}

The following result shows that every derived $k$-manifold can be embedded into a smooth derived $k$-manifold.

\begin{prop}
Let $A=(|A|,\cO_A)$ be a derived $k$-manifold.  Then there exists a closed immersion
\[   A\ra\Spec^{\Man_k}(k^n)  \] 
for some $n$.
\end{prop}

\begin{proof}
We will work by induction.  Consider the diagram
\begin{diagram}
A &\rTo  & B'\\
\dTo & &\dTo\\
B  &\rTo  &B''
\end{diagram}
in the $\infty$-category $\bTop^\loc(\Man_k)$ of local $k$-manifold structured spaces.  If $B$, $B'$ and $B''$ are $k$-manifolds then $A$ is a $1$-quasi smooth derived $k$-manifold and, by Theorem~6.1.5 of \cite{Sp1}, there exists a closed immersion $A\ra\Spec^{\Man_k}(k^N)$ for some $N$.  

Now let $B$, $B'$ and $B''$ be $n$-quasi smooth derived $k$-manifolds and assume that there exists closed immersions $B\ra\Spec^{\Man_k}(k^N)$, $B'\ra\Spec^{\Man_k}(k^{N'})$ and $B''\ra\Spec^{\Man_k}(k^{N''})$.  We need to show that the $(n+1)$-quasi smooth derived $k$-manifold $A$ admits a closed immersion into something smooth.  However, there exists a closed immersion $A\ra X$ where $X$ is the pullback of the diagram
\[     \Spec^{\Man_k}(k^N)\ra\Spec^{\Man_k}(k^{N''})\leftarrow\Spec^{\Man_k}(k^{N'})   \]
in $\bTop^\loc(\Man_k)$.  The derived $k$-manifold $X$ is itself $1$-quasi smooth so there exists a closed immersion $X\ra\Spec^{\Man_k}(k^M)$ for some $M$.  Closed immersions are stable under composition so there exits a closed immersion $A\ra\Spec^{\Man_k}(k^M)$. 
\end{proof}

\section{Derived smooth stacks}\label{dsas}

In Section~\ref{dm} we defined the notion of a derived $k$-manifold and showed that the $\infty$-category of such objects admits some nice formal properties.  In addition, the fully faithful functor from the category of $k$-manifolds into derived $k$-manifolds preserves the correct geometric structure.  Moreover, under certain conditions, the structure of a derived $k$-manifold can be probed using a purely algebraic model.  

In this section we introduce the $\infty$-category of derived $k$-smooth stacks which contains spaces arising from moduli problems which are not contained in the $\infty$-category of derived $k$-manifolds.  We begin by stating which objects we regard as being affine.

\begin{dfn}
A local $k$-manifold structured space $A=(|A|,\cO_A)$ is said to be a \textit{$1$-quasi smooth affine derived $k$-manifold} if there exists a fiber product
\begin{diagram}
A  &\rTo  & k^0\\
\dTo &     &\dTo_0 \\
k^m &\rTo^{f} & k^{m'}
\end{diagram}
in the $\infty$-category $\bTop^\loc(\Man_k)$ of local $k$-manifold structured spaces for some $k$-smooth function $f:k^m\ra k^{m'}$.  Given a $n$-quasi smooth affine derived $k$-manifold, a \textit{$(n+1)$-quasi smooth affine derived $k$-manifold} is a fiber product of $n$-quasi smooth affine derived $k$-manifolds.  An \textit{affine derived $k$-manifold} is a $n$-quasi smooth affine derived $k$-manifold for some $n$.
\end{dfn}

The full subcategory of $\bTop^\loc(\Man_k)$ spanned by affine derived $k$-manifolds is denoted $\bdAff_k$.  The $\infty$-category $\bdAff_k$ of affine derived $k$-manifolds can be given the structure of an $\infty$-site using the following notion of an \'etale morphism between local $k$-manifold structured spaces.  

\begin{dfn}\label{coveringmaps}
A morphism $(|A|,\cO_A)\ra(|B|,\cO_B)$ between affine derived $k$-manifolds is said to be \textit{\'etale} if 
\begin{enumerate}
\item The underlying morphism $f:|A|\ra |B|$ of topological spaces is a local homeomorphism.
\item The map $f^*\cO_B\ra\cO_A$ is an equivalence in $\Str^\loc_{\Man_k}(A)$.
\end{enumerate}
\end{dfn}

The \'etale maps between affine derived $k$-manifolds are stable by composition, finite limits and equivalences.  Therefore the \'etale maps generate a topology on the $\infty$-category of affine derived $k$-manifolds.  

The $\infty$-site of affine derived $k$-manifolds with the \'etale topology will be denoted $(\bdAff_k,et)$.  This is the $\infty$-site on which we define our notion of derived stack in the $k$-smooth setting.

\begin{dfn}
A \textit{derived $k$-smooth stack} is a stack on the $\infty$-site $(\bdAff_k,et)$.
\end{dfn}

Let $\dSmSt_k:=\bSh(\bdAff_k,et)$ denote the full subcategory of $\textbf{Pr}(\bdAff_k)$ spanned by derived $k$-smooth stacks.  

We will now show that the \'etale topology on the $\infty$-category of affine derived $k$-manifolds is subcanonical.  To do so we introduce an appropriate $\infty$-category of modules over an affine derived $k$-manifold.  

Sheaves on a topological space $|A|$ with values in a simplicial model category have an induced pointwise injective simplical model category structure.  Therefore the category of modules over a commutative monoid object $F$ in $\Sh_{\dg_k}(A)$ has a natural simplicial model structure \cite{SS}.  We define
\[   \bMod(F):=L(\Mod_F(\Sh_{\dg_k}(A)) ) \]
to be the $\infty$-category of modules over $F$.  Note that we can identify $F$ as an object in the $\infty$-category $\bSh_{\bcdga_k}(A)$ of sheaves of coconnective commutative dg-algebras.

One can give another characterization of the $\infty$-category of $F$-modules using spectra.  More precisely, given any $\infty$-category $C$ with finite limits, recall that one can define the $\infty$-category $\bSp(C)$ of $\Omega$-spectrum objects in $C$ where $\Omega$ is the loop space endofunctor. 

\begin{prop}\label{modisspectra}
Let $|A|$ be a topological space and $F$ a commutative monoid object in the category of sheaves of dg-modules on $|A|$.  Then there exists an equivalence
\[     \bMod(F)\ra\bSp(\textbf{\tu{Sh}}_{\bcdga_k}(A)_{/F})  \]
of $\infty$-categories.
\end{prop}

\begin{proof}
A concrete model for the right hand side is the localization $L(\Sp(\Sh_{\cdga_k}(A)_{/F}))$ of the stable model category of $\Omega$-spectrum objects in $\Sh_{\cdga_k}(A)_{/F}$ (see \cite{Ho}).  There exists a chain of Quillen equivalences
\[  \Sp(\Sh_{\cdga_k}(A)_{/F})\xras\Sp((\Sh_{\cdga_k}(A)_{F//F})\xras\Sp(\CMon(\Mod(F))_{/F})\xras    \]
\[    \Sp(\CMon^{nu}(\Mod(F))\xras\Sp(\Mod(F))\xras\Mod(F)  \]
of model categories where the superscript \textit{nu} refers to non-unital algebras.  The first equivalence can be deduced from Lemma~7.3.3.9 of \cite{L2}, the third equivalence follows from Lemma~1.2.1.3 of \cite{TVII} and the fifth equivalence from the fact that the model category of $F$-modules is a stable model category.
\end{proof}

It follows that the $\infty$-category of $F$-modules is a stable $\infty$-category.  Using this characterization of modules as spectrum objects in Proposition~\ref{modisspectra}, the $\infty$-category of $\cO_A$-modules over a local $k$-manifold structure $\cO_A$ will be defined as
\[   \bMod(\cO_A):=\bSp(\Str_{\Man_k}^{\tu{loc}}(A)_{/\cO_A}).  \]

Another convenient description of the $\infty$-category of $\cO_A$-modules is the following, where we denote by $\tu{Sp}(\sM)$ the model category of spectra in a model category $\sM$.  We denote by $\tu{Sp}:=\tu{Sp}(\sSet_*)$ the model category of spectra with respect to the positive stable model structure of \cite{Sh}.

\begin{lem}
Let $|A|$ be a topological space and $\cO_A$ a $k$-manifold structure on $|A|$.  There exists an equivalence
\[     \bMod(\cO_A)\ra\tu{\bSh}_{\tu{\Str}^\loc_{\Man_k}(\tu{Sp})}(A)_{/\cO_A}  \]
of $\infty$-categories.
\end{lem}

\begin{proof}
We have a chain of equivalences
\[   \bSp(\Str^\loc_{\Man_k}(A)_{/\cO_A})\xras\bSp(\Str^\loc_{\Man_k}(A))_{/\cO_A}\xras\Str^\loc_{\Man_k}(\Sp(\Sh_{\sSet}(A)))_{/\cO_A}\xras   \]
\[  \Str^\loc_{\Man_k}(\Sh_\Sp(A))_{/\cO_A}\xras\tu{\bSh}_{\Str^\loc_{\Man_k}(\tu{Sp})}(A)_{/\cO_A}  \]
of $\infty$-categories, where by abuse of notation, we have denoted by $\cO_A$ each of the relevant constant functors.
\end{proof}

It is often easier to work with modules over the underlying algebraic model of a $k$-manifold structure and the following result shows that we lose no information when making this choice.

\begin{prop}\label{modequivalence}
Let $|A|$ be a topological space and $\cO_A$ a $k$-manifold structure on $|A|$.  There exists an equivalence
\[    \psi:\bMod(\cO_A)\ra\bMod(\cO_A^\alg)  \]
of $\infty$-categories 
\end{prop}

\begin{proof}
We will show that the map
\[  \psi:\bSp(\Str_{\Man_k}^{\tu{loc}}(A)_{/\cO_A})\ra \bMod(\cO_A^\alg)  \]
sending a spectrum object $u:\cO_A'\ra\cO_A$ to its algebraic fiber $\tu{fib}(u)^\alg$ is fully faithful and essentially surjective.  For the first part, we construct a right adjoint $\gamma$ to $\psi$ and show that the unit $1\ra\gamma\circ\psi$ is an equivalence.  The map $\gamma$ sends a $\cO_A^\alg$-module $E$ to a spectrum object over $\cO_A$ which we denote by $\cO_A\oplus E$.  We define it as follows.  

Let $M$ be a manifold and denote by $T_M$ the tangent sheaf to $M$.  This is an object of the $\infty$-category $\bMod(\cO_M^\disc)\subset\bMod(\cO_M^\alg)$.  Let $O(A)$ be the category of open sets in $|A|$.  Consider the functor
\[  F:\Man_k\times O(A)^\circ\ra\sSet  \]
sending a pair $(M,U)$ to the simplicial set $\Map_{\bTop}(U,|M|)$.  Denote by 
\[   \int_{\Man_k^\circ\times O(A)} F\ra\Man_k^\circ\times O(A)  \] 
the fibered space classified by $F$.  We can construct the category
\[     \<F\>:=\tu{Fun}(*,\f{C}(\int_{\Man_k^\circ\times O(A)} F))\times_{\tu{Fun}(*,\Man_k^\circ\times O(A))} (\Man_k^\circ\times O(A))   \]
of triples $(M,U,f)$ where $f\in\Map_{\bTop}(U,|M|)$.  Here $\f{C}(S)$ is the simplicial category associated to a simplicial set $S$ given by the left adjoint to the simplicial nerve functor and the pullback is taken in the category of categories.

Consider the functor 
\[  G_E:\<F\>^\circ\ra\sSet  \]
sending a triple $(M,U,f)$ to the space of sections $\Gamma(\cO_A^\alg|U,(f^\alg)^*(T_M)\otimes_{\cO_A^\alg}E)$ for a $\cO_A^\alg$-module $E$.  Denote by 
\[  \int_{\<F\>} G_E\ra \<F\>  \] 
the fibered space classified by $G_E$.

Since fibered spaces are stable under composition, we have a fibered space 
\[  \int_{\<F\>} G_E\ra\Man_k^\circ\times O(A).   \]  
Let
\[    \cO_A\oplus E:\Man_k\times O(A)^\circ\ra\sSet   \]
denote its corresponding straightening functor and take its sheafification.  It is not difficult to check that it is a local $k$-manifold structure on $A$.  For the preservation of pullback diagrams along admissible maps we note that the tangent sheaf construction is functorial in $M$, preserves pullbacks and that $\Gamma(A,\lim_i Z_i)\simeq\lim_i\Gamma(A,Z_i)$ for any $\{Z_i\}$.

We now show that there exists an equivalence
\[ \cO_A\oplus\tu{fib}(u)^\alg(M,U)\simeq \cO_A'(M,U)  \]
for $M\in\Man_k$ and $U\subset |A|$.  It suffices to prove the equivalence on $M$ a (closed) submanifold of $k^n$.  Since this is a map of local $k$-manifold structures, we have a pullback diagram
\begin{diagram}
\cO_A\oplus\tu{fib}(u)^\alg(M,U) &\rTo &\cO_A'(M,U)\\
\dTo  &        &\dTo\\
\cO_A\oplus\tu{fib}(u)^\alg(k^n,U)  &\rTo       &\cO_A'(k^n,U)
\end{diagram}
of spaces.  Therefore it suffices to prove the equivalence on $k^n$ itself, or since $k$-manifold structures are product preserving, on $k$.  It is also sufficient to study the $0$-space of the algebraic spectrum due to the equivalence
\[    \cO_A(k)\simeq\Omega^\infty(\cO_A^\alg)    \]
of stacks on $|A|$.  When $M=k$ and $f:U\ra k$, then $(f^\alg)^*(T_k)=\cO_U^\alg$ and we have an equivalence
\[  \cO_A\oplus\tu{fib}(u)^\alg(k,U)\simeq\Omega^\infty\cO_A'^{\alg}(U)  \] 
of spaces as required. 

To prove essential surjectivity, it suffices to show that the counit $\rho\circ\gamma\ra 1$ is an equivalence or simply observe that $\rho$ commutes with colimits and use the generating properties of $\cO_A^\alg$ in $\bMod(\cO_A^\alg)$.
\end{proof}

The construction of the $\infty$-category of modules over a local $k$-manifold structure is functorial in $A$ which we describe as follows.  The $\infty$-category of $\infty$-categories will be denoted $\bCat_\infty$ (see the introduction to Section~\ref{cg}).  We define a presheaf
\[   \cS:(\bdAff_k)^\circ\ra\bCatinf \]
of $\infty$-categories sending an affine derived $k$-manifold $A=(|A|,\cO_A)$ to $\Str^\loc_{\Man_k}(A)_{/\cO_A}$.  It sends a morphism $(f,f^\sharp):A=(|A|,\cO_A)\ra B=(|B|,\cO_B)$ to the functor
\[    F:\Str_{\Man_k}^\loc(B)_{/\cO_B}\ra\Str_{\Man_k}^\loc(A)_{/\cO_A}   \]
sending $\alpha:\cO'_B\ra\cO_B$ to the composition $f^\sharp\circ f^*(\alpha)$.  

Let $\bCatinf^{\lp}$ denote the subcategory of $\bCatinf$ consisting of locally presentable $\infty$-categories and colimit preserving functors.  Let $\bCatinf^{\lp,\perp}$ denote the full subcategory of $\bCatinf^{\lp}$ spanned by objects which are moreover stable.  Then by Corollary~1.4.4.5 of \cite{L2}, the functor 
\[    \varsigma:\bCatinf^{\lp}\ra\bCatinf^{\lp,\perp}   \]
sending $C$ to $\bSp(C)$ is left adjoint to the forgetful functor.  

The functor $\cS$ factors through the $\infty$-category $\bCatinf^{\lp}$ of locally presentable $\infty$-categories and so we can form the composition functor
\[  \cM:=i\circ\varsigma\circ\cS:(\bdAff_k)^\circ\ra\bCatinf  \]
where $i:\bCatinf^{\lp,\perp}\ra\bCatinf$ is the inclusion.  It sends the affine derived $k$-manifold $A$ to the $\infty$-category $\bMod(\cO_A)$ of $\cO_A$-modules and $(f,f^\sharp)$ to a functor we denote by $\partial F$.  The latter satisfies the property that there exists a commutative diagram
\begin{diagram}
\bMod(\cO_B) &\rTo^{\partial F} &\bMod(\cO_A) \\
\dTo_{\Omega^\infty_*}     &   &\dTo_{\Omega^\infty_*}\\
\Str_{\Man_k}^\loc(B)_{/\cO_B}  &\rTo^{F}  &\Str_{\Man_k}^\loc(A)_{/\cO_A}
\end{diagram}
of $\infty$-categories.

Next we prove that this presheaf is actually a sheaf of $\infty$-categories.  

\begin{prop}\label{modisasheaf}
The functor $\cM$ is a sheaf of $\infty$-categories with respect to the \'etale topology.
\end{prop}

\begin{proof}
It suffices to check the conditions of Proposition 4.8 of \cite{W1}.  Let $A=(|A|,\cO_A)$ be an affine derived $k$-manifold.  From Proposition~\ref{modequivalence} there exists an equivalence $\bMod(\cO_A)\xras\bMod(\cO_A^\alg)$ of $\infty$-categories so we can work with the underlying algebraic model.
Firstly, the $\infty$-category $\bMod(\cO_{A}^\alg)$ is locally presentable since $\bdg^{\leq 0}_k$ is locally presentable (and thus so is the $\infty$-category of $\bdg^{\leq 0}_k$-valued sheaves on $|A|$).  Therefore, the $\infty$-category $\bMod(\cO_A^\alg)$ admits all (small) limits.  

Let $u:B=(|B|,\cO_B)\ra A$ be an \'etale map.  The map $\bMod(\cO_{A}^\alg)\ra\bMod(\cO_{B}^\alg)$ induced from $u^*$ then preserves limits.  Furthermore, the forgetful functor 
\[  u_*:\bMod(\cO_{B}^\alg)\ra\bMod(\cO_{A}^\alg)  \] 
is conservative and is right adjoint to $u^*$.  Finally consider the pullback square
\begin{diagram}
D &\rTo^{t} &C\\
\dTo^{s}  &        &\dTo_{v}\\\
B  &\rTo^{u}       &A
\end{diagram}
in $\bdAff_k$.  We must show that $v^*u_*\ra t_*s^*$ is an equivalence in the $\infty$-category 
\[  \bFun(\bMod(\cO_B^\alg),\bMod(\cO_C^\alg))   \]
of functors.  

Let $M\in\bMod(\cO_B^\alg)$.  Then we have a natural map
\[  v^*u_*(M)=v^{*}(u_*M\otimes_{u_*\cO_B^\alg}\cO^\alg_{A})\simeq v^{*}(u_*M\otimes_{u_*s_*\cO^\alg_D}v_*\cO^\alg_{C})\simeq t_{*}(s^*M\otimes_{s^*\cO^\alg_{B}}\cO^\alg_{D})=t_{*}s^{*}(M) \]
where the first equivalence follows from the equivalence 
\[  u_*\cO^\alg_{B}\coprod_{u_*s_*\cO^\alg_{D}}v_*\cO^\alg_{C}\simeq u_*\cO^\alg_{B}\otimes_{u_*s_*\cO^\alg_{D}}v_*\cO^\alg_{C} \] 
in the $\infty$-category $\bSh_{\bcdga_k}(A)$.
\end{proof}

\begin{prop}\label{subcanonical}
The \'etale topology on the $\infty$-category of affine derived $k$-manifolds is subcanonical.
\end{prop}

\begin{proof}
From Proposition~\ref{modisasheaf} the presheaf $\cM$ is a sheaf of $\infty$-categories on $(\bdAff_k,et)$ and so
\[    \bMod(\cO_{A})\ra\lim_\Delta\bMod(\cO_{B*})  \]
is an equivalence for any covering $B_*\ra A$ of $A$.  Therefore, for any $\cO_{A}$-module $E$, the unit map $E\ra\lim_{\Delta}(E\otimes_{A}B_*)$ is an equivalence.  If we take $E=A$ then $A\ra\lim_\Delta B_*$ is an equivalence and thus for any affine derived $k$-manifold  $C$, the composition
\[  \Map(C,A)\ra\Map(C,\lim_\Delta B_*)\ra\lim_\Delta\Map(C,B_*)  \]
is an equivalence.  Therefore the prestack $h_C$ is a stack with respect to the \'etale topology.
\end{proof}

Since the \'etale topology is subcanonical, the Yoneda embedding factors through the subcategory of derived $k$-smooth stacks and we denote by 
\[   \Spec:\bdAff_k\ra\dSmSt_k  \]
the resulting fully faithful functor.  Derived $k$-smooth stacks in the essential image of this functor will be called \textit{affine}.    

The Yoneda functor also embeds the $\infty$-category of derived $k$-manifolds into the $\infty$-category of derived $k$-smooth stacks.

\begin{prop}\label{danspdanst}
The Yoneda embedding
\[      h:\tu{\dMan}_k\ra\dSmSt_k   \]
with image restricted to affine derived $k$-manifolds is fully faithful.
\end{prop}

\begin{proof}
Every derived $k$-manifold $A=(|A|,\cO_A)$ is locally given by an affine derived $k$-manifold $U=(U,\cO_A|U)$ by choosing a refinement of the cover.  Therefore there exists an equivalence
$  h(A)\simeq\colim_{\Delta}(U_n) $
where the right hand side is the $\check{\tu{C}}$ech nerve of the cover and the result follows from the Yoneda lemma.  
\end{proof}

\section{Derived smooth Artin stacks}\label{dsas}

We will now review the types of geometric objects we would like to consider in this article.  We study a subcategory of derived stacks in the smooth setting called derived Artin stacks.  The benefit of restricting to this subcategory is that its objects are general enough to include all the examples of interest to us and restrictive enough to guarantee the existence of a workable infinitesimal theory.  This subcategory also retains good formal properties.  We use the language of geometries laid out in \cite{TVII} (see also \cite{LV}).  After defining these objects, we describe how to move between them in a functorial manner.    

\begin{dfn}\label{geometry}
A \textit{geometry} is a pair $((C,\tau),\sP)$ consisting of a $\infty$-site $(C,\tau)$ together with a collection of maps $\sP$ in $C$ satisfying the following conditions~:
\begin{enumerate}
\item The $\infty$-category $C$ admits finite limits and the topology $\tau$ on $C$ is subcanonical with covering families consisting of morphisms in $\sP$.
\item The class $\sP$ of maps is stable by composition, pullbacks and contains all equivalences.
\item Let $u:x\ra y$ be a map in $C$ and $\{y_i\ra y\}\in\tau(x)$ such that each $y_i\ra y$ and each $y_i\ra x$ is in $\sP$.  Then $u$ is in $\sP$. 
\end{enumerate}
\end{dfn}

Note that the first and final conditions highlight the local nature of the maps in $\sP$ with respect to the topology on the $\infty$-site.  

Let $(C,\tau)$ and $(D,\eta)$ be two $\infty$-sites.  We will say that a functor $f:C\ra D$ is \textit{topologically continuous} if the induced map
\[     f^*:\bPr(D)\ra\bPr(C)    \]
of $\infty$-categories preserves the full subcategory of stacks.   

\begin{lem}\label{fiscontinuous}
Let $(C,\tau)$ and $(D,\eta)$ be two $\infty$-sites such that $C$ and $D$ admit finite limits.  Let $f:C\ra D$ be a left exact functor which preserves covering families.  Then $f$ is topologically continuous.
\end{lem}

\begin{proof}
Let $G$ be a stack in $\bSh(D,\eta)$.  We must show that the prestack $f^*G$ is a stack on $C$, ie. that 
\[    f^*G(x)\ra\lim_\Delta f^*G(u_*)  \]
is an equivalence for all coverings $\{u_i\ra x\}$ in $C$.  This is equivalent to the condition that the map
\[    \Map(x,f^*G)\ra\lim_\Delta\Map(u_*,f^*G)\xras\Map(\colim_\Delta u_*,f^*G)  \]
is an equivalence or that, by adjunction, the map
\[   \Map(f_!x,G)\ra\lim_\Delta\Map(f_!u_*,G)  \]
is an equivalence.  However, since by assumption, $\{f_!u_i\ra f_!x\}$ is a cover of $f_!x$ in $D$ and $G$ satisfies descent, we have that this map is indeed an equivalence.  
\end{proof}

The induced functor $f^*:\bSh(D,\eta)\ra\bSh(C,\tau)$ admits a left exact left adjoint
\[     f_!^a:=a\circ f_!:\bSh(C,\tau)\ra\bSh(D,\eta)  \]
given by the composition of $f_!:\bSh(C,\tau)\ra\bPr(D)$ with the associated stack functor $a$.

\begin{dfn}
Let $((C,\tau),\sP)$ and $((D,\eta),\sQ)$ be geometries.  A functor $f:C\ra D$ is said to be a \textit{transformation of geometries} if
\begin{enumerate}
\item The functor $f$ preserves finite limits and is topologically continuous.
\item The functor $f$ sends maps in $\sP$ to maps in $\sQ$.
\end{enumerate} 
\end{dfn}

Let $\textbf{Geom}$ denote the subcategory of the $\infty$-category of $\infty$-sites and continuous maps consisting of geometries and transformations of geometries.  

Now recall the inductive definition of a geometric stack from \cite{TVII}.  A stack is $n$-geometric if it admits a $n$-atlas.  A map of representable stacks will be said to be in $\sP$ if it is the image under the Yoneda embedding of a map in $\sP$.  It follows from Corollary 1.3.3.5 of \cite{TVII} that the $\infty$-category of $n$-geometric stacks is stable under pullbacks and disjoint coproducts.  A stack is said to be \textit{geometric} if it is $n$-geometric for some $n$.  A map of stacks is said to be \textit{in} $\sP$ if it is in $n$-$\sP$ for some $n$.  

The subcategory of $\bSh(C,\tau)$ spanned by the geometric stacks will be denoted $\bSh(C,\tau;\sP)$.  We have a well defined presheaf of $\infty$-categories
\[    q:\textbf{Geom}^\circ\ra\bCatinf  \]
sending a geometry $((C,\tau),\sP)$ to $\bSh(C,\tau;\sP)$ and a transformation of geometries $f$ to $f^*$.

Some algebraic examples include the following.

\begin{ex}\cite{TVII}\label{spectraldm}
Let $(\bdAff^{\,\alg}_k,et)$ denote the $\infty$-site of affine stacks over $k$ for the \'etale topology.  Then this $\infty$-site together with the class $\sP$ of \'etale morphisms $et$ defines a geometry.  A \textit{derived Deligne-Mumford stack} over $k$ is a geometric stack on the geometry $((\bdAff^{\,\alg}_k,et),et)$.  We denote by 
\[  \textbf{dDM}_k:=\bSh(\bdAff^{\,\alg}_k,et;et)  \] 
the $\infty$-category of derived Deligne-Mumford stacks over $k$.
\end{ex}

\begin{ex}\cite{TVII}\label{spectralart}
The $\infty$-site $(\bdAff^{\,\alg}_k,et)$ together with the class $\sP$ of smooth morphisms $sm$ defines a geometry.  A \textit{derived Artin stack} over $k$ is a geometric stack on the geometry $((\bdAff^{\,\alg}_k,et),sm)$.  We denote by 
\[  \textbf{dAr}_k:=\bSh(\bdAff^{\,\alg}_k,et;{sm})  \] 
the $\infty$-category of derived Artin stacks over $k$.
\end{ex}

These examples also have an interpretation in complex analytic geometry.  A \textit{derived complex analytic space} is a local $k$-manifold structured space $(|A|,\cO_A)$ such that for any $x\in|A|$ there exists a neighborhood $U\subset|A|$ of $x$ such that $(U,\pi_0(\cO^\alg_A|U))$ is a complex analytic space and for each $k\geq 0$, the sheaf $\pi_k(\cO_A^\alg|U)$ is a coherent sheaf of $\pi_0(\cO_A^\alg|U)$-modules (see \cite{LIX}).  It is said to be affine if $(U,\pi_0(\cO^\alg_A|U))$ is a Stein space.  

\begin{ex}\cite{Po1}\label{analyticdm}
Let $(\bdAff^{\,\an}_\bb{C},et)$ denote the full subcategory of $\bTop^\loc(\Man_{\bb{C}})$ spanned by affine derived complex analytic spaces endowed with the \'etale topology.  This $\infty$-site of derived affine complex analytic spaces together with the class $\sP$ of \'etale morphisms $et$ defines a geometry.  A \textit{derived complex analytic Deligne-Mumford stack} is a geometric stack on the geometry $((\bdAff^{\,\an}_\bb{C},et),et)$.  We denote by 
\[  \textbf{dAnDM}_\bb{C}:=\bSh(\bdAff^{\,\an}_\bb{C},et;et)  \] 
the $\infty$-category of derived complex analytic Deligne-Mumford stacks. 
\end{ex}

\begin{ex}\cite{Po1}\label{analyticart}
The $\infty$-site $(\bdAff^\an_\bb{C},et)$ of affine derived complex analytic spaces together with the class $\sP$ of smooth morphisms $sm$ defines a geometry.  A \textit{derived complex analytic Artin stack} is a geometric stack on the geometry $((\bdAff^{\,\an}_\bb{C},et),sm)$.  We denote by 
\[  \textbf{dAnAr}_\bb{C}:=\bSh(\bdAff^{\,\an}_\bb{C},et;sm)  \] 
the $\infty$-category of derived complex analytic Artin stacks. 
\end{ex}

In the smooth setting, analogous results can be found.  The $\infty$-site $(\bdAff_k,et)$ of affine derived $k$-manifolds endowed with the \'etale topology together with the class $\sP$ of \'etale morphisms $et$ defines a geometry.  

\begin{dfn}
A \textit{derived $k$-smooth Deligne-Mumford stack} is a geometric stack on the geometry $((\bdAff_k,et),et)$.  
\end{dfn}

We denote by $\textbf{dSmDM}_k:=\bSh(\bdAff_k,et;et)$ the $\infty$-category of derived $k$-smooth Deligne-Mumford stacks.  

We would like to define an analogue of a derived Artin stack in the smooth setting.  We will use the following notion of smooth morphism of affine derived $k$-manifolds.

\begin{dfn}\label{smooth}
Let $A=(|A|,\cO_A)$ and $B=(|B|,\cO_B)$ be affine derived $k$-manifolds.  A morphism $f:A\ra B$ is said to be \textit{smooth} if, for every $x\in |A|$, there exists a neighborhood $U$ of $x$ such that the map $(U,\cO_A|U)\ra B$ factors as 
\[   (U,\cO_A|U)\xra{g} k^n\times B\ra B  \]
for some $n$ where $g$ is \'etale.   
\end{dfn}

The class $\sP$ of smooth maps between affine derived $k$-manifolds introduced in Definition~\ref{smooth} will be denoted \textit{sm}.  The $\infty$-site $(\bdAff_k,et)$ admits finite limits by Proposition~\ref{finitelimits} and the \'etale is subcanonical by Proposition~\ref{subcanonical}.  Since \'etale morphisms are smooth and smooth maps between affine derived $k$-manifolds are stable by composition, finite limits and equivalences, the pair $((\bdAff_k,et),sm)$ defines a geometry.  This enables us to define the following key geometric objects, structures on which will be explored for the remainder of this article.

\begin{dfn}
A \textit{derived $k$-smooth Artin stack} is a geometric stack on the geometry given by $((\bdAff_k,et),sm)$.
\end{dfn}

We denote by $\dSmAr_k:=\bSh(\bdAff,et;{sm})$ the $\infty$-category of derived $k$-smooth Artin stacks.  As a full subcategory of $\dSmSt_k$, it is stable under pullbacks.

Every derived $k$-smooth Artin stack has a presentation as a quotient of a derived $k$-smooth stack by a groupoid action.  This uses the characterization of geometric stacks as quotients by groupoid actions as detailed in Section 1.3.4 of \cite{TVII}.  Here we state the main result in our context.  

Recall that since the $\infty$-category of derived $k$-smooth stacks $\dSmSt_k$ admits finite limits, there exists an $\infty$-category $\textbf{Gpd}(\dSmSt_k)$ of groupoid objects in $\dSmSt_k$.  If $G$ is a groupoid object in $\dSmSt_k$, then $G$ is said to be a \textit{$n$-smooth groupoid object} if $G_0$ and $G_1$ are disjoint unions of $n$-geometric stacks and the face map $d_0:G_1\ra G_0$ is in $n\mbox{-}\sP=n\mbox{-}sm$.

One can follow the same argument used in Proposition 1.3.4.2 of \cite{TVII} to prove the following proposition with minimal modifications so we leave the details to the reader.  Let $X$ be a derived $k$-smooth stack and $n\geq 0$.  Then $X$ is $n$-geometric if and only if there exists a $(n-1)$-smooth groupoid object $G$ in the $\infty$-category of derived $k$-smooth stacks and an equivalence
\[     X\ra |G|:=\underset{n}{\colim}~G_n  \]
of derived $k$-smooth stacks.  We note that if $X$ is $n$-geometric and 
\[   f:G_0:=\coprod_iU_i\ra X  \]
is the effective epimorphism of an $n$-atlas for $X$, then the groupoid object $G$ in the $\infty$-category of derived $k$-smooth stacks is given by
\[  G_n:=G_0\times_X G_0\times_X\ldots\times_XG_0  \]
where $G_0$ on the right hand side occurs $n$ times.  It is not difficult to check that $G_0$ and $G_1$ are disjoint unions of $(n-1)$-geometric stacks and that $d_0:G_1\ra G_0$ is in $(n-1)$-$sm$.  If we realize a derived $k$-smooth Artin stack $X$ through this result, we say that $X$ is the \textit{quotient stack} of the $(n-1)$-smooth groupoid $G$.

\begin{ex}\label{gmex}
Let $G$ be an affine derived smooth group stack over $k$.  Then the classifying stack $K(G,n)$ is an example of a derived $k$-smooth $n$-Artin stack.  It is constructed as follows.  We define the derived $k$-smooth stack $\B G:=[*/G]$.  This can be described as the sheafification of the presheaf sending an affine derived $k$-manifold $A$ to the $\infty$-groupoid $\textbf{B}(G(A))$.  However, $K(G,1):=\B G$ is itself a smooth abelian group object in derived $k$-smooth stacks so by induction we define
\[      K(G,n+1):=\B K(G,n)  \]
for all $n\geq 1$.

More generally, let $C$ be an $\infty$-category and $\textbf{Gp}(C)$ the $\infty$-category of group objects in $C$.  Let $G$ be a $\textbf{Gp}(C)$-valued sheaf on $\bdAff_k$.  Then $K(G,1)$ will be the derived $k$-smooth group stack given by the sheafification of the presheaf sending $A$ to $\textbf{B}G(A)$.  The inductive definition holds as above.  

The most important example for this work is the multiplicative derived smooth group stack $\bb{G}_m$.  Let 
\[   \textbf{ShTop}:=\bFun(*,\int_{\bTop}\Sh_{\sSet})\times_{\bFun(*,\bTop)}\bTop   \]
denote the $\infty$-category of pairs $(|A|,F)$ consisting of a topological space $|A|$ together with a stack on $A$ (see Construction~\ref{constr}).  We define
\[  \bb{G}_m:(\bdAff_{\bb{R}})^\circ\ra\textbf{Gp}(\textbf{ShTop})    \]
to be the smooth group stack sending $A=(|A|,\cO_A)$ to the group object $|\cO_A^*|$ of complex valued invertible elements in $\cO_A$.  Then
\[    K(\bb{G}_m,1):(\bdAff_{\bb{R}})^\circ\ra\textbf{Gp}(\S)   \]
is the sheafification of the functor sending $A$ to $\textbf{B}|\cO_A^*|$.  Then by induction we have the object $K(\bb{G}_m,n)$ in the $\infty$-category $\bSh_{\textbf{Gp}(\S)}(\bdAff_{\bb{R}})$.
\end{ex}

We conclude this section with a review on how one moves between geometric stacks.

\begin{prop}\label{fshreikpreservesgs}
Let $((C,\tau),\sP)$ and $((D,\eta),\sQ)$ be geometries and $f:C\ra D$ a transformation of geometries.  Then the following hold.
\begin{enumerate}
\item The induced functor $f_!^a$ preserves geometric stacks.
\item The induced functor $f_!^a$ sends maps in $\sP$ to maps in $\sQ$.
\end{enumerate}
\end{prop}

\begin{proof}
This is proven in the categorical realm in \cite{TV} and the proof is similar in the $\infty$-categorical case (see also Section~2 of \cite{PY} for a proof in a close context).  
\end{proof}

It follows from Proposition~\ref{fshreikpreservesgs} that we have a functor 
\[    p:\textbf{Geom}\ra\bCatinf    \]
of $\infty$-categories sending a geometry $((C,\tau),\sP)$ to $\bSh(C,\tau;\sP)$ and a transformation of geometries $f:C\ra D$ to $p(f)=f^a_!$.  

\begin{ex}\label{exampledmtoar}
Consider the two geometries of Example~\ref{spectraldm} and Example~\ref{spectralart}.  Then the inclusion
$  i:((\bAff_k^{\,\alg},et),et)\ra((\bAff_k^{\,\alg},et),sm)   $
is a transformation of geometries and induces an inclusion
\[    i_!^a:\textbf{dDM}_k\ra\textbf{dAr}_k    \]
of geometric stacks.  Likewise, consider the two geometries of Example~\ref{analyticdm} and Proposition~\ref{analyticart}.  Then the map 
$   j:((\bAff_k^{\,\an},et),et)\ra((\bAff^{\,\an},et),sm)  $ 
is a transformation of geometries and induces an inclusion
\[   j_!^a:\textbf{dAnDM}_k\ra\textbf{dAnAr}_k    \]
of geometric stacks.  
\end{ex}

The situation represented in Example~\ref{exampledmtoar} obviously extends to the derived smooth setting described by a functor
\[   k_!^a:\textbf{dSmDM}_k\ra\textbf{dSmAr}_k   \]
induced from the transformation of geometries $k:((\bAff_k,et),et)\ra((\bAff_k,et),sm)$ since \'etale morphisms are smooth.

Let $\textbf{Aff}_k$ denote the category whose objects are open submanifolds of $k^n$ for some $n$.  We call objects in $\textbf{Aff}_k$ affine $k$-manifolds.  We endow $\textbf{Aff}_k$ with the \'etale topology generated by locally homeomorphic or locally biholomorphic maps depending on the ground field $k$.  We let $\textbf{SmSt}_k:=\bSh(\textbf{Aff}_k,et)$ denote the $\infty$-category of stacks on the site of affine $k$-manifolds endowed with the \'etale topology and 
\[  \textbf{SmAr}_k:=\bSh(\textbf{Aff}_k,et;sm)  \]
the $\infty$-category of geometric stacks on the geometry $((\textbf{Aff}_k,et),sm)$.

\begin{prop}\label{}
The map $\Spec^{\Man_k}$ induces a transformation of geometries 
\[  ((\textbf{\textup{Aff}}_k,et),sm)\ra((\bdAff_k,et),sm)  \] 
and induces a fully faithful functor
\[     \textbf{\tu{SmAr}}_k\ra\dSmAr_k  \]
between $\infty$-categories of geometric stacks.
\end{prop}

\begin{proof}
The functor $\Spec^{\Man_k}$ is left exact and preserves covering families and therefore, by Proposition~\ref{fiscontinuous}, is topologically continuous.  Therefore, the functor $\Spec^{\Man_k}$ is a transformation of geometries.  By Proposition~\ref{fshreikpreservesgs}, we have an induced functor
\[   (\Spec^{\Man_k})^a_!:\textbf{\tu{SmAr}}_k\ra\dSmAr_k  \]
between geometric stacks.
\end{proof}

The fully faithful functor $(\Spec^{\Man_k})^a_!$ admits a right adjoint 
\[   t_0:\dSmAr_k\ra\textbf{SmAr}_k  \]
called the truncation functor.

\section{Cotangent complexes}\label{candt}

In this section we define the cotangent complex of a derived $k$-manifold using the tangent bundle construction of Lurie \cite{L2} (see also \cite{Sc}).  This has been applied to the complex analytic setting in \cite{Po3}.  More generally, we define the relative cotangent complex in the derived smooth setting and describe its relationship to the algebraic cotangent complex.  The algebraic description will prove useful for defining differential forms in the next section.

Given a locally presentable $\infty$-category $C$, we will generically denote by 
\[   \Omega^\infty:\bSp(C)\ra C   \]
the functor given by evaluation at the $0$-sphere, ie. given a spectrum object $F:\S^{\tu{fin}}_*\ra C$ where $\S^{\tu{fin}}_*$ is the $\infty$-category of pointed finite spaces, then $\Omega^\infty(F):=F(S^0)$ where $S^0$ is the $0$-sphere.  By Proposition~1.4.4.4 of \cite{L2}, this functor admits a left adjoint which we denote by $\Sigma^\infty_+$.

Consider the functor 
\[    \Str^\loc_{\Man_k}(A)_{\cO_A//-}:\Str^\loc_{\Man_k}(A)_{\cO_A/}\ra\bCatinf   \]
sending a $k$-manifold structure $\cO'_A$ to $\Str^\loc_{\Man_k}(A)_{\cO_A//\cO'_A}$.  The cofibered $\infty$-category associated to this functor will be denoted by
\[    \pi:\int_{\cO_A}\Str^\loc_{\Man_k}(A)_{\cO_A//-}\ra\Str^\loc_{\Man_k}(A)_{\cO_A/}.  \]
The $\infty$-category $\Str^\loc_{\Man_k}(A)_{\cO_A//\cO'_A}$ is locally presentable by combining Corollary~1.16 of \cite{Po1} and Proposition~5.5.3.11 of \cite{L1} and thus $\pi$ is a bifibered $\infty$-category.  We denote the associated straightening functor of the fibration by
\[     S(\pi):(\Str^\loc_{\Man_k}(A)_{\cO_A/})^\circ\ra\bCatinf.  \]

The functor $S(\pi)$ factors through the $\infty$ category of locally presentable $\infty$-categories and so we can construct a functor
\[      i\circ\varsigma\circ S(\pi):(\Str^\loc_{\Man_k}(A)_{\cO_A/})^\circ\ra\bCatinf    \]
where $\varsigma:\bCatinf^{\lp}\ra\bCatinf^{\lp,\perp}$ is the functor introduced in Section~\ref{dsas} sending a locally presentable $\infty$-category $C$ to the stable, locally presentable $\infty$-category of spectrum objects in $C$ and $i:\bCatinf^{\lp,\perp}\ra\bCatinf$ is the inclusion.  

The fibered category 
\[    t:T_{\cO_A}=\int_{\cO_A} (i\circ\varsigma\circ S(\pi))\ra\Str^\loc_{\Man_k}(A)_{\cO_A/} \]
associated to this composition is called the \textit{tangent bundle} to the $\infty$-category $\Str^\loc_{\Man_k}(A)_{\cO_A/}$.  The objects of $T_{\cO_A}$ are pairs $(f,P)$ consisting of a local morphism of $k$-manifold structures $f:\cO_A\ra\cO_A'$ together with an object $P$ in $\bSp(\Str^\loc_{\Man_k}(A)_{\cO_A//\cO'_A})$.  By Lemma~7.3.3.9 of \cite{L2}, the forgetful functor $\Str^\loc_{\Man_k}(A)_{\cO_A//\cO'_A}\ra\Str^\loc_{\Man_k}(A)_{/\cO'_A}$ induces an equivalence
\[     \bSp(\Str^\loc_{\Man_k}(A)_{\cO_A//\cO'_A})\xras\bSp(\Str^\loc_{\Man_k}(A)_{/\cO'_A})   \]
of $\infty$-categories.  Therefore, we can consider $P$ as a $\cO'_A$-module.  The tangent bundle $T_{\cO_A}$ is locally presentable.

We have a natural morphism
\begin{diagram}
T_{\cO_A} & &\rTo^{F}  & & \bFun(\Delta^1,\Str^\loc_{\Man_k}(A)_{\cO_A/}) \\
            &\rdTo^{t}  & &  \ldTo_{c} &\\
            &&  \Str^\loc_{\Man_k}(A)_{\cO_A/} &&
\end{diagram}   
of fibered $\infty$-categories where the fiber of $F$ over a local $k$-manifold structure $\cO'_A$ can be identified with the functor
\[  F_{\cO'_A}:(T_{\cO_A})_{\cO'_A}\ra\Str_{\Man_k}^\loc(A)_{\cO_A//\cO'_A}  \] 
sending a pair $(f:\cO_A\ra\cO'_A,P)$ to $\Omega^\infty(P)$ where
\[  \Omega^\infty:\bSp(\Str_{\Man_k}^\loc(A)_{\cO_A//\cO'_A})\ra\Str_{\Man_k}^\loc(A)_{\cO_A//\cO'_A}  \]
and $c$ is the codomain functor induced from $\{1\}\hookrightarrow\Delta^1$.  We also use the notation
\begin{diagram}
&& \cO_A &&\\
 &\ldTo & &\rdTo^f &\\
\cO'_A\oplus P  && \rTo && \cO'_A
\end{diagram}  
for the object $\Omega^\infty(P)$ where $\cO'_A\oplus P$ is thought of as being a local $k$-manifold structure infinitesimally close to $\cO_A'$.  

Composing $F$ with the domain functor 
\[  d:\bFun(\Delta^1,\Str^\loc_{\Man_k}(A)_{\cO_A/})\ra\Str^\loc_{\Man_k}(A)_{\cO_A/}  \] 
induced by $\{0\}\hookrightarrow\Delta^1$, we obtain a composition functor
\[    (-\oplus_{\cO_A}-):=d\circ F:T_{\cO_A}\ra\Str^\loc_{\Man_k}(A)_{\cO_A/}   \]
whose image of an object $(f:\cO_A\ra\cO'_A,P)$ we denote by $\cO'_A\oplus_{\cO_A} P$.  

\begin{lem}\label{lametal}
The functor $(-\oplus_{\cO_A}-)$ admits a left adjoint.
\end{lem}

\begin{proof}
The functor $(-\oplus_{\cO_A}-)$ is accessible and preserves limits and therefore, by Proposition~5.5.2.9 of \cite{L1}, admits a left adjoint.
\end{proof}

The left adjoint supplied by Lemma~\ref{lametal} is called the \textit{relative cotangent complex functor} (on $\Str^\loc_{\Man_k}(A)_{\cO_A/}$) and is denoted $L$.  

The relative cotangent complex of a morphism $f:(|A|,\cO_A)\ra(|B|,\cO_B)$ of derived $k$-manifolds will be defined as the object given by applying the relative cotangent complex functor to the morphism $f^\sharp:f^*\cO_B\ra\cO_A$ of local $k$-manifold structures on $|A|$.

\begin{dfn}\label{cc}
Let $A=(|A|,\cO_A)$ and $B=(|B|,\cO_B)$ be derived $k$-manifolds and $f:A\ra B$ a morphism.  The \textit{relative cotangent complex} of $f$ is given by $\bb{L}_f:=L(f^\sharp)$.  
\end{dfn}

The $\cO_A$-module $\bb{L}_f$ is also denoted $\bb{L}_{A/B}$ and called the cotangent complex of $A$ over $B$ if the map $f$ is understood.  If $B$ is a final object, then $\bb{L}_{A/B}\simeq\bb{L}_A$.

By definition, we have the interpretation that $\bb{L}_f=\Sigma_+^\infty(\id_{\cO_A}\circ f^\sharp)$ where 
\[  \Sigma_+^\infty:\Str_{\Man_k}^\loc(A)_{f^*\cO_B//\cO_A}\ra\bMod(\cO_A)   \]
is the suspension spectrum functor precomposed with the forgetful functor 
\[  \Str_{\Man_k}^\loc(A)_{f^*\cO_B//\cO_A}\ra\Str_{\Man_k}^\loc(A)_{/\cO_A}  \] 
(which is an equivalence of $\infty$-categories by Lemma~7.3.3.9 of \cite{L2}).  

Let 
\[  \tu{Der}_{\cO_A}(\cO'_A,P):=\Map_{\Str^\loc_{\Man_k}(A)_{\cO_A//\cO'_A}}(\cO'_A,\cO'_A\oplus_{\cO_A} P)  \] 
be the space of morphisms thought of as $k$-smooth $\cO_A$-derivations of $\cO'_A$ into $P$.  Then there exists an equivalence
\[    \Map_{\bMod(\cO_A)}(\bb{L}_{A/B},P)\ra\tu{Der}_{\cO_A}(\cO'_A,P)   \]
of spaces.  

Some properties of the cotangent complex of a morphism of derived $k$-manifolds are collected in the result below.  For this result we will use the (abuse of) notation $f^*$ for the composite map
\[    \bMod(\cO_B)\xra{(f^\sharp)^*}\bMod(f^*\cO_B)\xra{f^\sharp_!}\bMod(\cO_A)  \]
of $\infty$-categories.

\begin{prop}\label{cca}
Let $A=(|A|,\cO_A)$, $B=(|B|,\cO_B)$, $C=(|C|,\cO_C)$ and $D=(|D|,\cO_D)$ be derived $k$-manifolds.
\begin{enumerate}
\item Let $f:A\ra B$ be a morphism.  Then there exists an equivalence
\[   \bb{L}_f\ra\Sigma^\infty(\cO_A\coprod_{f^*\cO_B}\cO_A ) \]
of $\cO_A$-modules where $\Sigma^\infty:\tu{\Str}^\loc_{\Man_k}(A)_{\cO_A//\cO_A}\ra\bMod(\cO_A)$.
\item For any commutative diagram
\begin{diagram}
 & & B & &  \\
 &\ruTo^f & &\rdTo^g &\\
 A& &\rTo & & C
\end{diagram} 
the induced sequence
\[    f^*\bb{L}_{g}\ra\bb{L}_{g\circ f}\ra\bb{L}_{f}   \]
is a cofiber sequence.  In particular, there exists an equivalence
\[     \bb{L}_f\ra\tu{cofib}(f^*\bb{L}_A\ra\bb{L}_B)  \]
of $\cO_A$-modules.
\item For any pullback diagram
\begin{diagram}
D &\rTo^{} &C\\
\dTo_{f}  &        &\dTo_{}\\
B  &\rTo^{}       &A
\end{diagram}
in the $\infty$-category of derived $k$-manifolds, the induced morphism 
\[  f^*(\bb{L}_{A/B})\ra\bb{L}_{D/B}   \]
is an equivalence of $\cO_D$-modules.  
\end{enumerate}
\end{prop}

\begin{proof}
By definition we have an equivalence
\[  \Map_{\bMod(\cO_A)}(\bb{L}_{A/B},P)\xras\Map_{\Str^\loc_{\Man_k}(A)_{f^*\cO_B//\cO_A}}(\cO_A,\Omega^\infty(P))   \]
of spaces.  Now consider the adjunction
\[     -\otimes_{f^*\cO_B}\cO_A:\Str^\loc_{\Man_k}(A)_{f^*\cO_B//\cO_A}\rightleftarrows \Str^\loc_{\Man_k}(A)_{\cO_A//\cO_A}:F   \]
where $F$ is the forgetful functor.  We have a chain of equivalences
\begin{align*}
\Map_{\Str^\loc_{\Man_k}(A)_{f^*\cO_B//\cO_A}}(\cO_A,\Omega^\infty(P)) &\xras\Map_{\Str^\loc_{\Man_k}(A)_{f^*\cO_B//\cO_A}}(\cO_A,F\circ\Omega^\infty(P))  \\
             &\xras\Map_{\Str^\loc_{\Man_k}(A)_{\cO_A//\cO_A}}(\cO_A\otimes_{f^*\cO_B}\cO_A,\Omega^\infty(P))  
\end{align*}
By Proposition~\ref{} of \cite{L2}, the tensor product $\cO_A\otimes_{f^*\cO_B}\cO_A$ corresponds to cofibered product $\cO_A\coprod_{f^*\cO_B}\cO_A$.  Therefore, there exists an equivalence
\[  \Map_{\Str^\loc_{\Man_k}(A)_{\cO_A//\cO_A}}(\cO_A\otimes_{f^*\cO_B}\cO_A,\Omega^\infty(P))\xras\Map_{\bMod(\cO_A)}(\Sigma^\infty(\cO_A\coprod_{f^*\cO_B}\cO_A),P)   \]
which proves 1.

The second statement is an application of Corollary~7.3.3.6 of \cite{L2} which in our context states that there exists a canonical cofiber sequence
\[    f^\sharp_!L(f^*\circ g^\sharp)\ra L((g\circ f)^\sharp)\ra L(f^\sharp).  \]
The result then follows from the equivalence $f^\sharp_!L(f^*\circ g^\sharp)\simeq f^\sharp_!\circ (f^\sharp)^*L(g^\sharp)$.  The complementary statement follows from setting $C$ to be the final object.  The final statement is an application of Proposition~7.3.3.7 of \textit{loc. cit.}.
\end{proof}

To confirm to the reader that Definition~\ref{cc} is reasonable, we show that it reduces to the standard notion of cotangent space when $A$ is a (non-derived) smooth $k$-manifold.  Recall that the sheaf of smooth one-forms on a $k$-manifold $M=(|M|,\cO_M^\disc)$ can be realized as the $\cO_M^\disc$-module representing (K\"ahler) derivations where the category of $\cO_M^\disc$-modules is given by the category of abelian group objects in category of locally $k$-smooth ringed spaces over $\cO_M^\disc$.

\begin{lem}\label{usual}
Let $M=(|M|,\cO_M^\disc)$ be a $k$-manifold and $A=\Spec^{\Man_k}M$.  Then $\pi_0\bb{L}_{A}$ is equivalent to the usual cotangent sheaf of smooth forms on $M$.
\end{lem}

\begin{proof}
The inclusion $\Str_{\Man_k}^\loc(A)_{/\cO_A}\ra\bFun(\Man_k,\bSh(A))_{/\cO_A}$ induces a fully faithful functor
\[  j:\bMod(\cO_A)\ra\bSp(\bFun(\Man_k,\bSh(A))_{/\cO_A})\xras(\bSh_{\bSp}(A)^{\Man_k})_{/\cO_A}  \]
between $\infty$-categories.  By Proposition~1.9 of \cite{Po2}, there exists a t-structure on $\bMod(\cO_A)$ given by $(\bMod(\cO_A)_{\leq 0},\bMod(\cO_A)_{\geq 0})$ where a $\cO_A$-module $P$ belongs to $\bMod(\cO_A)_{\leq 0}$ if and only if $j(P)$ belongs to $((\bSh_{\bSp}(A)^{\Man_k})_{/\cO_A})_{\leq 0}$ and likewise for $\bMod(\cO_A)_{\geq 0}$.  Further, there exists an equivalence
\[    \pi_0:\bMod(\cO_A)^\heartsuit\xras\textbf{Ab}(\Str_{\Man_k}^{\loc,0}(A)_{/\pi_0(\cO_A)})  \]
of $\infty$-categories where the left hand side denotes the heart of the t-structure and $\Str_{\Man_k}^{\loc,0}(A)$ denotes the subcategory of $\Str_{\Man_k}^\loc(A)$ spanned by objects whose $k$-manifold structure is $0$-truncated.  This $\infty$-category, the $\infty$-category of \textit{discrete $k$-manifold structured spaces} on $|A|$, can be identified with the (nerve of the) category of abelian group objects in locally $k$-smooth ringed spaces over $\cO_M^\disc\simeq\pi_0(\cO_A)$.  We denote this $\infty$-category by $\bMod(\cO_M)$.

Let $P$ be an object in $\bMod(\pi_0(\cO_A))$.  Then there exists a chain of equivalences
\[    \Map_{\bMod(\cO_A)}(\pi_0(\bb{L}_A),P)\xras\Map_{\bMod(\cO_A)}(\bb{L}_A,P)\xras\Map_{\Str^\loc_{\Man_k}(A)}(\cO_A,\cO_A\oplus P)\xras  \]
\[   \Map_{\bMod(\cO_A)^\heartsuit}(\pi_0(\cO_A),\pi_0(\cO_A\oplus P))\xras\Map_{\bMod(\cO_M)}(\cO_M^\disc,\cO_M^\disc\oplus P))     \]
where in the last equivalence we have identified $P$ with its corresponding $\cO_M^\disc$-module.  The final discrete space is the space of derivations for the $k$-manifold $M$ and is thus representable by the module $\Omega_{M}$ of K\"ahler differentials.  This module is a model for the sheaf of $k$-smooth forms on $|M|$.
\end{proof}

When $f:A=(|A|,\cO_A)\ra B=(|B|,\cO_B)$ is a smooth morphism of affine derived $k$-manifolds, then $\bb{L}_f$ is concentrated in degree zero.  Now let $f:A\ra B$ be an arbitrary morphism of derived $k$-manifolds.  The relationship between the relative cotangent complex $\bb{L}_f$ and the relative algebraic cotangent complex $\bb{L}_{f^\alg}$ is highlighted in the following result (for the algebraic setting see \cite{LXIV}).  Consider the following diagram
\begin{diagram}
\Str^\loc_{\Man_k}(A)_{f^*\cO_B/} &\rTo^{} &\Str^\loc_{\cP_k}(A)_{f^*\cO_B^\alg/}\\
\dTo_{\Sigma^\infty}  &        &\dTo_{\Sigma^\infty}\\
\bMod(\cO_A)  &\rTo^{\psi}     &\bMod(\cO_A^\alg)
\end{diagram}
where the bottom horizontal map $\psi$ is the equivalence of Proposition~\ref{modequivalence}.  This diagram is in general, not commutative, so the functor $\psi$ does not send $\bb{L}_f$ to $\bb{L}_{f^\alg}$ in general.  However there is a general situation where this is the case.

\begin{prop}
Let $f:A\ra B$ be a closed immersion between derived $k$-manifolds.  Then the natural map
\[    \bb{L}_{f^\alg}\ra\psi(\bb{L}_{f})   \]
is an equivalence.
\end{prop}

\begin{proof}
We refer the reader to Corollary~5.2.7 of \cite{Sp1} (or to Corollary~1.31 of \cite{Po3} where similar notation is utilized).
\end{proof}

Using this result, many properties of the relative cotangent complex can be transported to the algebraic setting where existing results abound.  One example is the following.

\begin{prop}
Let $f:A\ra B$ be a closed immersion between derived $k$-manifolds such that
\[    \pi_0f:\pi_0A\ra\pi_0B  \]
is an equivalence of sheaves.  Then $f$ is an equivalence if and only if $\bb{L}_{f}$ vanishes.  
\end{prop}

\begin{proof}
Assume $\bb{L}_f$ vanishes.  It is enough to check the condition on stalks.  In this case, for a point $p:*\ra|A|$, the map $p^*(\pi_0f)$ induces an isomophism $\pi_0(p^*(f^*\cO^\alg_B))\ra\pi_0(p^*\cO^\alg_A)$ of commutative $k$-algebras.  Since $\bb{L}_{f^\alg}$ vanishes, by Corollary~7.4.3.4 and Proposition~7.1.4.11 of \cite{L2}, the map $p^*(f^*\cO^\alg_B))\ra p^*\cO^\alg_A$ is an equivalence of coconnective commutative dg-algebras over $k$.  Thus the map $f^\alg$ is an equivalence and since the algebraic model functor is conservative, the map $f$ is an equivalence.  Following this same argument in the opposite direction we obtain the result.
\end{proof}

\section{Shifted preplectic stacks}\label{spds}

In this section we define the notion of $n$-shifted $p$-forms and $n$-shifted $p$-preplectic forms on a derived $k$-smooth stack.  The algebraic theory was introduced in \cite{PTVV} (see also \cite{BZN} and \cite{S1eq} for further background).  A more general approach which can be utilized in other contexts, and which we largely follow here, is contained in \cite{CPTVV}.  

We begin by giving a definition of the space of $n$-shifted $p$-forms and $n$-shifted closed $p$-forms on an affine derived $k$-smooth stack.  We obtain the corresponding definitions on derived $k$-smooth stacks by gluing the structures on affine objects.  

Let $\sM$ be a combinatorial symmetric monoidal model category enriched over $\dg_k$ and satisfying the monoid axiom.  As a result, the model category $\sM$ is tensored and cotensored over $\dg_k$ and we denote by
\[  \epsilon-\sM^{gr}:=\Comod_{k[t,t^{-1}]\otimes_k k[\epsilon]}(\sM)    \]
the symmetric monoidal model category of comodules over the commutative and cocommutative Hopf dg-algebra $k[t,t^{-1}]\otimes_k k[\epsilon]$.  The symmetric monoidal model structure is defined through the forgetful functor $\epsilon-\sM^{gr}\ra\sM^{gr}:=\sM^{\bb{Z}}$.  Given a graded object $F=\oplus_pF(p)$ in $\sM$, we will refer to $F(p)$ as the \textit{weight} $p$-piece of $F$.  

Let $A=(|A|,\cO_A)$ be an affine derived $k$-manifold.  The symmetric monoidal model category $\Sh_{\dg_k}(A)$ of $\dg_k$-valued sheaves on $|A|$ is tensored and cotensored over the category $\dg_k$ of dg-modules over $k$ in the obvious way.  We denote by
\[  \epsilon-\bMod(\cO_A^\alg)^{gr}:=L(\epsilon-\Mod(\cO_A^\alg)^{gr})  \]
the $\infty$-category of graded mixed $\cO_A^\alg$-modules and
\[  \epsilon-\bSh_{\bdg_k}(A)^{gr}:=L(\epsilon-\Sh_{\dg_k}(A)^{gr})    \]
the $\infty$-category of graded mixed sheaves of dg-modules on $|A|$.  We define the $\infty$-category of algebras of graded mixed $\cO_A^\alg$-modules by
\[   \epsilon-\bCAlg(\cO_A^\alg)^{gr}:=L(\CMon(\epsilon-\Mod(\cO_A^\alg)^{gr}))  \]
and the $\infty$-category of algebras of graded mixed sheaves of dg-modules on $|A|$ by
\[   \epsilon-\bSh_{\bcdga_k}(A)^{gr}:=L(\CMon(\epsilon-\Sh_{\dg_k}(A)^{gr})).   \]

Note that if we define the $\infty$-category of graded mixed complexes by the localization 
\[    \epsilon-\bdg_k^{gr}:=L(\epsilon-\dg_k^{gr})   \]
of the symmetric monoidal model category of graded mixed complexes then we obtain a chain of equivalences
\[    \epsilon-\bSh_{\bdg_k}(A)^{gr}\xras L(\Sh_{\epsilon-\dg_k^{gr}}(A))\xras\bSh_{\epsilon-\bdg_k^{gr}}(A)   \]
of $\infty$-categories.  Likewise, we define the $\infty$-category of graded mixed commutative dg-algebras by 
\[      \epsilon-\bcdga_k^{gr}:=L(\CMon(\epsilon-\dg_k^{gr}))     \]
and there exists a chain of equivalences
\[    \epsilon-\bSh_{\bcdga_k}(A)^{gr}\xras L(\CMon(\Sh_{\epsilon-\dg_k^{gr}}(A)))\xras L(\Sh_{\CMon(\epsilon-\dg_k^{gr})}(A))\xras\bSh_{\epsilon-\bcdga_k^{gr}}(A)   \]
of $\infty$-categories.  

Consider the chain of equivalences
\[  \Map(\bb{L}_A,\bb{L}_A)\simeq\Map((\bb{L}_A)^\alg,(\bb{L}_A)^\alg)\simeq\Map(\cO_A,\cO_A\oplus\bb{L}_A)\simeq\Map(\cO_A^\alg,(\cO_A\oplus\bb{L}_A)^\alg)  \]
\[   \simeq\Map(\cO_A^\alg,\cO_A^\alg\oplus(\bb{L}_A)^\alg)\simeq\Map(\cO_A^\alg,(\bb{L}_A)^\alg)  \]
of mapping spaces.  The identity map in $\Map(\bb{L}_A,\bb{L}_A)$ induces the universal derivation $\delta$ in $\Map(\cO_A^\alg,(\bb{L}_A)^\alg)$ from this chain of equivalences.  Identifying $\delta$ with $\epsilon$, then
\[    \delta:\cO^\alg_A\ra (\bb{L}_A)^\alg   \]
is an element in $\epsilon-\bMod(\cO_A^\alg)^{gr}$.  

We apply the mixed graded symmetric algebra functor
\[    \epsilon-\Sym^{gr}:\epsilon-\bMod(\cO_A^\alg)^{gr}\ra\epsilon-\bCAlg(\cO_A^\alg)^{gr}   \]
to the map $\delta$ followed by the forgetful functor
\[  g:\epsilon-\bCAlg(\cO_A^\alg)^{gr}\ra\epsilon-\bSh_{\bcdga}(A)^{gr}  \]
induced from the forgetful functor $f:\epsilon-\Mod(\cO_A^\alg)^{gr}\ra\epsilon-\Sh_{\dg}(A)^{gr}$ to obtain the de~Rham algebra of $A$.

\begin{dfn}\label{derhamalg}
Let $A=(|A|,\cO_A)$ be an affine derived $k$-manifold.  The \textit{de Rham algebra} of $A$ is given by
\[    \cD\cR(A):=g\circ\epsilon-\Sym^{gr}(\delta)  \]  
and is a sheaf of graded mixed dg-algebras over $k$ on $|A|$.
\end{dfn}

Note that the underlying graded object of $\cD\cR(A)$ is $\Sym(\bb{L}_A[1])$ with weight $p$-piece given by $\cD\cR(p)=\wedge^p\bb{L}_A[p]$.  Here, for any $U\subseteq |A|$, $\wedge^p_U$ is the levelwise derived $p$th exterior power of the dg-module $\bb{L}_{A}(U)$ over $k$.  The extra mixed structure 
\[   \delta:\cD\cR(p)\ra\cD\cR(p+1)  \]
will be called the \textit{de~Rham differential}.

This construction of the de Rham algebra is functorial in $A$ as follows.  Consider the functor
\[    \Sh_{\CMon(\epsilon-\dg_k^{gr})}:\Top^\circ\ra\PC(\sSet)   \]
sending $|A|$ to $\Sh_{\CMon(\epsilon-\dg_k^{gr})}(A)$.  We define 
\[    \textbf{gmTop}_k:=\bFun(*,\int_{\Top}\Sh_{\CMon(\epsilon-\dg_k^{gr})})\times_{\bFun(*,\bTop)}\bTop   \]
and consider the functor 
\[     \cD\cR:(\bdAff_k)^\circ\ra\textbf{gmTop}_k    \]
sending $A$ to $\cD\cR(A)$.  

Using the de Rham algebra in Definition~\ref{derhamalg} we can define the spaces of shifted (closed) forms on an affine derived $k$-manifold.

\begin{dfn}\label{pforms}
Let $A=(|A|,\cO_A)$ be a affine derived $k$-manifold, $p\geq 0$ and $n\in\bb{Z}$.  Then
\[    \cF^p(A,n):=\Map_{\bSh_{\bdg_k}(A)}(1[-n],\cD\cR(A)(p))   \] 
is called the space of \textit{$n$-shifted $p$-forms} on the derived $k$-smooth stack $\Spec A$.  The space 
\[   \cF^{p,\cl}(A,n):=\Map_{\epsilon-\bSh_{\bdg_k}(A)^{gr}}(1(p)[-p-n],\cD\cR(A))   \] 
is called the space of \textit{$n$-shifted closed $p$-forms} on the derived $k$-smooth stack $\Spec A$.  
\end{dfn}

We will denote by
\[   \cF^p(-,n):(\bdAff_k)^\circ\ra\textbf{S}  \] 
the presheaf which sends $A$ to the space $\cF^p(A,n)$ of $n$-shifted $p$-forms on $A$ and by
\[   \cF^{p,\cl}(-,n):(\bdAff_k)^\circ\ra\textbf{S}  \] 
the presheaf which sends $A$ to the space $\cF^{p,\cl}(A,n)$ of $n$-shifted closed $p$-forms on $A$.

\begin{prop}\label{formsarestacks}
The presheaves $\cF^p(-,n)$ and $\cF^{p,\cl}(-,n)$ are sheaves on the $\infty$-site of affine derived $k$-manifolds with respect to the \'etale topology.
\end{prop}

\begin{proof}
By definition, it will suffice to prove that the functor
\[   \cD\cR(p):(\bdAff_k)^\circ\ra\textbf{gmTop}_k   \]
sending $A$ to $\cD\cR(A)(p)$ satisfies \'etale descent.  It suffices to consider this functor taking values in the $\infty$-category $\textbf{dgTop}_k$, or simply $\int_{\bdAff_k}\cM^\alg$ where $\cM^\alg:=\cM\circ-^\alg$, whose objects are pairs $(A,M_A)$ consisting of an affine derived $k$-manifold together with a module $M_A$ in $\bMod(\cO_A^\alg)$.  The result now follows from Proposition~\ref{modisasheaf}.
\end{proof}

In view of Proposition~\ref{formsarestacks} we make the following definition. 

\begin{dfn}
Let $X$ be a derived $k$-smooth stack.  Then
\[   \cF^p(X,n):=\Map_{\dSmSt_k}(X,\cF^p(-,n))  \] 
is called the space of \textit{$n$-shifted $p$-forms} on $X$ and 
\[  \cF^{p,\cl}(X,n):=\Map_{\dSmSt_k}(X,\cF^{p,\cl}(-,n))   \] 
the space of \textit{$n$-shifted closed $p$-forms} on $X$.
\end{dfn}

Consider the functor 
\[  |-|:\epsilon-\bSh_{\bdg_k}(A)^{gr}\ra\epsilon-\bdg_k^{gr}  \] 
sending a mixed graded object $F$ to $|F|:=\Mor(1,F)$.  Here $\Mor$ is the natural morphism object using the $\epsilon-\bdg_k^{gr}$ enrichment of $\epsilon-\bSh_{\bdg_k}(A)^{gr}$.  Following \cite{CPTVV} we call this the realization functor.  Composition of $|-|$ with the functor 
\[  \prod_{p\leq i\leq q}(-)(i):\epsilon-\bdg_k^{gr}\ra\bdg_k  \] 
sending a graded mixed commutative dg-algebra $\cA$ to the dg-module $\prod_{p\leq i\leq q}\cA(i)$ endowed with the total differential (consisting of a sum of the mixed differential and the internal differential) will be denoted
\[   \tu{tot}_{[p,q]}:\epsilon-\bSh_{\bcdga_k}(A)^{gr}\ra\bdg_k.   \]
Likewise, $\tu{tot}_{[p,q)}$ (resp. $\tu{tot}_{(p,q]}$ and $\tu{tot}_{(p,q)}$) will denote the functors on the half closed intervals where the product above runs over $p\leq i<q$ (resp. $p<i\leq q$ and $p<i<q$).

Let $A$ be an affine derived $k$-manifold.  The derived $k$-smooth prestack $\DR_{\geq p}[n]$ sending $A$ to
\[  \DR_{\geq p}(A)[n]:=\Map_{\bdg_k}(k[-n],\tu{tot}_{[p,\infty)}(\cD\cR(A)))  \]
with total differential consisting of the de~Rham differential $\delta$ and internal differential $\tu{d}$ is a derived $k$-smooth stack from Proposition~\ref{formsarestacks}.  

\begin{lem}\label{fpcldr}
There exists an equivalence
\[    \cF^{p,\cl}(-,n)\ra\DR_{\geq p}[n]   \]
of derived $k$-smooth stacks.
\end{lem}

\begin{proof}
The realization functor is explicitly given as follows.  Using the Quillen adjunction
\[   -\otimes 1:\dg_k\rightleftarrows\epsilon-\Sh_{\dg_k}(A)^{gr} :\uHom(1,-) \]
of model categories, the realization functor is simply $\bb{R}\uHom(1,-)$.  Therefore, there exists an equivalence
\[    \cF^{p,\cl}(A,n)\xras\Map_{\epsilon-\bdg_k^{gr}}(k(p)[-p-n],|\cD\cR(A)|)   \]
of spaces.  The mapping space on the right hand side is then computed in Section~1.4 of \cite{CPTVV} to coincide with $\DR_{\geq p}(A)[n]$.  By functoriality, the result follows.
\end{proof}

The truncation of the complex $\DR_{\geq p}(A)[n]$ at $p$ induces a map 
\[     k:\cF^{p,\cl}(A,n)\ra\cF^p(A,n)      \]
of spaces.  The map $\pi_0k$ is surjective but not injective and so a number of different $n$-shifted closed $p$-forms may have the same underlying $n$-shifted $p$-form.  

Let $X$ be a derived $k$-smooth stack.  Using Lemma~\ref{fpcldr} one may define the space of $n$-shifted closed $p$-forms on $X$ by
\[   \DR_{\geq p}(X):=\Map_{\dSmSt_k}(X,\DR_{\geq p})   \]
and an inclusion
\[    i:\DR_{\geq q}(X)\ra\DR_{\geq p}(X)   \]
for any $q>p$.

\begin{ex}\label{preplecman}
To confirm to the reader that our notion of $n$-shifted $p$-preplectic derived $k$-smooth Artin stack is reasonable, consider the simple example when $M$ is a smooth manifold over $k$.  Let 
\[  A=\Spec^{\Man_k}(M).  \]
By Lemma~\ref{usual}, the cotangent complex $\bb{L}_A$ is simply the sheaf $\Omega_M^1$ of one-forms on $M$ and we have an equivalence
\[   \cF^p(A,n)\ra |\Gamma(A,\Omega_A^p)[n]|  \]
of spaces.  Furthermore, a short calculation determines the space of $n$-shifted closed $p$-forms on $A$ to be
\[   \cF^{p,\cl}(A,n)^{cl}\simeq\Map(k,\Omega _M^{\geq p}[n])  \]
where $\Omega_M^{\geq p}=\{\Omega_M^p\ra\Omega_M^{p+1}\ra\ldots\}$ is the truncated de Rham complex with $\Omega_M^p$ in degree zero.  

Looking at the homotopy groups of the space of $n$-shifted closed $p$-forms on $A$ we have that $\pi_n(\cF^{p,\cl}(A,n))$ is equivalent to the usual set $\Gamma(M,\Omega_M^{p,\cl})$ of closed $p$-forms on $A$.  In particular, a $0$-shifted closed $p$-form on $A$ in our sense is a usual closed $p$-form on $M$.  For $0\leq i\leq n-1$, $\pi_i(\cF^p(A,n))$ is equivalent to the de Rham cohomology group $H_{dR}^{p+n-i}(M)$ .  The space is empty for $n<0$.  
\end{ex}

\section{The derived Weil-Kostant integrality theorem}\label{dwkit}

In this section we state and prove the Weil-Kostant theorem in the derived context.  For the related theorem in the non-derived setting see \cite{Ga} which extends work of \cite{Br}.  We first define what it means for a $n$-shifted complex closed $p$-form on a derived smooth Artin stack to be integral.  

To simplify the presentation, we will call a (affine) derived $\bb{R}$-smooth manifold a \textit{(affine) derived manifold} and a (affine) derived $\bb{C}$-smooth manifold a \textit{(affine) derived complex manifold}.  Similarly, a derived $\bb{R}$-smooth (Artin) stack will be called a \textit{derived smooth (Artin) stack} and a derived $\bb{C}$-smooth (Artin) stack a \textit{derived complex (Artin) stack}.

Let $A=(|A|,\cO_A)$ be an affine derived manifold, $p\geq 0$ and $n\in\bb{Z}$.  Then we denote by 
\[    \cF^p_{\bb{C}}(A,n):=\Map_{\bSh_{\bdg_{\bb{C}}}(A)}(1[-n],(\cD\cR(A)\otimes\bb{C})(p))   \] 
the space of \textit{complex valued $n$-shifted $p$-forms} on the derived smooth stack $\Spec A$.  Similarly, the space 
\[   \cF^{p,\cl}_{\bb{C}}(A,n):=\Map_{\epsilon-\bSh_{\bdg_{\bb{C}}}(A)^{gr}}(1(p)[-p-n],\cD\cR(A)\otimes\bb{C})   \] 
is called the space of \textit{complex valued $n$-shifted closed $p$-forms} on $\Spec A$.  

Let $X$ be derived smooth stack.  All the results of Section~\ref{spds} carry over and therefore we have the stacks $\cF^p_{\bb{C}}(-,n)$ and $\cF^{p,\cl}_{\bb{C}}(-,n)$ on the $\infty$-site of affine derived manifolds with respect to the \'etale topology and thus the space
\[   \cF^p_{\bb{C}}(X,n):=\Map_{\dSmSt_{\bb{R}}}(X,\cF_{\bb{C}}^p(-,n))  \] 
of \textit{complex valued $n$-shifted $p$-forms} on $X$ and the space
\[  \cF^{p,\cl}_{\bb{C}}(X,n):=\Map_{\dSmSt_{\bb{R}}}(X,\cF^{p,\cl}_{\bb{C}}(-,n))   \] 
of \textit{complex valued $n$-shifted closed $p$-forms} on $X$.  

\begin{dfn}
A derived smooth Artin stack endowed with an $n$-shifted complex closed $(p+1)$-form will be called a \textit{$n$-shifted $p$-preplectic derived smooth Artin stack}. 
\end{dfn}
 
For any $p\leq q$, consider the truncated $n$-shifted complexified de Rham stack
\[     \DR_{[p,q]}^{\bb{C}}[n]:(\bdAff_\bb{R})^\circ\ra\textbf{S}   \]
sending $A$ to 
\[  \DR_{[p,q]}^{\bb{C}}(A)[n]:=\Map_{\bdg_\bb{C}}(\bb{C}[-n],\tu{tot}_{[p,q]}(\cD\cR(A))\otimes\bb{C}).  \] 
Similarly, we write $\DR_{\geq p}^{\bb{C}}[n]$ for the stack indexed by the half closed interval $[p,\infty)$ and by $\DR_{<p}^{\bb{C}}[n]$ for the stack indexed by the open interval $(\infty,p)$.  

A complex valued closed 2-form on a smooth manifold has integral periods if it lies in an integral cohomology class.  The generalization to complex valued closed 2-forms on a derived smooth stack is given by the fiber product of the inclusion $i:\DR_{\geq 2}^{\bb{C}}(X)[2]\ra\DR^{\bb{C}}(X)[2]$ along the constant functions map $2\pi i$.  More generally, we have the following~:

\begin{dfn}
Let $X$ be a derived smooth stack.  The space of \textit{integral} $n$-shifted closed $p$-forms on $X$ is given by the pullback
\begin{diagram}
\cF_{\bb{C}}^{p,\cl,\tu{in}}(X,n)  &\rTo  &\DR_{\geq p}^{\bb{C}}(X)[p+n] \\
\dTo     &        &\dTo_i\\
\bb{Z}[p+n]  &\rTo^{2\pi i}  &\DR^{\bb{C}}(X)[p+n]
\end{diagram}
of spaces.  
\end{dfn}

We will refer to a derived smooth stack endowed with a integral $n$-shifted complex closed $(p+1)$-form as an \textit{integral $n$-shifted $p$-preplectic} derived smooth stack.  

The prestack 
\[  \cF_{\bb{C}}^{p,\cl,\tu{in}}(-,n):(\bdAff_{\bb{R}})^\circ\ra\textbf{S}  \]
sending an affine derived manifold $A$ to $\cF^{p,\cl,\tu{in}}_{\bb{C}}(A,n)$ is a stack on the $\infty$-site of affine derived manifolds with respect to the \'etale topology.  There exists an equivalence 
\[     \cF_{\bb{C}}^{p,\cl,\tu{in}}(X,n)\ra\Map_{\dSmSt_{\bb{R}}}(X,\cF_{\bb{C}}^{p,\cl,\tu{in}}(-,n))  \]
of spaces.  When $n<0$, a $n$-shifted $p$-preplectic form on a derived smooth stack is automatically integral.

Let $A=(|A|,\cO_A)$ be an affine derived manifold.  The exponential map $\exp:\bb{C}\ra\bb{C}^*$ induces a map
\[     \log:\cO_A^*\ra\cO_A^\alg\otimes\bb{C}   \]
where $\cO_A^*$ is the sheaf of complex valued invertible elements on $|A|$.  We have an equivalence
\[    |\cO^*_A(|A|)|\ra\cO_A(\bb{C}^*)(|A|)  \]
of spaces where we have abused notation by also using $|-|$ on the left hand side to denote the result of applying the Dold-Kan functor to the $0$-truncation of the complex.

We define a stack
\[    \cO^*:(\bdAff_{\bb{R}})^\circ\ra\textbf{S}  \]
sending $A$ to $\Map_{\bSh_{\bcdga_{\bb{C}}}(A)}(1,\cO_A^*)$ and consider the map 
\[    \delta\log_{[1,p]}:\cO^*\ra\DR_{[1,p]}^{\bb{C}}[1]   \]
of stacks sending $\cO^*$ to the shifted $[1,p]$ truncation.  We make the following definition for the cofiber of the corresponding $n$-shifted map
\[   \delta\log_{[1,p]}[n]:\cO^*[n]\ra\tu{DR}_{[1,p]}^{\bb{C}}[1+n]   \] 
of derived smooth stacks.

\begin{dfn}\label{kpgerbe}
Let $A$ be an affine derived manifold and $n\in\bb{Z}$, $p>0$.  We call 
\[ \Ger^{(n,p)}(A):=\tu{cofib}(\delta\log_{[1,p]}(A)[n])    \] 
the space of \textit{$(n,p)$-gerbes} on $A$.  An $(n,p)$-gerbe will also be called a $n$-gerbe with \textit{$p$-connection data}.
\end{dfn}

\begin{rmk}
The terminology of $p$-connection \textit{data} comes from the fact that for gerbes, one can have a collection of $i$-forms up to level $i=p$.  Note also that Definition~\ref{kpgerbe} still makes sense for $n<0$ giving meaning to (flat) $n$-gerbes with connection for any $n\in\bb{Z}$.
\end{rmk}

The construction of the space of $(n,p)$-gerbes is functorial in $A$ and defines a prestack
\[   \Ger^{(n,p)}:(\bdAff_\bb{R})^\circ\ra\textbf{S}     \]
on the $\infty$-category of affine derived manifolds.

\begin{prop}\label{stackofgerbes}
The prestack $\Ger^{(n,p)}$ is a stack on the $\infty$-site of affine derived manifolds with respect to the \'etale topology.
\end{prop}

\begin{proof}
We know from Proposition~\ref{formsarestacks} that $\DR\otimes\bb{C}$ is a stack and therefore $\DR_{[p,q]}\otimes\bb{C}$ is a stack.  The result now follows from the fact that the $\infty$-category $\dSmSt_{\bb{R}}$ of derived smooth stacks is closed under colimits.
\end{proof}

We will now introduce the derived smooth stack of flat $(n,p)$-gerbes.  Informally, the complex underlying the space of flat $(n,p)$-gerbes on an affine derived manifold $A$ corresponds to the space of $(n,p)$-gerbes on $A$ whose top forms (in degree $p$) are closed.  Note there is a natural map
\[    \Ger^{(n,p)}(A)\ra\cF_{\bb{C}}^p(A,n)   \]
projecting the complex underlying $\Ger^{(n,p)}(A)$ to its $p$-th factor.

\begin{dfn}\label{flatkpgerbe}
Let $A$ be an affine derived manifold.  The pullback 
\[     \o{\Ger}^{(n,p)}(A):= \Ger^{(n,p)}(A)\times_{\cF_{\bb{C}}^p(A,n)} \cF_{\bb{C}}^{p,\cl}(A,n)  \]     
will be called the space of \textit{flat $(n,p)$-gerbes} on $A$.  A flat $(n,p)$-gerbe will also be called a flat $n$-gerbe with $p$-connection.
\end{dfn}

Again, the construction of the space of flat $(n,p)$-gerbes is functorial in $A$ and so defines a stack
\[   \o\Ger^{(n,p)}:(\bdAff_\bb{R})^\circ\ra\textbf{S}     \]
on the $\infty$-category of affine derived manifolds with respect to the \'etale topology.  We use this property together with Proposition~\ref{stackofgerbes} as follows to obtain a global definition of a $(n,p)$-gerbe on any derived smooth stack.

\begin{dfn}\label{gerbeonstack}
Let $X$ be a derived smooth stack.  The space of \textit{$(n,p)$-gerbes} on $X$ is given by
\[    \Ger^{(n,p)}(X):=\Map_{\dSmSt_\bb{R}}(X,\Ger^{(n,p)}).  \]
The space of \textit{flat $(n,p)$-gerbes} on $X$ is given by
\[    \o{\Ger}^{(n,p)}(X):=\Map_{\dSmSt_\bb{R}}(X,\o{\Ger}^{(n,p)}).  \]
\end{dfn}

\begin{ex}
Let $A$ be an affine derived manifold and consider the case $n=0$.  A complex line bundle on $A$ is a sheaf $F$ in $\bMod(\cO_A^\alg\otimes\bb{C})$ which is locally equivalent to $\cO_A^*$.  In Section~\ref{cg} we will show that $F$ is equivalent to a $0$-gerbe (without connection) in the sense of Definition~\ref{kpgerbe}.  Therefore a $(0,1)$-gerbe is simply a complex line bundle with connection on $A$.   A flat $(0,1)$-gerbe is a line bundle with flat connection.  A similar statement can be made for $(n,p)$-gerbes in general.
\end{ex}

The following is the first main result of this paper, the derived Weil-Kostant integrality theorem.

\begin{thm}\label{mainthm}
Let $p\geq 1$, $n\in\bb{Z}$ and $(p+n)>0$.  Let $(X,\omega)$ be a $n$-shifted $p$-preplectic derived smooth Artin stack.  
\begin{enumerate}
\item There exists a $(p+n-1,p)$-gerbe on $X$ with curvature $\omega$ if and only if $\omega$ is integral.
\item The space of $(p+n-1,p)$-gerbes on $X$ with curvature $\omega'$ is parametrized by the space of flat $(p+n-1,p)$-gerbes.
\end{enumerate}
\end{thm} 

\begin{proof}
By definition, one can realize the derived smooth stack of integral $n$-shifted closed $p$-forms as the pullback diagram
\begin{diagram}
\cF_{\bb{C}}^{p+1,\cl,\tu{in}}(-,n)  &\rTo  &\DR_{\geq p}^{\bb{C}}(-)[p+1+n]\\
\dTo_g     &        &\dTo_f\\
\bb{Z}(1)[p+1+n]  &\rTo^{}  &\DR^{\bb{C}}(-)[p+1+n]
\end{diagram}
in the $\infty$-category of sheaves of complex valued graded mixed cdga's where $\bb{Z}(1):=2\pi\sqrt{-1}\cdot\bb{Z}$.  Now observe that 
\[  \tu{cofib}~f\simeq\DR_{< p}(-)[p+1+n]\otimes\bb{C}  \]
and since formally (in any stable $\infty$-category)
\[   \tu{cofib}~g\simeq\tu{cofib}~f  \]
there exists a cofiber sequence
\begin{diagram}
\cF_{\bb{C}}^{p+1,\cl,\tu{in}}(-,n)  &\rTo  & \bb{Z}(1)[p+1+n]  \\
\dTo     &        &\dTo_{c}\\
0              &\rTo & \DR_{<p}^{\bb{C}}(-)[p+1+n]
\end{diagram}
of sheaves where $c$ is the constant functions map.  

We compose this cofiber sequence with the cofiber sequence of $c$ to give a diagram
\begin{diagram}
\cF_{\bb{C}}^{p+1,\cl,\tu{in}}(-,n)  &\rTo  &  \bb{Z}(1)[p+1+n]  &\rTo     & 0     \\
\dTo     &        &\dTo_{c[p+1+n]}   &        &\dTo\\
0              &\rTo & \DR_{\leq p}^{\bb{C}}(-)[p+1+n]  &\rTo  &\tu{cofib}(c[p+1+n])
\end{diagram}
of sheaves.  Again, cofiber sequences coincide with fiber sequences in a stable $\infty$-category and so since the two inner squares in the diagram above are pullbacks the outer diagram is a pullback leading to the equivalence
\[      \cF_{\bb{C}}^{p+1,\cl,\tu{in}}(-,n)\xras\tu{cofib}(c[p+n])   \]
of sheaves.

The exponential exact sequence $\bb{Z}(1)\ra\bb{C}\xra{\exp}\bb{C}^*$ induces a cofiber sequence
\begin{diagram}
\bb{Z}(1)  &\rTo  & \cO\otimes\bb{C}  \\
\dTo     &        &\dTo_{\exp}\\
0              &\rTo & \cO^*
\end{diagram}
and therefore   
\[   \tu{cofib}(c)\xras\tu{cofib}(\delta\log_{(1,p)}[-1])    \]
is an equivalence of sheaves.  As a result, for any derived smooth Artin stack $X$, there exists an equivalence
\[   \cF_{\bb{C}}^{p+1,\cl,\tu{in}}(X,n)\xras\Ger^{(p+n-1,p)}(X)  \]
of spaces.  This proves (1).

For (2), we must show that the following sequence
\begin{diagram}
\o\Ger^{(p+n-1,p)}(X)    &\rTo  &\Ger^{(p+n-1,p)}(X)  \\
\dTo  &                                                &\dTo\\
0       &\rTo                                    &  \cF_{\bb{C}}^{p+1,\cl}(X,n)   
\end{diagram}
is a cofiber sequence.  This follows from the diagram
\begin{diagram}
\o\Ger^{(p+n-1,p)}(-)    &\rTo  &\Ger^{(p+n-1,p)}(-)  \\
\dTo  &                                                &\dTo\\
\cF_{\bb{C}}^{p,\cl}(-,p+n-1)  &\rTo     &  \cF_{\bb{C}}^{p}(-,p+n-1) \\ 
\dTo  &                                                &\dTo_{\delta}\\
0       &\rTo                                    &  \cF_{\bb{C}}^{p+1,\cl}(-,p+n-1)   
\end{diagram}
of sheaves.  The top square is a pullback by definition and therefore a pushout since we are in the stable setting.  The bottom square is a pushout and therefore the composite square is a pushout.  
Therefore, for any derived smooth Artin stack $X$, we have the stated cofiber sequence.
\end{proof}

\begin{ex}
When our $n$-shifted $p$-preplectic derived smooth stack is simply a $p$-preplectic smooth manifold, then Theorem~\ref{mainthm} reduces to the classical Weil-Kostant integrality theorem stated in the introduction.  From Example~\ref{preplecman} we deduce that an integral zero-shifted $p$-preplectic derived smooth Artin stack of the form $X=\Spec(\Spec^{\Man_{\bb{R}}}M)$ for a smooth manifold $M$ is simply $M$ endowed with a complex closed $(p+1)$-form $\omega$ with integral periods.  The proof of Theorem~\ref{mainthm} then states that there exists a $(p-1,p)$-gerbe on $M$ whose curvature is $\omega$.  For the case $p=1$, the $(0,1)$-gerbe is a line bundle with connection.  For the case $p=2$, the $(1,2)$-gerbe is precisely a gerbe with connection and curving in the language of \cite{Br}.
\end{ex}

\section{Categorical gerbes}\label{cg}

In this section we introduce a complementary notion of $n$-gerbe based on a sheaf of linear $(\infty,n)$-categories.  We call these $n$-gerbes categorical $n$-gerbes to distinguish them from the $n$-gerbes introduced in Section~\ref{dwkit}.  Every $n$-gerbe admits a presentation as a categorical $n$-gerbe.

There exists a number of approaches to defining $(\infty,n)$-categories for $n\geq 0$.  We will utilize the theory of Segal $n$-categories contained in \cite{Si}.  However our results do not depend on the particular model and hold in any equivalent theory of $\inftyn$-categories.  For a review using our notation see Appendix~A.  

The cartesian closed model category of $(\infty,n)$-precategories is denoted $\PC^n(\sSet)$ where $\sSet$ is the category of simplicial sets endowed with the Kan model structure.  It is defined inductively by setting $\PC^0(\sSet):=\sSet$.  The objects of $\PC^n(\sSet)$ for $n>0$ are then pairs $(S,C)$ consisting of a set $S$ of \textit{objects} and a functor 
\[    C:\Delta^\circ_{S}\ra\PC^{n-1}(\sSet),    \] 
where $\DS$ is the category of sequences of objects in $S$ such that $C(x)$ is a final object for all $x\in S$.  An $\inftyn$-precategory $(S,C)$ is said to be an \textit{$\inftyn$-category} if it satisfies the Segal condition and for any sequence of objects $(x_0,\ldots,x_m)$ in $S$, the $\inftynmo$-precategory $C(x_0,\ldots,x_m)$ is an $\inftynmo$-category.  We will refer to the pair $(S,C)$ as simply $C$.

For any $\PC^{n}(\sSet)$-enriched model category $\sM$ with weak equivalences $W$, we let $\u{L}(\sM)$ denote the \textit{enriched localization} of $\sM$ with respect to $W$, ie. there exists an $\inftynpo$-category $\u{L}(\sM)$ and a map $l:\sM\ra\u{L}(\sM)$ satsifying the following universal property~: for any $\inftynpo$-category $C$, the induced map 
\[   \bb{R}\uHom(\u{L}(\sM),C)\ra\bb{R}\uHom(\sM,C)  \]
is fully faithful and its essential image consists of functors $f:\sM\ra C$ sending arrows in $W$ to equivalences in $C$.  This construction is described in Appendix~A.

The $\infty$-category of $\inftyn$-categories will be denoted 
\[   \bCatinfn:=L(\PC^n(\sSet))  \] 
and the $\inftynpo$-category of $\inftyn$-categories by 
\[    \uCatinfnpo:=\u{L}(\PC^n(\sSet)) \] 
using the enriched localization.  In Appendix~B we provide a definition of a symmetric monoidal $\inftyn$-category and we denote by 
\[   \bCatinfn^\otimes:=L(\SeMon(\PC^n(\sSet))_{\s{S}})  \] 
the $\infty$-category of symmetric monoidal $\inftyn$-categories and
\[   \uCatinfn^\otimes:=\u{L}(\SeMon(\PC^n(\sSet))_{\s{S}})  \] 
the $\inftynpo$-category of symmetric monoidal $\inftyn$-categories.  Here $\SeMon(\PC^n(\sSet))_{\s{S}}$ is the (enriched) model category of Segal monoid objects in $\PC^n(\sSet)$ with respect to the \textit{special} model structure.

Given any commutative monoid object $R$ in the enriched model category $\PC^n(\sSet)$, the model category $\Mod_R(\PC^n(\sSet))$ of $R$-modules is a $\PC^n(\sSet)$-enriched model category (see Appendix~B for more details).  Since for any symmetric monoidal $\PC^{n}(\sSet)$-enriched model category $\sM$ with weak equivalences $W$ the $\inftynpo$-category $\u{L}^\otimes(\sM)$ given by the enriched monoidal localization of $\sM$ is symmetric monoidal, the $\inftynpo$-category 
\[   \u\Mod_R(\PC^n(\sSet)):=\u{L}^\otimes(\Mod_R(\PC^n(\sSet)))    \]
of $R$-modules is a symmetric monoidal $\inftynpo$-category.

Let $R$ be a commutative dg-algebra over $\bb{R}$.  We will define the $(\infty,n+1)$-category of complex $R$-linear $(\infty,n)$-categories by induction.  Let 
\[    {\Lin}_{\bb{C}}^0(R):=\Mod(R\otimes\bb{C})  \]
be the symmetric monoidal model category of complex $R$-modules and for all $n>0$, 
\[   {\Lin}_{\bb{C}}^n(R):=\Mod_{\Lin_{\bb{C}}^{n-1}(R)}(\PC^n(\sSet))  \]
the symmetric monoidal ($\PC^n(\sSet)$-enriched) model category of complex $R$-linear $\inftyn$-categories.  

\begin{dfn}
Let $R$ be a commutative dg-algebra over $\bb{R}$.  A $(\infty,n)$-category is said to be \textit{complex $R$-linear} if it is endowed with the structure of a $\Lin_{\bb{C}}^{n-1}(R)$-module object in the $\PC^n(\sSet)$-enriched symmetric monoidal model category $\PC^n(\sSet)$ of $(\infty,n)$-categories.  
\end{dfn}

For $n\geq 0$, let
\[   \bLin_{\bb{C}}^n(R):=L^\otimes({\Lin}_{\bb{C}}^n(R))   \]
be the symmetric monoidal $\infty$-category of \textit{complex $R$-linear $(\infty,n)$-categories}.  Using the enriched monoidal localization described in Appendix~B, then
\[   \u{\Lin}_{\bb{C}}^n(R):=\u{L}^\otimes({\Lin}_{\bb{C}}^n(R))   \]
is the symmetric monoidal $(\infty,n+1)$-category of \textit{complex $R$-linear $(\infty,n)$-categories}.

\begin{dfn}
Let $R$ be a commutative dg-algebra over $\bb{R}$.  A commutative monoid object in the model category of complex $R$-linear $\inftyn$-categories is called a \textit{complex $n$-algebra} over $R$.
\end{dfn}

Let
\[      \textbf{CAlg}_{\bb{C}}^n(R):=L(\CMon(\Lin_{\bb{C}}^{n}(R))).  \]
denote the $\infty$-category of complex $n$-algebras over $R$.  There exists an equivalence equivalence
\[  \textbf{CAlg}_{\bb{C}}^n(R)\xras \bCMon(\bLin_{\bb{C}}^n(R))\xras (\bCat_{(\infty,n)}^{\otimes})_{\u{\Lin}_{\bb{C}}^{n-1}(R)/}    \]
of $\infty$-categories where $\bCMon(C)$ denotes the $\infty$-category of commutative monoid objects in an $\infty$-category (see Appendix~B).  For any complex $n$-algebra $S$ over $R$, we consider the $\inftynpo$-category 
\[   \u\Mod^n(S):=\u{L}(\Mod_{S}(\Lin_{\bb{C}}^n(R)))   \]
of modules over $S$.  

Let $A=(|A|,\cO_A)$ be an affine derived manifold and $U\subseteq|A|$.  We will define the $\infty$-category of sheaves of complex valued $A$-linear $\inftyn$-categories also by induction.  Let 
\[   A-{\Lin}_{\bb{C}}^0(U):=\Mod(\cO_A^\alg(U)\otimes\bb{C})  \]
be the symmetric monoidal model category of $\cO_A^\alg(U)\otimes\bb{C}$-modules, for the constant sheaf $\bb{C}$, and for all $n>0$, 
\[   A-{\Lin}_{\bb{C}}^n(U):=\Mod_{A-{\Lin}_{\bb{C}}^{n-1}(U)}(\PC^n(\sSet))  \]
the symmetric monoidal ($\PC^n(\sSet)$-enriched) model category of complex valued $\cO_A^\alg(U)\otimes\bb{C}$-linear $\inftyn$-categories.  For all $n\geq 0$ we let
\[  A-\u\Lin_{\bb{C}}^n:O(A)^\circ\ra\bCat_{\inftynpo}^\otimes   \]
denote the sheaf of symmetric monoidal $\inftynpo$-categories sending $U$ to $\u{L}^\otimes(A-{\Lin}_{\bb{C}}^n(U))$.

\begin{dfn}
Let $A=(|A|,\cO_A)$ be an affine derived manifold and $n\geq 0$.  A  $(A-\u\Lin_{\bb{C}}^{n})$-module object in the model category of sheaves of $\inftynpo$-categories on $|A|$ will be called a \textit{sheaf of complex $A$-linear $\inftyn$-categories}.
\end{dfn}

It will be useful to use the following notation.  Let
\[  \Lin\Sh_{\bb{C}}^0(A):=\Mod(\cO_A^\alg\otimes\bb{C}) \] 
denote the symmetric monoidal model category of complex $\cO_A^\alg$-modules and
\[  \bLin\bSh_{\bb{C}}^0(A):=L^\otimes(\Lin\Sh_{\bb{C}}^0(A)) \] 
the symmetric monoidal $\infty$-category of complex valued $\cO_A^\alg$-modules.  For all $n>0$, let 
\[   \Lin\Sh_{\bb{C}}^n(A):=\Mod_{A-\u\Lin_{\bb{C}}^{n-1}}(\Sh_{\Catinfn}(A))   \]
denote the symmetric monoidal model category of sheaves of complex valued $A$-linear $\inftyn$-categories and
\[   \bLin\bSh_{\bb{C}}^n(A):=L^\otimes(\Lin\Sh_{\bb{C}}^n(A))   \]
the symmetric monoidal $\infty$-category of sheaves of complex valued $A$-linear $\inftyn$-categories.

We will refine this definition somewhat by only allowing $\inftyn$-categories that are \textit{locally presentable}.

\begin{dfn}
Let $C$ be a $\inftyn$-category and $n>0$.  Then $C$ is said to be \textit{locally presentable} if there exists a combinatorial $\PC^{n-1}(\sSet)$-enriched model category $\sN$ and an equivalence
\[    C\simeq\u{L}(\sN)   \]
of $\inftyn$-categories.
\end{dfn}

We will denote by $\bCatinfn^{\tu{lp}}$ the subcategory of $\bCatinfn$ spanned by locally presentable $\inftyn$-categories and continuous functors (see \cite{Si} for more details on colimits in $\inftyn$-categories).  The $\infty$-category of sheaves of complex valued $A$-linear locally presentable $\inftyn$-categories will be denoted by $\bLin\bSh_{\bb{C}}^n(A)^{\tu{lp}}$.

\begin{dfn}
Let $A$ be an affine derived manifold.  A commutative monoid object in the model category of complex $\cO_A^\alg$-modules is called a \textit{sheaf of complex $0$-algebras} over $A$.  For all $n>0$, a commutative monoid object in the model category of sheaves of complex $A$-linear $\inftyn$-categories is called a \textit{sheaf of complex $n$-algebras} over $A$.
\end{dfn}

For all $n\geq 0$, let
\[   \textup{CAlgSh}_{\bb{C}}^n(A):=\CMon(\Lin\Sh_{\bb{C}}^n(A))  \]
denote the model category of sheaves of complex $n$-algebras over $A$ and
\[   \textbf{CAlg}\bSh_{\bb{C}}^n(A):=L(\textup{CAlgSh}_{\bb{C}}^n(A))   \]
the $\infty$-category of sheaves of complex $n$-algebras over $A$.

These definitions lead to the chain of equivalences
\[    \textbf{CAlg}\bSh_{\bb{C}}^n(A)\xras\bCMon(\bLin\bSh_{\bb{C}}^n(A))\xras\bSh_{\bCatinfn^\otimes}(A)_{A-\u\Lin_{\bb{C}}^{n-1}/}  \]
of $\infty$-categories for all $n\geq 0$.

Modules over sheaves of complex $n$-algebras are now defined in the obvious way.  Let 
\[   \Mod^n(R_A):=\Mod_{R_A}(\Lin\Sh_{\bb{C}}^n(A))  \]
denote the model category of modules over a sheaf $R_A$ of complex $n$-algebras over $A$ and
\[    \bMod^n(R_A):=L(\Mod^n(R_A))  \]
the $\infty$-category of $R_A$-modules.

We can impose a finiteness condition on all of the above definitions.  Recall that an object in a symmetric monoidal $\infty$-category is said to be \textit{rigid} or \textit{dualizable} if it admits a dual in its homotopy category \cite{rig}.  In higher categorical dimension, we will need the notion of \textit{fully dualizable} as contained in \cite{Lu7}.  We briefly recall the definition.  

Let $C$ be a $\inftyn$-category for $n\geq 2$ and denote by $\h_2(C)$ its homotopy 2-category.  Then $C$ is said to admit adjoints for $1$-morphisms if the $2$-category $\h_2(C)$ admits adjoints for $1$-morphisms.  Let $1<k<n$.  Then $C$ is said to admit adjoints for $k$-morphisms if the $\inftynmo$-category of morphisms between any two objects admits adjoints for $(k-1)$-morphisms.  An $\inftyn$-category is said to have \textit{adjoints} if $C$ admits adjoints for $k$-morphisms for all $0<k<n$.

When $C$ is endowed with a symmetric monoidal structure, its homotopy category $\h(C)$ is symmetric monoidal.  In this case, let $B_2(\h(C))$ be the $2$-category with a single object and whose category of $1$-morphisms is $\h(C)$ with composition is given by the tensor product in $C$.  

\begin{dfn}
Let $C$ be a symmetric monoidal $\inftyn$-category for $n\geq 2$.  Then $C$ is said to \textit{have duals} if 
\begin{enumerate}
\item The $2$-category $B_2(\h(C))$ has adjoints for $1$-morphisms.
\item The $\inftyn$-category $C$ has adjoints for $(k-1)$-morphisms.  
\end{enumerate}
\end{dfn}

We denote by $C^{\fd}$ the symmetric monoidal $\inftyn$-category with duals satisfying the following universal property~: for any symmetric monoidal $\inftyn$-category $D$ admitting duals, there exists a symmetric monoidal functor $i:C^{\fd}\ra C$ such that composition with $i$ induces an equivalence
\[    \u{\tu{Fun}}^\otimes(D,C^{\fd})\ra\u{\tu{Fun}}^\otimes(D,C)  \]
of $\inftyn$-categories.  
For $n=1$, fully dualizable will mean \textit{dualizable} as is standard in the setting of $\infty$-categories.

Since the collection of fully dualizable objects is stable with respect to the symmetric monoidal structure, we have an endofunctor on the $\infty$-category $\bCatinfn^\otimes$ of symmetric monoidal $\inftyn$-categories sending a symmetric monoidal $\inftyn$-category $C$ to its symmetric monoidal $\inftyn$-category $C^{\fd}$ of fully dualizable objects.  We extend this construction levelwise to obtain a functor
\[    (-)^{\tu{fd}}_n:\bSh_{\bCatinfn^\otimes}(A)\ra\bSh_{\bCatinfn^\otimes}(A)   \]
between $\infty$-categories for any topological space $|A|$. 

Using these definitions, for $n\geq 0$, we denote the image of $A-\u\Lin_{\bb{C}}^{n}$ under the functor $(-)^\fd_{n+1}$ by $A-\u\Lin_{\bb{C}}^{n,\fd}$.  Explicitly, 
\[   A-\u\Lin_{\bb{C}}^{n,\fd}:O(A)^\circ\ra\bCat_{\inftynpo}^\otimes   \]
is the sheaf sending $U$ to the symmetric monoidal $\inftynpo$-category $\u\Lin^{n}_\bb{C}(A)(U)^{\fd}$ of $\bb{C}$-valued fully dualizable $\cO_A^\alg(U)$-linear $\inftyn$-categories.  

Likewise, for $n\geq 0$ and some sheaf $R_A$ of fully dualizable complex $n$-algebras over $A$, we denote by $R_A-\u\Lin_{\bb{C}}^{n,\fd}$ the image of $R_A-\u\Lin_{\bb{C}}^{n}\in\bMod^{n+1}(R_A)$ under the functor $(-)^{\fd}_{n+1}$.  Explicitly, 
\[  R_A-\u\Lin_{\bb{C}}^{n,\fd}:O(A)^\circ\ra\bCat_{\inftynpo}^\otimes    \]
is the sheaf of symmetric monoidal $\inftynpo$-categories sending $U$ to the symmetric monoidal $\inftynpo$-category $(\u{L}^\otimes(\Mod_{R_A(U)}(\u\Lin^{n}_\bb{C}(A)(U))))^{\fd}$ of fully dualizable $R_A(U)$-modules.  This is an $A$-linear $\inftynpo$-category and therefore defines an object in $\bLin\bSh_{\bb{C}}^{n+1}(A)$.  

\begin{dfn}
Let $A$ be an affine derived manifold and $n\geq 0$.  A \textit{perfect $n$-sheaf} over $A$ is an object in $\bLin\bSh_{\bb{C}}^n(A)^{\lp}$ locally, for the \'etale topology, of the form $R_A-\u\Lin_{\bb{C}}^{n-1,\fd}$ where $R_A$ is a sheaf of complex fully dualizable $(n-1)$-algebras over $A$.
\end{dfn}

We denote the full subcategory of $\bLin\bSh_{\bb{C}}^{n}(A)^{\lp}$ spanned by perfect $n$-sheaves by $\textbf{PerSh}_{\bb{C}}^n(A)$.  

Let us now restrict the $\infty$-category of perfect $n$-sheaves to those which satisfy a local triviality assumption.  This definition should eventually be replaced by a theorem, that is, all perfect $n$-sheaves over $A$ are expected to be locally trivial under some mild assumptions.

\begin{dfn}\label{lt}
Let $A=(|A|,\cO_A)$ be an affine derived manifold.  The $\infty$-category of \textit{locally trivial} perfect $n$-sheaves over $A$ is the full subcategory of $\textbf{PerSh}_{\bb{C}}^n(A)$ satisfying the following conditions~:
\begin{enumerate}
\item Let $0\leq m\leq n$.  For any sheaf $R_A$ of fully dualizable $m$-algebra's over $A$, there exists an equivalence
\[       R_A\simeq B^m\cO_A^\alg    \] 
of sheaves.
\item Let $0\leq m\leq n$ and $R$ be a commutative dg-algebra over $\bb{C}$.  For any fully dualizable $R$-linear $\inftyn$-category $C$, there exists an equivalence
\[       C\simeq\u{\Mod}^{n-1}(S)    \]
of $\inftyn$-categories for some fully dualizable $(n-1)$-algebra $S$ over $R$.
\end{enumerate}
\end{dfn}

We denote the $\infty$-category of locally trivial perfect $n$-sheaves by $\textbf{PerSh}_{\bb{C}}^n(A)^{\tu{lt}}$.

The $\infty$-category of locally trivial perfect $n$-sheaves on $A$ is a symmetric monoidal $\infty$-category using the pointwise monoidal structure.  We can then consider the $\infty$-category of \textit{invertible} objects in $\textbf{PerSh}_{\bb{C}}^n(A)^{\tu{lt}}$, ie. objects $F_A$ for which there exists another object $F_A'$ such that
\[  F_A\otimes F_A'\ra 1\]
is an equivalence in $\textbf{PerSh}_{\bb{C}}^n(A)^{\tu{lt}}$ where the right hand side denotes the unit object.   More generally, for any symmetric monoidal $\inftyn$-category $C$, we will denote the full subcategory of $C$ spanned by invertible objects by $C^{\tu{inv}}$.

\begin{dfn}
Let $A$ be an affine derived manifold.  A \textit{categorical $n$-gerbe} on $A$ is an invertible object in $\textbf{PerSh}_{\bb{C}}^n(A)^{\tu{lt}}$.
\end{dfn}

Let $\textbf{CatGer}^n(A)$ denote the full subcategory of $\textbf{PerSh}_{\bb{C}}^n(A)$ spanned by categorical $n$-gerbes on $A$.  

The $\infty$-category of categorical $n$-gerbes on $A$ is a $\infty$-groupoid.  The construction of the $\infty$-category of categorical $n$-gerbes is functorial in $A$ and thus we have a presheaf of spaces on the $\infty$-site $(\bdAff_{\bb{R}},et)$ of affine derived manifolds sending $A$ to $\textbf{CatGer}^n(A)$.  We denote by 
\[      \tu{CatGer}^n:(\bdAff_{\bb{R}})^\circ\ra\textbf{S}   \]
the stack \textit{associated} to the prestack sending $A$ to $\textbf{CatGer}^n(A)$.

We can now define the notion of a categorical $n$-gerbe on an arbitrary derived smooth stack.

\begin{dfn}
Let $X$ be a derived smooth stack.  Then
\[    \tu{CatGer}^n(X):=\Map_{\dSmSt_{\bb{R}}}(X,\tu{CatGer}^n).     \] 
is the space of \textit{categorical $n$-gerbes} on $X$.
\end{dfn}

Let $A=(|A|,\cO_A)$ be an affine derived manifold.  Recall that we have the group stack
\[   \bb{G}_m(A):O(A)^\circ\ra\textbf{Gp}(\textbf{S})   \]
on $A$ sending $U$ to $|\cO^*_A(U)|$ (see Example~\ref{gmex}).  The main result in this section enabling a link to the results of Section~\ref{dwkit} is the following~:  

\begin{prop}\label{gerandk}
Let $n\geq 0$.  There exists an equivalence
\[   \tu{CatGer}^n\ra K(K(\bb{G}_m,1)\rtimes\bb{Z},n)  \]
of derived smooth stacks.
\end{prop}

\begin{proof}
Let $A=(|A|,\cO_A)$ be an affine derived manifold and $\{U_i\}$ be a cover of $|A|$.  The proof will be by induction.  

Let $F_A$ be a categorical $0$-gerbe on $A$.  Then, by definition, $F_A$ is an invertible object in the $\infty$-category $\bMod(\cO_A^\alg\otimes\bb{C})^{\fd}$ of dualizable complex $\cO_A^\alg$-modules, ie. $F_A$ is a (shifted) line bundle on $A$.  Therefore, there exists an equivalence
\[  \tu{CatGer}^{0}\simeq G^{(0)}   \]
where $G^{(0)}$ is the group stack on the $\infty$-category of derived manifolds sending $A$ to the grouplike commutative monoid object in $\bSh(A)$ with the following explicit description.  It is the sheaf of symmetric monoidal groupoids on $|A|$ whose levelwise objects at $U$ are integers $n\in\bb{Z}$ and whose non-trivial morphisms for any object $n$ is given by $\Aut(n)=\bb{G}_m(A)(U)=|\cO_A^*(U)|$.  The tensor product $\otimes:\bb{Z}\times\bb{Z}\ra\bb{Z}$ is given by addition and the symmetry structure is 
\[   (-1)^{nm}:n\otimes m\simeq m\otimes n   \]
to take into account the non-commutativity of shifts.  It follows that 
\[    G^{(0)}\simeq K(\bb{G}_m,1)\rtimes\bb{Z}  \] 
where $\bb{Z}$ is the constant stack.

Set $G:=K(\bb{G}_m,1)\rtimes\bb{Z}$.  We now assume that there exists an equivalence
\[   \tu{CatGer}^n\xras K(G,n)  \]
of stacks.  By definition, for any categorical $n$-gerbe $F_A$ on $A$, there exists an equivalence
\[     F_A\simeq R_A-\u\Lin_{\bb{C}}^{n-1,\fd}  \]
in $\bLin\bSh_{\bb{C}}^n(A)^\lp$ for some sheaf $R_{A}$ of complex fully dualizable $(n-1)$-algebras on $A$.  Due to the local triviality condition (i) of Definition~\ref{lt}, there exists an equivalence $R_A\simeq B^{n-1}\cO_A^\alg$ of sheaves.  Therefore, any object $F_A$ in $ \tu{CatGer}^{n}(A)$ is locally, for the \'etale topology, equivalent to $A-\u\Lin_{\bb{C}}^{n-1,\fd}$.  As a result, using the fact that we only consider invertible objects, we have an equivalence
\[     \tu{CatGer}^{n}\xras K(G^{(n-1)},1)   \]
where $G^{(n-1)}$ is the group stack sending $A$ to the stack $\u\Aut(A-\u\Lin^{n-1,\fd})$ of automorphisms of $A-\u\Lin^{n-1,\fd}$ (here $A-\u\Lin^{n-1,\fd}$ is regarded as a module over itself and $\u\Aut$ denotes the internal hom in the $\infty$-category $\bLin\bSh_{\bb{C}}^n(A)^\lp$).  

Now let $F_A$ be a categorical $(n+1)$-gerbe on $A$.  In analogy to above, there exists an equivalence
\[   F_A\simeq A-\u\Lin_{\bb{C}}^{n,\fd}  \]
by the local triviality condition (i) in Definition~\ref{lt} and thus an equivalence
\[     \tu{CatGer}^{n+1}\xras K(G^{(n)},1)   \]
where $G^{(n)}$ is the group stack sending $A$ to the stack $\u\Aut(A-\u\Lin^{n,\fd})$ of automorphisms of $A-\u\Lin^{n,\fd}$.  However, there exists an equivalence
\[   \u\Aut(A-\u\Lin^{n,\fd})\xras A-\u\Lin^{n,\tu{inv}} \]
of stacks since every linear automorphism is given by tensor product with an invertible object (here the stack $A-\u\Lin^{n,\tu{inv}}$ sends $U$ to $A-\u\Lin^{n}(U)^{\tu{inv}}$ where $U\in\{U_i\}$).  

Every invertible object is fully dualizable.  Therefore, by the local triviality condition (ii) of Definition~\ref{lt}, every object in $A-\u\Lin_{\bb{C}}^{n}(U)^{\tu{inv}}$ is equivalent to $\u\Mod^{n-1}(S(U))$ for some complex $(n-1)$-algebra $S(U)$ over $\cO_A^\alg(U)$.  But again by the local triviality condition (i), there exists an equivalence $S(U)\simeq B^{n-1}\cO_A^\alg(U)$ of sheaves.  Therefore we have an equivalence 
\[     \tu{CatGer}^{n+1}\xras K(G^{(n-1)},2)  \]
of stacks.  By assumption, $G^{(n-1)}\simeq K(G,n-1)$ and therefore
\[   \tu{CatGer}^{n+1}\xras K(G,n+1)  \]
as required.
\end{proof}

It follows from the proof of Proposition~\ref{gerandk} that there exists an equivalence
\[    \Omega(\tu{CatGer}^{n+1}(A))\xras\tu{CatGer}^n(A)     \]
of spaces.  This is functorial in $A$ which leads to an equivalence 
\[  \Omega(\tu{CatGer}^{n+1}(X))\xras\tu{CatGer}^n(X)  \] 
for any derived smooth stack $X$.

In Section~\ref{dwkit} we introduced the notion of a $(n,p)$-gerbe on an affine derived manifold.  This is a $n$-gerbe with connections up to level $p$.  When we forget the connection data, a \textit{$n$-gerbe} on a derived affine manifold $A$ is a stack $F$ on $|A|$ locally equivalent to $\B^{n+1}\cO_A^*$.  The full subcategory of $\bSh(A)$ spanned by $n$-gerbes on $A$ will be denoted $\bGer^n(A)$.  In the notation of Section~\ref{dwkit}, the space $\bGer^n(A)$ is simply $\bGer^{(n,0)}(A)$.  We denote by 
\[    \Ger^n:(\bdAff_{\bb{R}})^\circ\ra\textbf{S} \]
the stack associated to the prestack sending $A$ to the space of $n$-gerbes $\bGer^n(A)$ on $A$.    

\begin{dfn}
Let $X$ be a derived smooth stack.  Then 
\[     \Ger^n(X):=\Map_{\dSmSt_{\bb{R}}}(X,\Ger^n).     \] 
is the space of \textit{$n$-gerbes} on $X$.
\end{dfn}

A $n$-gerbe is classifed by $K(\bb{G}_m,n+1)$, ie. for any derived smooth stack $X$, there exists an equivalence
\[     \Ger^n(X)\xras\Map_{\dSmSt_{\bb{R}}}(X,\B^{n+1}\bb{G}_m)   \]
of spaces.  Therefore, the relationship between $n$-gerbes and categorical $n$-gerbes that we need for our main result in the following section is that there exists a canonical morphism
\[    \chi:\Ger^{n}\ra\tu{CatGer}^n    \]
of stacks sending $n$-gerbes to categorical $n$-gerbes.

\section{Prequantum categories}\label{pqc}

In this section we provide the second main result of this paper which constructs a canonical functor associating to any integral $n$-shifted $p$-preplectic derived smooth Artin stack, a $\bb{C}$-linear $(\infty,p+n-1)$-category.  

Forming vector spaces of sections of complex line bundles over moduli spaces has a long history in geometric (pre)quantization of classical mechanical systems.  Hence we will use the term \textit{prequantum category} to describe the higher linear categories arising from taking sections of gerbes over derived moduli problems.  

Let $X$ be a derived smooth stack.  In Section~\ref{spds} we defined the space 
\[   \cF_{\bb{C}}^{p,\cl,\tu{in}}(X,n):=\Map_{\dSmSt_{\bb{R}}}(X,\cF_{\bb{C}}^{p,\cl,\tu{in}}(-,n))  \] 
of integral $n$-shifted complex closed $p$-forms on $X$.  By abuse of notation, we let 
\[  \cF_{\bb{C}}^{p,\cl,\tu{in}}(-,n):(\dSmSt_{\bb{R}})^\circ\ra\S  \]
be the prestack on the $\infty$-category of derived smooth stacks sending $X$ to $\cF_{\bb{C}}^{p,\cl,\tu{in}}(X,n)$.  

By the Grothendieck construction over a base $\infty$-category (see Chapter~3 of \cite{L1}), this functor corresponds to a fibered space
\[   p:\int_{\dSmSt_{\bb{R}}}\cF_{\bb{C}}^{p,\cl,\tu{in}}(-,n)\ra\dSmSt_{\bb{R}}   \]
over the $\infty$-category of derived smooth stacks.  We define
\[   p\mbox{-}\PrPlSt_{n}^{\tu{in}}:=\bFun(*,\int_{\dSmSt_{\bb{R}}}\cF_{\bb{C}}^{p+1,\cl,\tu{in}}(-,n))\times_{\bFun(*,\dSmSt_{\bb{R}})}\dSmSt_{\bb{R}}   \]
to be the $\infty$-category of \textit{integral $n$-shifted $p$-preplectic derived smooth stacks}.  It is the $\infty$-category of pairs $(X,\omega)$ where $X$ is a derived smooth stack and $\omega$ is an integral $n$-shifted $p$-preplectic structure on $X$, ie. a $n$-shifted complex closed $(p+1)$-form.  

Likewise, we let
\[   \tu{CatGer}^n:(\dSmSt_{\bb{R}})^\circ\ra\S   \]
be the prestack on the $\infty$-category of derived smooth stacks sending $X$ to $\tu{CatGer}^n(X)$.  This functor corresponds to a fibered space
\[   q:\int_{\dSmSt_{\bb{R}}}\tu{CatGer}^n\ra\dSmSt_{\bb{R}}   \]
over the $\infty$-category of derived smooth stacks.  We define
\[  n\mbox{-}\textbf{GerSt}:=\bFun(*,\int_{\dSmSt_{\bb{R}}} \tu{CatGer}^n)\times_{\bFun(*,\dSmSt_{\bb{R}})}\dSmSt_{\bb{R}}  \]
to be the $\infty$-category of derived smooth stacks endowed with a categorical $n$-gerbe. 

Finally, recall that to any $\inftyn$-category we can associate its underlying space by discarding non-invertible morphisms.  This arises through a functor 
\[    \kappa:\bCatinfn\ra\S \]
sending an $\inftyn$-category to its interior (see Section~21 of \cite{Si} for the definition of this functor).  We define a presheaf
\[    \Lin_{\inftyn}:\bdAff^\circ\ra\S   \]
of spaces sending $A$ to the space $\kappa(\bLin^n_{\bb{C}}(A))$ of complex valued $A$-linear $\inftyn$-categories.  We abuse notation by writing $\Lin_{\inftyn}$ for both its associated stack and  
\[   \Lin_{\inftyn}:(\dSmSt_{\bb{R}})^\circ\ra\S  \]
for the prestack on the $\infty$-category of derived smooth stacks sending $X$ to the space $\Lin_{\inftyn}(X)=\Map(X,\Lin_{\inftyn})$ of $\cO_X\otimes\bb{C}$-linear $\inftyn$-categories.  This functor corresponds to a fibered space
\[  r:\int_{\dSmSt_{\bb{R}}}\Lin_{\inftyn}\ra\dSmSt_{\bb{R}}   \]
over the category of derived smooth stacks.  We define
\[  n\mbox{-}\textbf{LinSt}:=\bFun(*, \int_{\dSmSt_{\bb{R}}}\Lin_{\inftyn})\times_{\bFun(*,\dSmSt_{\bb{R}})}\dSmSt_{\bb{R}}  \]
to be the $\infty$-category of derived smooth stacks $X$ together with a complex valued $\cO_X\otimes\bb{C}$-linear $\inftyn$-category.  

We have inclusions
\[  p\mbox{-}\PrPlAr_{n}^{\tu{in}}\subset p\mbox{-}\PrPlSt_{n}^{\tu{in}}, \hspace{6mm} n\mbox{-}\textbf{GerAr}\subset n\mbox{-}\textbf{GerSt}, \hspace{6mm} n\mbox{-}\textbf{LinAr}\subset n\mbox{-}\textbf{LinSt},
\]
of $\infty$-categories spanned by objects whose derived smooth stacks are Artin. 

\begin{thm}\label{mainthm2}
Let $p>0$, $n\in\bb{Z}$ and $(p+n)>0$.  There exists a canonical functor
\[    \sN_n^p:p\mbox{-}\tu{\PrPlAr}_{n}^{\tu{in}}\ra (p+n-1)\mbox{-}\textbf{\tu{LinAr}}  \]
between $\infty$-categories.
\end{thm}

\begin{proof}
Let $(X,\omega)$ be a derived smooth Artin stack endowed with an integral $n$-shifted complex closed $(p+1)$-form.  Then the derived Weil-Kostant integrality Theorem~\ref{mainthm} states that there exists a $(p+n-1)$-gerbe $G$ on $X$.  By Section~\ref{cg}, there exists a morphism $\chi$ associating to $G$ a categorical $(p+n-1)$-gerbe $H:=\chi(G)$.  Therefore there exists a map $\phi$
\begin{diagram}
\int_{\dSmSt_{\bb{R}}}\cF_{\bb{C}}^{p,\cl,\tu{in}}(-,n) &&\rTo^{\phi} &&\int_{\dSmSt_{\bb{R}}}\tu{CatGer}^{(p+n-1)} \\
        &\rdTo^{p}&         &\ldTo_{q}&\\
      && \dSmSt_{\bb{R}} &&
\end{diagram}
between spaces which preserves cartesian morphisms sending an element $(X,\omega)$ to $(X,H)$.  This induces a functor 
\[   \phi_*:p\mbox{-}\PrPlAr_{n}^{\tu{in}}\ra (p+n-1)\mbox{-}\textbf{GerAr}  \]
between $\infty$-categories.  

A categorical $(p+n-1)$-gerbe is a quasi-coherent sheaf of complex valued $\cO_X$-linear $(\infty,p+n-1)$-categories and therefore its collection of sections is a complex valued $\cO_X$-linear $(\infty,p+n-1)$-category.  Therefore we compose $\phi$ with the map
\begin{diagram}
\int_{\dSmSt_{\bb{R}}}\tu{CatGer}^{(p+n-1)} &&\rTo^{\psi} &&\int_{\dSmSt_{\bb{R}}}\Lin_{(\infty,p+n-1)} \\
        &\rdTo^{q}&         &\ldTo_{r}&\\
      && \dSmSt_{\bb{R}} &&
\end{diagram}
sending a pair $(X,H)$ to $(X,\Gamma(H))$ which preserves cartesian morphisms where $\Gamma(H)$ is the complex valued $\cO_X$-linear $(\infty,p+n-1)$-category of sections of $H$.  This induces a functor
\[   \psi_*:(p+n-1)\mbox{-}\textbf{GerAr}\ra (p+n-1)\mbox{-}\textbf{LinAr}    \]
between $\infty$-categories.  We set 
$\sN_n^p(X,\omega):=\psi_*\circ\phi_*(X,\omega)=(X,\Gamma(H))$ with the obvious notion of morphism. 
\end{proof}

Let $(X,\omega)$ be a $n$-shifted $p$-preplectic smooth Artin stack.  Let $n\mbox{-}\bLin$ denote the $\infty$-category of $\bb{C}$-linear $\inftyn$-categories and  
\[    f:(p+n-1)\mbox{-}\textbf{LinAr}\ra (p+n-1)\mbox{-}\bLin   \]
the forgetful functor sending a pair $(X,C_X)$, consisting of a derived smooth Artin stack $X$ together with a complex $\cO_X$-linear $(\infty,p+n-1)$-category $C_X$, to $C_{\bb{C}}$ where $C_{\bb{C}}$ is the $\bb{C}$-linear $(\infty,p+n-1)$-category associated to $C_X$ induced from the natural map $X\ra *$.  Using the notation in the proof of Theorem~\ref{mainthm2}, we can define a $\bb{C}$-linear $(\infty,p+n-1)$-category by setting
\[   \sP^p_n(X,\omega):=f\circ\sN^p_n(X,\omega) \]
which one may call the prequantum functor on the $\infty$-category of integral $n$-shifted $p$-preplectic derived smooth Artin stacks.

\begin{cor}\label{maincor}
Let $p\geq 2$, $n\in\bb{Z}$ and $(p+n)>0$.  There exists a canonical functor
\[    \sP_n^p:p\mbox{-}\tu{\PrPlAr}_{n}^{\tu{in}}\ra (p+n-1)\mbox{-}\bLin  \]
of complex linear $(\infty,p+n-1)$-categories.
\end{cor}

\appendix
\numberwithin{equation}{section}

\section{$(\infty,n)$-categories}\label{inftyncategories}

This work makes extensive use of the theory of higher categories.  We refer the reader to \cite{L1} for a self contained exposition of the theory of $(\infty,1)$-categories and to \cite{Si} for the foundational results of $\inftyn$-categories.  In this appendix we simply state our definition of $(\infty,n)$-category following \cite{Si} and discuss the process of obtaining a $(\infty,n)$-category from an enriched model category.  We include the $(\infty,n)$-Yoneda lemma as an illustration of these techniques.  Other equivalent approaches to the theory of $\inftyn$-categories exist in the rapidly growing literature, see for example \cite{Be}, and we leave it to the reader to make the often straightforward adjustments for their preferred model.

An $\inftyn$-category is a category with $k$-morphisms for all $k\in[1,\infty]$, invertible for $k>n$.  The simplest way to formulate a definition of $\inftyn$-category is by induction.  We begin, more generally, by defining the notion of a weak $\sM$-category for $\sM$ an arbitrary model category.

\begin{notn}\label{deltas}
Let $S$ be a set.  We denote by $\DS$ the category consisting of~:      
\begin{itemize}
\item An object of $\DS$ is a pair $([n],c)$ where $[n]\in\Delta$ and $c:[n]\ra S$ is an arbitrary map taking values in the set $S$.  These objects will be written as strings of elements $(x_0=c(0),...,x_n=c(n))$ of $S$.
\item Let $([n],c)$ and $([m],d)$ be two objects of $\DS$.  An arrow from $([n],c)$ to $([m],d)$ is an element of the set $\DS(([n],c),([m],d))=\{u\in\Delta([n],[m]): c=d\circ u\}$.
\end{itemize}
\end{notn}

\begin{dfn}\label{mprecategory}
Let $\sM$ be a model category.  A \textit{$\sM$-precategory} is a pair $(S,C)$ where $S$ is a set of \textit{objects} and
\[    C:\Delta^\circ_{S}\ra\sM    \] 
is a functor such that $C(x)$ is a final object of $\sM$ for all $x\in S$.  
\end{dfn}

A map $(S,C)\ra (T,D)$ of $\sM$-precategories is a pair $(f,F)$ where $f:S\ra T$ is a map of sets and 
\[   F:C\ra D\circ f_{*}   \] 
is a natural transformation where $f_{*}:\DS^\circ\ra\Delta^\circ_{T}$ denotes the natural map.   Let $\PC(\sM)$ denote the category of $\sM$-precategories.  We will commonly abuse notation by referring to a $\sM$-precategory $(S,C)$ as simply $C$ and a map $(\alpha,F):(S,A)\ra(T,B)$ as simply $F:A\ra B$.  Thus $x,y\in C$ will mean $x,y\in S$ for a $\sM$-precategory $(S,C)$.  For two objects $x,y\in C$, we will also utilise the notation $\Map_{C}(x,y)$ for the object $C(x,y)$ in $\sM$ or simply $\Map(x,y)$ if the $\sM$-precategory $C$ is clear from the context.   

\begin{rmk}
One of the main reasons for imposing the condition $C(x)=*$ in Definition~\ref{mprecategory} is to obtain a cartesian structure on the model category of $\sM$-precategories.  It effectively amounts to requiring strict units.  See Section~19.3 of \cite{Si} for further discussion.
\end{rmk}

\begin{ex}\label{cattoprecat}
Every $\sM$-enriched category $D$ is a $\sM$-precategory $(S,C)$ setting $S=\Ob(D)$ and 
\[   C(x_{0},\ldots,x_{n})=D(x_{0},x_{1})\times\ldots\times D(x_{n-1},x_{n})   \]
for all $x_{i}\in D$.  This induces a fully faithful functor
\[      p:\Cat(\sM)\ra\PC(\sM)      \]
where $\Cat(\sM)$ is the category of $\sM$-enriched categories.  We will very often consider a $\sM$-enriched category $D$ as a $\sM$-precategory by identifying $D$ with $p(D)$.
\end{ex}

\begin{dfn}\label{segalcondition}
Let $\sM$ be a model category whose weak equivalences are stable under finite products.  A $\sM$-precategory $(S,C)$ is said to satisfy the \textit{Segal condition} if for all $n\geq 2$ the map
\[   C([n],c)\ra\displaystyle\prod_{1\leq i\leq n} C([1],c_i)  \]
is a weak equivalence in $\sM$ where $c_i([1]):=c(\{i-1,i\})$.
\end{dfn}

A $\sM$-precategory satisfying the Segal condition will be called a \textit{$\sM$-category}.  To avoid confusion, we will also refer to a $\sM$-enriched category as a \textit{strict} $\sM$-category.

Recall that a combinatorial model category $\sM$ is a model category which is locally presentable and cofibrantly generated.  It is moreover \textit{tractable} if the generating (trivial) cofibrations have cofibrant domains.  It is \textit{left proper} if in any pushout diagram $w:=x\coprod_z y$ in $\sM$ such that $w\ra y$ is a cofibration and $w\ra x$ is a weak equivalence, then $y\ra z$ is also a weak equivalence.

A combinatorial model category is said to be \textit{cartesian} if the cartesian product is a Quillen bifunctor and preserves colimits and the terminal object is cofibrant.  In this case the internal Hom is denoted $\uHom$.  If $\sM$ is endowed with a compatible model structure then the homotopy category $\h\sM$ is cartesian closed.  The internal Hom objects of $\h\sM$ will be denoted $\RHom(x,y)$.  

The key result for constructing the internal model category of $\inftyn$-categories is the following theorem due to Carlos Simpson (Theorem~19.3.2 of \cite{Si}).  

\begin{thm}[\cite{Si}]\label{simpsonmodel}
Let $\sM$ be a left proper tractable cartesian model category.  There exists a left proper tractable cartesian model structure on the category $\PC(\sM)$ of $\sM$-precategories in which
\begin{itemize}
\item[$(\sC)$] The cofibrations are the Reedy cofibrations.
\item[$(\sW)$] The weak equivalences are the global weak equivalences (see Definition~12.4.3 of \cite{Si}).
\end{itemize}
The fibrant objects are those $\sM$-precategories $(S,C)$ that are Reedy fibrant and satisfy the Segal condition.  When $\sM$ is a presheaf category and the cofibrations are the monomorphisms, this model structure coincides with the injective model structure on $\PC(\sM)$.
\end{thm}

The model structure supplied by Theorem~\ref{simpsonmodel} will always be assumed when referring to the model category $\PC(\sM)$ unless otherwise stated.

We are now in a position to define an $\inftyn$-category by induction.  Let $\PC^{0}(\sM):=\sM$ and for any $n\geq 1$ 
\[    \PC^{n}(\sM):=\PC(\PC^{n-1}(\sM)).     \]
Note that since for a fibrant object $C$ in $\PC^{n}(\sM)$ the $\PC^{n-1}(\sM)$-precategory $C(x_{0},\ldots,x_{n})$ is fibrant in $\PC^{n-1}(\sM)$ for any collection of objects $x_{0},\ldots x_{n}$ in $C$, the object $C$ satisfies the Segal condition iteratively on each sub-mapping space for all $1\leq i\leq n$.

The category ${\sSet}$ of simplicial sets with the Kan model structure is a left proper tractable cartesian model category so we define ${\sSet}$ to be the model category of $\inftyz$-categories.  Thus the model category of $\inftyn$-categories is given by $\PC^{n}({\sSet})$ with the model structure of Theorem~\ref{simpsonmodel}.  The model category $\PC^n(\sSet)$ will always be regarded as a $\PC^n(\sSet)$-enriched category unless otherwise stated.  An $(\infty,n)$-precategory will refer to an arbitrary object of $\PC^n(\sSet)$.  

A fibrant $\inftyn$-precategory is a $\PC^{n-1}(\sSet)$-precategory satisfying the Segal conditions and a Reedy condition.  This Reedy condition can often be ignored in applications since every fibrant $\inftyn$-precategory is equivalent to a locally fibrant $\PC^{n-1}(\sSet)$-precategory by the equivalence to an object in the projective model structure on $\PC(\sM)$.  

In particular we have a chain of Quillen equivalences (see Theorem~2.2.16 of \cite{Lu6} and Proposition~15.7.2 of \cite{Si})
\[  \Cat(\sM)\xleftarrow{q}\PC(\sM)_{\s{P}}\xra{\id}\PC(\sM)_{\s{R}}\xra{\id}\PC(\sM)_{\s{I}}   \]
where $q$ is the left adjoint to the inclusion functor $p$ and the subscripts refer to the projective, Reedy and injective model structures respectively. 

\begin{dfn}
Let $n\geq 1$.  An $\inftyn$-precategory $C$ is said to be an \textit{$\inftyn$-category} if it is a $\PC^{n-1}(\sSet)$-precategory satisfying the Segal condition such that for any object $x$ in $C$, the $\inftynmo$-precategory $C(x)$ is an $\inftynmo$-category.
\end{dfn}   

When $\sM$ is a model category we will let $L(\sM)$ be the localisation of $\sM$ along the set of weak equivalences $W$ of $\sM$.  Thus $\h(L(\sM))\ra\h(\sM)$ is an equivalence of categories.  

Let $\sM$ be a left proper tractable cartesian model category.  Then a $\sM$-enriched model category is a $\sM$-enriched category $\sN$ endowed with a model structure such that $\sN$ is tensored and cotensored over $\sM$ and the product $\otimes:\sN\times\sM\ra\sN$ is a left Quillen bifunctor.

When $\sN$ a $\sM$-enriched model category, we will write $\u{L}(\sN)$ for the localization of $\sN$ with respect to its set of weak equivalences $W$ in the sense of enriched category theory, ie. the localization is a pair $(\u{L}(\sN),l)$ where $\u{L}(\sN)$ is a $\sM$-enriched category and $l:\sN\ra\u{L}(\sN)$ is a $\sM$-enriched functor such that for any $\sM$-enriched category $C$, the induced map
\[   \Hom_{\h(\Cat(\sM))}(\u{L}(\sN),C)\ra\Hom_{\h(\Cat(\sM))}(\sN,C)  \]
is fully faithful and its essential image consists of those $\sM$-enriched functors which send each arrow in $W$ to an equivalence in $C$.  

Let $\sM$ be a left proper tractable cartesian model category and $\sN$ an $\sM$-enriched model category.  We will denote by $\sN^{f/c}$ the full subcategory of $\sN$ spanned by the fibrant-cofibrant objects.  The following is the very useful \textit{strictification theorem}.  

\begin{prop}\label{strictificationthm}
Let $\sM$ be a left proper tractable model category and $\sN$ a combinatorial $\sM$-enriched model category.  Let $C$ be a $\sM$-enriched category and $\sN^{C}$ be endowed with the projective model structure.  Then there exists an equivalence
\[  \u{L}(\sN^{C})\ra\RHom(C,\u{L}(\sN))    \]
of $\sM$-enriched categories.  
\end{prop}

\begin{proof}
Let $\h(\sN^{C})^{\tu{iso}}$ be the set of isomorphism classes of objects of $\h(\sN^{C})$.  By Lemma 6.2 of \cite{htdgc} (after making the admissible replacement of the monoidal model category of complexes of $k$-modules by an arbitrary left proper tractable cartesian model category), the map 
\[  [C,\sN^{f/c}]\ra\h(\sN^{C})^{\tu{iso}}  \] 
is an isomorphism where $[-,-]$ denotes the set of morphisms in the category $\h(\Cat(\sM))$.  Thus we have the following chain of isomorphisms
\[  [D,\RHom(C,\sN^{f/c})]\simeq[D\times C,\sN^{f/c}]\simeq(\h(\sN^{C\times D}))^{\tu{iso}}\simeq(\h((\sN^{C})^{D}))^{\tu{iso}}\simeq[D,(\sN^{C})^{f/c}] \]
for any $\sM$-enriched category $D$.  Since the construction is functorial in $D$, it results that there exists an equivalence
\[    \RHom(C,\sN^{f/c})\ra (\sN^{C})^{f/c}   \]
of $\sM$-enriched categories.  Let $\sN^{c}$ be the subcategory of $\sN$ spanned by the cofibrant objects.  The natural equivalences $Q:\sN\ra\sN^{c}$ and $R:\sN^{c}\ra\sN^{f/c}$ induce a chain of equivalences 
\[  \u{L}(\sN)\simeq\u{L}(\sN^{c})\simeq\u{L}(\sN^{f/c})\simeq\sN^{f/c}  \] 
between $\sM$-enriched categories and the result follows.
\end{proof} 

In Section~\ref{cg} we defined the $\inftynpo$-category 
\[    \uCatinfn:=\u{L}(\PC^n(\sSet))   \] 
of $\inftyn$-categories using the the localisation functor with respect to the $\PC^n(\sSet)$-enriched model structure on $\PC^n(\sSet)$.  The $\inftynpo$-category $\uCatinfn$ is equivalent to $\PC^n(\sSet)^{f/c}$.  The $\infty$-category $\uCat_{(\infty,0)}$ of $(\infty,0)$-categories was denoted $\S$.  The \textit{$\infty$-category of $\inftyn$-categories} is given by the localization 
\[  \bCat_{\inftyn}:=L(\PC^n(\sSet))    \] 
of the $\sSet$-enriched model category $\PC^n(\sSet)$.  

\begin{ex}\label{strictification}
Let $C$ be an $\inftyn$-precategory and $D$ a $\PC^{n-1}(\sSet)$-enriched category.  Assume that we are given an equivalence $q(C)\ra D$.  Then the induced map
\[   \u{L}(\PC^n(\sSet)^{D})\ra\RHom(C,\uCatinfn)  \]
is an equivalence of $\inftyn$-categories.
\end{ex}

The theory of $\inftyn$-categories behaves much like the theory of $\infty$-categories and with some care, one can develop all the tools available in that simpler setting.  However, apart from the definitions above, we will utilize only a small part of the theory.  Already from the definitions of this section though, we can deduce fundamental results like the $\inftyn$-Yoneda lemma with which we will conclude this section.

Let $X$ be an $\inftyn$-category and consider the functor
\[   \Pr_X:\PC^n(\sSet)^\circ\ra\Catinfn     \]               
sending an $\inftyn$-precategory $C$ to $\Pr_X(C)=\RHom(C^\circ,X)$.  The $\inftyn$-category $\Pr_{X}(C)$ will be called the $\inftyn$-category of \textit{$X$-valued prestacks} on $C$.  When $C$ is an $\inftyn$-precategory and $X$ is the $\inftyn$-category $\uCat_{(\infty,n-1)}$ of $(\infty,n-1)$-categories, we write $\Pr(C)$ for $\Pr_{\uCat_{(\infty,n-1)}}(C)$ and refer to $\Pr(C)$ as the $\inftyn$-category of \textit{prestacks} on $C$.  

Let $C$ be an $\inftyn$-precategory.  Then we can replace $C$ by a strict $\PC^{n-1}(\sSet)$-enriched category $D:=q(C)$.  Let $D^\circ\times D\ra\PC^{n-1}(\sSet)$ be the natural $\PC^{n-1}(\sSet)$-enriched bifunctor.  By adjunction this gives a map $D\ra\PC^{n-1}(\sSet)^{D^\circ}$ where the right hand side is equivalent to $\Pr(C)$ by the strictification theorem.  We will refer to the composition
\[    C\xras D\ra\PC^{n-1}(\sSet)^{D^\circ}\xras\Pr(C),  \]
which is well defined in $\h(\PC^n(\sSet))$, as the Yoneda embedding.  We conclude this section with the $\inftyn$-Yoneda Lemma.

\begin{prop}\label{yoneda}
Let $C$ be an $\inftyn$-precategory.  Then the Yoneda embedding 
\[   C\ra\Pr(C)   \] 
is fully faithful.
\end{prop}

\begin{proof}  
Every $\inftyn$-precategory $C$ can be associated with a $\PC^{n-1}(\sSet)$-enriched category $D:=q(C^\circ)$.  Let $E$ be a fibrant replacement of $D$ and $\PC^{n-1}(\sSet)^{E}$ be endowed with the projective model structure.  The Yoneda embedding can be written as the following composition of maps
\[ C\xra{F}p((\PC^{n-1}(\sSet)^{E})^{f/c})\xra{j}\RHom(C^\circ,p(\PC^{n-1}(\sSet)^{f/c}))\xras\RHom(C^\circ,\uCatinfnmo).  \]
Since $p(\PC^{n-1}(\sSet)^{f/c})$ is an $\inftyn$-category, the object $\RHom(C^\circ,p(\PC^{n-1}(\sSet)^{f/c}))$ can be identified with an exponential object $[p(\PC^{n-1}(\sSet)^{f/c})]^{[C^\circ]}$ in $\h(\PC^n(\sSet))$.  

Using the equivalence 
\[  \h(\PC^n(\sSet))\xras\h(\Cat(\PC^{n-1}(\sSet))), \]
the map $j$ is an equivalence from Example~\ref{strictification}.  We then apply the adjoint to $F$ and factor it as 
\[ q(C)\xras E^\circ\ra (\uCatinfnmo)^{E^\circ}. \]
It remains to show that the second map of $\Catinfnmo$-enriched categories is fully faithful.  This follows from the classical enriched Yoneda lemma.
\end{proof}

\section{Monoidal $\inftyn$-categories}\label{monoidal}

In this section we define the $\inftynpo$-category of symmetric monoidal $\inftyn$-categories.  This notion will be used from Section~\ref{cg} onwards.  To do so, we work in greater generality so that results may be utilized in other contexts by the reader.        

Let $\Gamma$ denote the category of pointed finite ordinals and point preserving maps.  This is equivalent to the category of all linearly ordered finite sets with a distinguished point $*$.  We denote the pointed ordinal $n\coprod \{*\}$ by $[n]$.  The category $\Gamma$ is a monoidal category with monoidal structure $(\Gamma,\vee,[0])$.  Consider the $n$ pointed maps $p_i:[n]\ra [1]$ in $\Gamma$ given by $p_i(j)=\{j\}$ if $i=j$ and $p_i(j)=*$ otherwise.  

\begin{dfn}\label{segalmonoidobject}
Let $\sM$ be a model category.  A \textit{commutative Segal monoid object} in $\sM$ is a functor $A:\Gamma\ra\sM$ such that for each $n\geq 0$, the map
\[   \prod_{1\leq i\leq n}A(p_{i}):A([n])\ra A([1])^n   \] 
is a weak equivalence in $\sM$.  
\end{dfn}

Let $\sM$ be a left proper tractable cartesian model category and $\sN$ a $\sM$-enriched model category.  We would like to construct a $\sM$-enriched model category whose fibrant objects are precisely the commutative Segal monoid objects in $\sN$.  This is a straightforward generalization of the corresponding result in \cite{rig} using the theory of enriched (left) Bousfield localization. 

Enriched Bousfield localization enables us, under some conditions, to deduce a new model structure on an enriched model category with an enlarged collection of weak equivalences whilst retaining the same cofibrations.  More precisely, let $\sM$ be a left proper tractable symmetric monoidal model category, $\sN$ a $\sM$-enriched model category and $S$ a collection of morphisms in $\h(\sN)$.  An object $z$ in $\h(\sN)$ is said to be \textit{$S$-local} if, for every arrow $f:x\ra y$ in $S$, the induced map
\[    \RHom(y,z)\ra\RHom(x,z)    \]
is an isomorphism in $\h(\sM)$.  An object $z$ in $\sN$ is said to be $S$-local if its image in $\h(\sN)$ is $S$-local.  A morphism $f:x\ra y$ in $\h(\sN)$ is said to be a \textit{$S$-equivalence} if, for every $S$-local object $z$ in $\h(\sM)$, the induced map
\[   \RHom(y,z)\ra\RHom(x,z)    \]
is an isomorphism in $\h(\sM)$.  An arrow $f$ in $\sN$ is said to be a $S$-equivalence if its image in $\h(\sN)$ is an $S$-equivalence.

We state the following result in the more general, not necessarily cartesian, setting.  Let $\sM$ be a symmetric monoidal model category.  An $\sM$-enriched category $\sN$ is said to be a \textit{$\sM$-enriched model category} if it is equipped with a model structure such that the category $\sN$ is tensored and cotensored over $\sM$ and the tensor product functor $\otimes:\sM\times\sN\ra\sN$ is a left Quillen bifunctor.  

\begin{thm}[\cite{Ba}]\label{barthm}
Let $\sM$ be a left proper tractable symmetric monoidal model category and $\sN$ a left proper tractable $\sM$-enriched model category.  Let $S$ be a collection of morphisms in $\h(\sN)$.  Then there exists a left-proper tractable $\sM$-enriched model category $\u{L}^b_S(\sN)$ where~:
\begin{itemize}
\item[($\sC$)] The cofibrations in $\u{L}^b_S(\sN)$ coincide with the cofibrations in $\sN$.
\item[($\sW$)] The weak equivalences in $\u{L}^b_S(\sN)$ are the $S$-equivalences.
\end{itemize}
The fibrant objects in $\u{L}^b_S(\sN)$ coincide with objects of $\sN$ which are both $S$-local and fibrant in $\sN$.
\end{thm}

This result can be used to construct an enriched model category of symmetric monoidal $\inftyn$-categories.  For two morphisms $f:x\ra y$ and $g:z\ra w$ in a symmetric monoidal model category $\sM$, we will denote by
\[      f\square g:=(y\otimes z)\coprod_{x\otimes z}(x\otimes w)\ra y\otimes w  \]
the pushout product of $f$ and $g$.

\begin{prop}\label{smom}
Let $\sM$ be a left proper tractable cartesian model category and $\sN$ a left proper tractable $\sM$-enriched model category.  There exists a left proper tractable $\sM$-enriched model structure on $\sN^{\Gamma}$ satisfying the following~:
\begin{enumerate}
\item The cofibrations are those in the projective model structure on $\sN^{\Gamma}$.
\item The weak equivalences are those morphisms $f:X\ra Y$ for which
\[    \RHom(Y,A)\ra\RHom(X,A)   \]
is an isomorphism in $\h(\sM)$ for any commutative Segal monoid object $A$ in $\sN$.
\item The fibrant objects are commutative Segal monoid objects in $\sN$ which are projectively fibrant.
\end{enumerate}
\end{prop}

\begin{proof}
It is clear from the proposition that we need to apply a Bousfield localisation to the projective model structure on $\sN^\Gamma$ whose weak equivalences are taken levelwise.  Therefore, the first condition is satisfied.  Since $\sN$ is combinatorial, we have at our disposal a collection $I$ of generating cofibrations.  Denote by $g_n:\coprod_{1\leq i\leq n}h^1\ra h^n$ the map between corepresentable functors induced by the morphisms $p_i:[n]\ra [1]$.  Let $S$ denote the collection of pushout products
\[     S:=\{f\square g_n\}_{f\in I,n\geq 0}     \]
for all generating cofibrations $f$ in $I$ and $n\geq 0$.

By Theorem~\ref{barthm}, there exists a left proper tractable $\sM$-enriched model category $\u{L}^b_S(\sN^\Gamma)$ whose cofibrations are projective cofibrations and whose weak equivalences are $S$-equivalences.  So it remains to show that the $S$-local objects are commutative Segal monoid objects.  For any generating cofibration $f:X\ra Y$ in $I$ we have a diagram
\begin{diagram}
\RHom(Y\otimes h^n,A)  &\rTo &\RHom(Y\otimes\amalg_{1\leq i\leq n}h^1,A)\coprod_{\RHom(X\otimes\amalg_{1\leq i\leq n}h^1,A)} \RHom(X\otimes h^n,A)\\
\dTo_{\sim}    &  &\dTo_{\sim}\\
\RHom(Y,A([n]))  &\rTo  &\RHom(Y,A([1])^n)\times_{\RHom(X,A([1])^n)}\RHom(X,A([n]))  
\end{diagram}
where the top arrow is an equivalence by definition for any $n\geq 0$.  Therefore the bottom arrow is an equivalence and since the arrows in $I$ generate $\sN^\Gamma$ we have an equivalence
\[    A([n])\ra A([1])^n  \]
in $\sN$.  Thus the remaining conditions are satisfied.
\end{proof}

We denote by $\SeMon(\sN)_{\s{S}}$ the $\sM$-enriched model category provided by Proposition~\ref{smom} and call it the \textit{special} model structure following \cite{Sc2}.  

Let $\SeMon(\sN)$ denote the full subcategory of $\RHom(\Gamma,\sN)$ spanned by the commutative Segal monoid objects.  Then from Proposition~\ref{strictificationthm} in  we can deduce equivalences
\[  \u{L}(\SeMon(\sN))_{\s{S}}\simeq(\SeMon(\sN)_{\s{S}})^{f/c}\simeq\SeMon(\sN^{f/c}) \]
in $\Cat(\sM)$.

\begin{dfn}
A \textit{symmetric monoidal $\inftyn$-category} is a commutative Segal monoid object in the model category $\PC^n(\sSet)$ of $\inftyn$-categories.  
\end{dfn}

The underlying $\inftyn$-category of a symmetric monoidal $\inftyn$-category $A$ is given by $A([1])$.  A symmetric monoidal functor between two symmetric monoidal $\inftyn$-categories is simply a natural transformation of functors.    

Since $\PC^n(\sSet)$ is a left proper tractable $\PC^n(\sSet)$-enriched model category, the model category 
\[     \Catinfn^\otimes:=\SeMon(\PC^n(\sSet))_{\s{S}}    \]
is the $\PC^n(\sSet)$-enriched model category of symmetric monoidal $\inftyn$-categories.  We now use the enriched localization functor to define the $\inftynpo$-category
\[   \uCatinfn^\otimes:=\u{L}(\Catinfn^\otimes)  \] 
of symmetric monoidal $\inftyn$-categories.  

Note that, from Proposition~\ref{strictificationthm}, there exists an equivalence
\[      \RHom(\Gamma,\PC^n(\sSet)^{f/c})\ra (\PC^n(\sSet)^{\Gamma})^{f/c}    \]
of $\PC^n(\sSet)$-enriched categories.  This important strictification result enables us to consider the $\inftynpo$-category of symmetric monoidal $\inftyn$-categories as ordinary functors into $\PC^n(\sSet)$ (as in the right hand side) as opposed to the much less explicit description of functors into some fibrant replacement of $\PC^n(\sSet)$.  When we are only interested in the non-invertibility of 1-morphisms, we will use the $\infty$-category
\[   \bCatinfn^\otimes:=L(\Catinfn^\otimes)  \] 
where we consider $\SeMon(\PC^n(\sSet))_{\s{S}}$ with its natural simplicial enrichment.

The model category $\PC^n(\sSet))$ with the Reedy model structure is a combinatorial symmetric monoidal model category which is freely powered (see Definition~4.5.4.2 of \cite{L2}).  Therefore, by Proposition~4.5.4.6 of \textit{loc.cit.}, there exists a combinatorial model structure on the category $\CMon(\PC^n(\sSet))$ of commutative monoid objects in $\PC^n(\sSet)$ where a morphism is a fibration (resp. weak equivalence) if it is a fibration (resp. weak equivalence) in the model category of $\inftyn$-precategories under the forgetful functor 
\[   \theta:\CMon(\PC^n(\sSet))\ra\PC^n(\sSet).   \]
It then follows from Theorem~4.5.4.7 of \textit{loc.cit.} that there exists an equivalence 
\[      {L}(\CMon(\PC^n(\sSet)))\ra\bCMon(\bCatinfn)      \]
of $\infty$-categories where the right hand side denotes the $\infty$-category of commutative monoid objects in the symmetric monoidal $\infty$-category $\bCatinfn$.

We now consider the important case of $\inftyn$-categories of modules.  Many examples of $\inftyn$-categories of modules arise from the localisation of enriched model categories of modules.  More precisely, let $\sM$ be a combinatorial symmetric monoidal model category and $R$ a commutative monoid object of $\sM$.  Assume that $\sM$ satisfies the monoid axiom \cite{SS}.  Then the category $\Mod_{R}(\sM)$ of $R$-modules admits a combinatorial model structure where a map is a fibration if and only if it is a fibration in $\sM$ and a weak equivalence if and only if it is a weak equivalence in $\sM$.  

Let $\sM$ be a symmetric monoidal model category such that the tensor product bifunctor preserves weak equivalences.  Then we can associate to $\sM$ a symmetric monoidal $\infty$-category $L^\otimes(\sM)$ (see \cite{rig} for details).  The underlying $\infty$-category of $L^\otimes(\sM)$ is simply $L(\sM)$.  

We will also need an enriched monoidal localization functor.  Let $C$ be a symmetric monoidal category and $D$ a $C$-enriched category.  A symmetric monoidal structure on $D$ is said to be \textit{weakly compatible} with the $C$-enriched structure on $D$ if the bifunctor $\otimes:D\times D\ra D$ admits the structure of a $C$-enriched functor which is compatible with the commutativity, associativity and unit constraints of $(D,\otimes)$.  

\begin{dfn}
Let $\sM$ be a symmetric monoidal model category.  A symmetric monoidal category $\sN$ is said to be a \textit{symmetric monoidal $\sM$-enriched model category} if $\sN$ is an $\sM$-enriched model category with a weakly compatible closed symmetric monoidal structure such that the natural map
\[     \RHom(Z,X^Y)\ra\RHom(Z\otimes Y, Z)  \]
induced by $ev:X^Y\otimes Y\ra X$ is an isomorphism in $\h(\sM)$ for all $X,Y,Z\in\sN$.
\end{dfn}

Let $\sN$ be a symmetric monoidal $\PC^{n}(\sSet)$-enriched category thought of as a commutative Segal monoid object in the model category $\Cat(\PC^{n}(\sSet))^\Gamma$ with weak equivalences $W$.  We consider this object as a commutative Segal monoid object in the model category $\PC^{n+1}(\sSet)$ using the fully faithful functor
\[    p:\Cat(\PC^n(\sSet))\ra\PC(\PC^{n}(\sSet))  \]
of Example~\ref{cattoprecat}.

Let $\sN^c$ denote the full subcategory of $\sN$ spanned by the cofibrant objects which are stable under the symmetric monoidal structure.  We will denote by 
\[  \coprod_W\Delta:\Gamma\ra\Cat(\sSet)\subset\PC(\PC^{n}(\sSet))  \]
the functor sending $[n]$ to $\coprod_W\Delta^n$.  We will denote by $\u{L}^\otimes\sN^c$ the pushout
\begin{diagram}
\coprod_W\Delta  &\rTo  &\sN^c\\
\dTo         &     &\dTo\\
\coprod_W * &\rTo  &\u{L}_W^\otimes \sN^c
\end{diagram}
in the model category $\Cat_{(\infty,n+1)}^\otimes$ of symmetric monoidal $(\infty,n+1)$-categories.  This $\inftynpo$-category will be denoted
\[    \u{L}^\otimes\sN:=\u{L}_W^\otimes \sN^c.     \]
The underlying $\inftynpo$-category of $\u{L}^\otimes(\sN)$ is the enriched localization $\u{L}(\sN)$.

The enriched symmetric monoidal localization just described satisfies the following universal property.  For any symmetric monoidal $\inftynpo$-category $B$, composition with $\sN\ra\u{L}^\otimes{\sN}$ induces a fully faithful functor
\[  \u{\tu{Fun}}^\otimes(\u{L}^\otimes{\sN},B)\ra\u{\tu{Fun}}^\otimes(\sN,B)   \]
whose essential image consists of those symmetric monoidal functors $F:\sN\ra B$ which send each arrow of $W$ in $\sN([1])$ to an equivalence in $B([1])$.  Here $\u{\tu{Fun}}^\otimes(-,-)$ denotes the internal hom $\RHom(-,-)$ of symmetric monoidal functors.

\begin{prop}\label{symmonmod}
Let $\sM$ be a left proper tractable symmetric monoidal model category and $\sN$ a left proper tractable symmetric monoidal $\sM$-enriched model category which is freely powered.  Let $R$ be a commutative monoid object of $\sN$.  Then $\Mod_R(\sN)$ is endowed with a symmetric monoidal $\sM$-enriched model structure.
\end{prop}

\begin{proof}
The model structure can be deduced from Theorem~4.1 of \cite{SS}.  A morphism in $\Mod_R(\sN)$ is a weak equivalence (resp. fibration) if and only if it is a weak equivalence (resp. fibration) in $\sN$.  For the compatible enrichment, we know that the categories $\sN$ and $\Mod_R(\sN)$ are cotensored over $\sM$.  Let 
\[     f:\Mod_R(\sN)\ra\sN  \]
be the forgetful functor.  Then there exist canonical isomorphisms $f(M^X)\simeq f(M)^X$ for $M\in\Mod_R(\sN)$ and $X\in\sM$.  Thus to prove the remainder of the proposition it suffices to verify the following two conditions.
\begin{enumerate}
\item The model category $\Mod_R(\sN)$ is tensored over $\sM$.
\item Given a fibration $i:M\ra N$ in $\Mod_R(\sN)$ and a cofibration $j:X\ra Y$ in $\sM$, the induced map
\[     i\wedge j:M^Y\ra M^X\times_{N^X}N^Y   \]
is a fibration in $\Mod_R(\sN)$.  Moreover, if either $i$ or $j$ is a trivial fibration, then $i\wedge j$ is also.
\end{enumerate}
The first claim follows from the fact that the map from $(-)^X:\Mod_R(\sN)\ra\Mod_R(\sN)$ admits a left adjoint from the adjoint functor theorem.  The second claim follows from the fact that the forgetful functor $f$ detects (trivial) fibrations.
\end{proof}

The example of interest which utilizes Proposition~\ref{symmonmod} is the case where $\sN$ is the monoidal model category $\PC^n(\sSet)$ of $\inftyn$-precategories.  Then given any commutative monoid object $R$ in $\PC^n(\sSet)$, ie. any symmetric monoidal $\inftyn$-precategory, then $\Mod_R(\PC^n(\sSet))$ is a symmetric monoidal $\PC^n(\sSet)$-enriched model category.  Taking the localization of this model category, we obtain the symmetric monoidal $\inftynpo$-category $\u{L}^\otimes(\Mod_R(\PC^n(\sSet)))$ of $R$-modules.

\newpage


\end{document}